\documentclass[11pt, leqno,parskip=half]{amsart}
\usepackage{amsfonts,amssymb}

\usepackage{amsmath}
\usepackage{amsthm}
\usepackage{amsrefs}
\usepackage{qsymbols}
\usepackage{latexsym}
\usepackage{chngcntr}
\usepackage{paralist}
\usepackage{mathtools}
\usepackage{esint}
\usepackage[dvipsnames]{xcolor}
\usepackage[hidelinks]{hyperref}
\usepackage{comment}
\usepackage{tikz-cd}
\usepackage{import}
\usepackage{todonotes}
\usepackage{csquotes}
\usepackage{euflag}

\newtheorem{theorem}{Theorem}[section]

\newtheorem{lemma}[theorem]{Lemma}
\newtheorem{proposition}[theorem]{Proposition}

\theoremstyle{definition}
\newtheorem{definition}[theorem]{Definition}
\newtheorem{remark}[theorem]{Remark}

\newcommand{\IR}{\mathbb{R}}
\newcommand{\IC}{\mathbb{C}}
\newcommand{\IN}{\mathbb{N}}
\newcommand{\IZ}{\mathbb{Z}}

\newcommand{\cM}{\mathcal{M}}

\newcommand{\cE}{\mathcal{E}}
\newcommand{\cF}{\mathcal{F}}
\newcommand{\cS}{\mathcal{S}}
\newcommand{\cP}{\mathcal{P}}

\newcommand{\cL}{\mathcal{L}}

\renewcommand{\L}{\mathrm{L}}

\newcommand{\B}{\mathrm{B}}

\newcommand{\F}{\mathrm{F}}

\renewcommand{\T}{\mathrm{T}}
\newcommand{\Z}{\mathrm{Z}}

\newcommand{\ind}{{\mathbf{1}}}

\newcommand{\D}{{\mathrm{D}}}

\renewcommand{\d}{\mathrm{d}}

\newcommand{\loc}{\mathrm{loc}}

\newcommand{\esssup}{\mathrm{ess\, sup}}

\DeclareMathOperator{\supp}{supp}

\DeclareMathOperator{\Id}{Id}
\DeclareMathOperator{\Rg}{\mathcal{R}}

\numberwithin{equation}{section}

\title{A new scale of function spaces characterizing homogeneous Besov spaces}

\author{Pascal Auscher}
\address{Université Paris-Saclay \\ France--Australia Mathematical Sciences and Interactions ANU -- CNRS International Research Laboratory \\ Canberra \\ ACT 2601 \\ Australia.}
\email{pascal.auscher@universite-paris-saclay.fr}

\author{Sebastian Bechtel}
\address{Université Paris-Saclay, CNRS \\ Laboratoire de Mathématiques d’Orsay \\ 91405 Orsay \\ France}
\email{sebastian.bechtel@universite-paris-saclay.fr}

\author{Luca Haardt}
\address{Karlsruhe Institute of Technology, Department of Mathematics, 76131 Karlsruhe, Germany}
\email{luca.haardt@kit.edu}



\date{\today}

\thanks{This project has received funding from the European Union’s Horizon 2020 research and innovation programme under the Marie Skłodowska-Curie grant agreement No 101034255 \euflag{}.
A CC-BY 4.0 \url{https://creativecommons.org/licenses/by/4.0/} public copyright license has been applied by the authors to the present document and will be applied to all subsequent versions up to the Author Accepted Manuscript arising from this submission.}

\dedicatory{This article is dedicated to Prof.\@ Hans Triebel on the occasion of his 90th birthday.}

\begin{document}

\begin{abstract}
    We introduce and study a new scale of function spaces that characterize the homogeneous Besov spaces $\mathrm{\dot B}^{\beta}_{p,q}$, hence completing earlier work by Ullrich. These new spaces include the ones introduced by Barton and Mayboroda, and systematically studied by Amenta under the name of weighted $\mathrm{Z}$-spaces, for the purpose of boundary value problems with $\mathrm{\dot B}^{\beta}_{p,p}$ data. They are the counterparts to the weighted tent spaces with Whitney averages, developed by Huang, and arise as their real interpolants. We describe their functional analytic properties: completeness, duality, embeddings, as well as their real and complex interpolants.
\end{abstract}

\maketitle

\allowdisplaybreaks

\section{Introduction}

\noindent This paper introduces a new scale of function spaces, naturally arising in the characterization of homogeneous Besov spaces, and establishes a full function space analysis for them. Characterizations of (homogeneous) Besov spaces $\dot \B^\beta_{p,q}$ and Triebel--Lizorkin spaces $\dot \F^\beta_{p,q}$ have attracted significant interest in recent decades due to their instrumentality in areas such as harmonic analysis, partial differential equations, and more. Traditionally, these spaces are defined via discrete Littlewood-Paley blocks $(\Delta_k)_{k\in\IZ}$ that are, in fact, convolution operators with a kernel having compact support away from zero in Fourier-space; see the monograph of Triebel \cite[Sec.\@ 5]{Triebel1}. In contrast, for applications to partial differential equations, it is of interest to have characterizations with respect to a continuous spectrum and using kernels without compactness of the Fourier-support. For example, kernels whose Fourier-transform vanishes sufficiently fast at zero and infinity, such as
\begin{align*}
    k_t(x) = \cF^{-1}\big(|t\xi|^N e^{-|t\xi|^2}\big)(x), \quad \text{for $t>0$ and some $N\in\IN$,}
\end{align*}
are of particular interest, because they lead to characterizations via the well-known Gauss--Weierstrass semigroup.\\

First continuous characterizations for general kernels of this kind were given by Triebel in \cite[Sec.\@ 2.4.2 + 2.5.1]{Triebel2} in the range $1<p<\infty$, $1<q\leq \infty$ and $\beta\in\IR$, where the case $p=\infty$ is included for Besov spaces. Later, they were complemented by Ullrich in \cite{Ullrich1} for the full range of parameters $0<p<\infty$ and $0<q\leq \infty$. Again he included the case $p=\infty$ for Besov spaces. Eventually, Hui and Taibleson \cite{Bui_Taibleson} treated also the Triebel--Lizorkin spaces in the case $p=\infty$ and $0<q\leq \infty$. For further development see also \cites{Kempka, Rauhut, Ullrich2,Ullrich3} and the references therein.
We take a closer look at Ullrich's work \cite{Ullrich1}.
Interestingly, in his paper he provides several characterizations for Triebel--Lizorkin spaces, which always have a corresponding counterpart for Besov spaces, except for one. More precisely, this additional characterization for Triebel--Lizorkin spaces uses (weighted) tent spaces $\T^{p,q}_\beta$ on $\IR^{d+1}_+= (0,\infty)\times \IR^d $. They were first introduced by Coifman, Meyer and Stein \cite{Coifman_Meyer_Stein} in the unweighted case $\beta=0$, and later studied by many other authors in the weighted case $\beta\in \IR$; see for instance \cites{Amenta2, Hofmann_Mayboroda_McIntosh}.
Roughly speaking, the tent space characterization reads
\begin{align*}
    f\in \dot \F^\beta_{p,q} \quad \Longleftrightarrow \quad   \cE(f)(t,x)\coloneqq  (\Phi_t\ast f)(x) \in \T^{p,q}_\beta,
\end{align*}
where $\Phi$ is a suitable kernel and $\Phi_t = t^{-d}\Phi(\cdot/t)$; see Section~\ref{sec: characterizations} for more details. In this construction, $(\Phi_t\ast f)(x)$ is an extension of $f$ to the upper half-space $\IR^{d+1}_+$, and $\T^{p,q}_\beta$ the extension space of $\dot \F^\beta_{p,q}$. However, a corresponding \enquote{tent space characterization} for the closely related Besov spaces $\dot \B^\beta_{p,q}$, at least when $p\neq q$, has remained a notable gap in the theory. Therefore, one central question of our paper is the following:
\vspace{0.5cm}
\begin{itemize}
    \item \textit{Is there a scale of function spaces, similar to tent spaces, that characterizes the homogeneous Besov spaces $\dot \B^\beta_{p,q}$}?
\end{itemize}
\vspace{0.5cm}
In the following, we use results from the literature to derive a suitable candidate for that scale of spaces. To do so, we want to leverage known real-interpolation properties for Triebel--Lizorkin spaces as well as for weighted tent spaces. To this end, recall that Barton and Mayboroda introduced in \cite{Barton_Mayboroda} a scale of function spaces for analyzing boundary value problems with fractional $\dot \B^\beta_{p,p}$-data; see also \cites{AA, Auscher_Egert}. Additionally, these spaces recently found applications for solving nonlinear parabolic problems; see \cite{Auscher_Bechtel}. Under the name of $\Z$-spaces, they were further investigated by Amenta in \cite{Amenta2}. In particular, he shows that the real interpolation spaces of weighted tent spaces can be identified with $\Z$-spaces. The following incomplete diagram uses these facts along with the tent space characterization of $\dot \F^\beta_{p,q}$ to identify a candidate for the characterization in the special case $\dot \B^\beta_{p,p}$.

\begin{figure}[ht]
\centering
\begin{tikzcd}[column sep=10pt, row sep=20pt]
        (\T^{p_0,r}_{\beta_0},\T^{p_1,r}_{\beta_1})_{\theta, p} \ar[r, phantom, "="] & \Z^{p,r}_\beta \\
        (\dot\F_{p_0,r}^{\beta_0},\dot\F_{p_1,r}^{\beta_1})_{\theta, p} \ar[u, "\cE"] \ar[r, phantom, "="] & \dot \B^{\beta}_{p,p}
\end{tikzcd}
\caption{Here, for $0<p_0,p_1 <\infty$, $r\in[1,\infty)$ and $\beta_0,\beta_1\in \IR$ with $\beta_0 \neq \beta_1$, we define $\frac{1}{p}= \frac{1-\theta}{p_0} + \frac{\theta}{p_1}$ and $\beta = (1-\theta) \beta_0 + \theta \beta_1$.}
\label{fig: diagram}
\end{figure}

This suggests that the $\Z$-spaces from \cites{Amenta2, Barton_Mayboroda} are the appropriate extension spaces to characterize $\dot \B^\beta_{p,p}$. We point out that the extension space $\Z^{p,r}_\beta$ introduces a new free parameter $r$ which does not appear on the level of $\dot \B^\beta_{p,p}$. This raises the question whether there is a new scale of spaces $\Z^{p,q,r}_\beta$ that characterizes $\dot \B^{\beta}_{p,q}$ also in the case $p\neq q$, and satisfies $\Z^{p,q,r}_\beta=\Z^{p,r}_\beta$ if $p=q$. Indeed, by our first main result, Theorem~\ref{thm: char of lifted Besov data}, for $r\in(0,\infty]$ the scale of spaces $\Z^{p,q,r}_\beta$, consisting of measurable functions on $\IR^{d+1}_+$ and defined by
\begin{align*}
    \Z^{p,q,r}_\beta = \bigg\{f\in\L^0(\IR^{d+1}_+): \bigg(\int\limits_{0}^\infty \bigg\|\bigg(\fint\limits_{\frac{t}{2}}^t \fint\limits_{B(\cdot,t)} |s^{-\beta}f(s,y)|^r \, \d y \d s\bigg)^\frac{1}{r} \bigg\|_{\L^p}^q\,\frac{\d t}{t}\bigg)^\frac{1}{q}<\infty\bigg\},
\end{align*}
possesses all the desired properties listed above. In the case of infinite parameters, the usual replacements of the integrals by an $\esssup$ are understood. As an illustration, we obtain a Gauss--Weierstrass characterization of $\dot \B^\beta_{p,q}$ within the space $\Z^{p,q,r}_\beta$, see Proposition~\ref{prop: Gauss--Weierstrass char}, saying that if $\beta<0$, then
\begin{align*}
    f\in \dot \B^\beta_{p,q} \quad \Longleftrightarrow \quad   \cE_\Delta(f)(t,x)\coloneqq  (e^{t^2\Delta}f)(x) \in \Z^{p,q,r}_\beta,
\end{align*}
for any $0<r\leq \infty$ with equivalence of norms $\|f\|_{\dot \B^\beta_{p,q}} \simeq \| \cE_\Delta(f) \|_{\Z^{p,q,r}_\beta}$, where $\Delta$ denotes the Laplace operator.

\subsection{Relation to and consequences for tent spaces}
As previously discussed and demonstrated in Figure~\ref{fig: diagram}, real interpolation of two weighted tent spaces $\T^{p,q}_\beta$ yields $\Z^{p,q}_\beta$-spaces. Here, we replaced the parameter $r$ by $q$. There, the interpolation was only done in the $p$-parameter for some fixed $q$. However, to the best of our knowledge, nothing is known for real interpolation of weighted tent spaces in the parameter $q$ when the parameter $p$ is fixed. More precisely, there is no known characterization of $(\T^{p,q_0}_{\beta_0},\T^{p,q_1}_{\beta_1})_{\theta, q}$. In contrast, on the level of Besov spaces it is indeed known that the corresponding real interpolation of Triebel--Lizorkin spaces $(\dot \F_{p,q_0}^{\beta_0},\dot\F_{p,q_1}^{\beta_1})_{\theta, q}$ can be identified with the Besov space $\dot \B^\beta_{p,q}$; see \cite[Thm.\@ 2.4.2]{Triebel1} for more details. Together with the fact that the $\Z^{p,q,r}_\beta$-space is the extension space of $\dot \B^\beta_{p,q}$ with some free parameter $r$, as will be established in Theorem~\ref{thm: char of lifted Besov data}, we have the following incomplete diagram:

\begin{figure}[ht]
\centering
\begin{tikzcd}[column sep=10pt, row sep=20pt]
        (\T^{p,q_0}_{\beta_0},\T^{p,q_1}_{\beta_1})_{\theta, q}  & \Z_\beta^{p,q,r} \\
        (\dot\F_{p,q_0}^{\beta_0},\dot\F_{p,q_1}^{\beta_1})_{\theta, q} \ar[u, "\cE"] \ar[r, phantom, "="] & \dot \B^{\beta}_{p,q}\ar[u, "\cE"]
\end{tikzcd}
\caption{Here, we have $0<p<\infty$, $0<q,q_0,q_1,r \leq\infty$ and $\beta_0,\beta_1\in \IR$ with $\beta_0 \neq \beta_1$ and $\beta = (1-\theta) \beta_0 + \theta \beta_1$.}
\label{fig: diagram2}
\end{figure}

Figure~\ref{fig: diagram2} suggests that the real interpolation space of two weighted tent spaces with fixed parameter $p$ but different parameters $q$ can be identified with a $\Z^{p,q,r}_\beta$-space. However, there is no reason that the free parameter $r$ appears in the real interpolation process. To overcome this anomaly, we need to introduce tent spaces with an additional parameter, which we call $\T^{p,q,r}_\beta$. We define these spaces of measurable functions on $\IR^{d+1}_+$ via
\begin{align*}
    \T^{p,q,r}_\beta = \bigg\{f\in\L^0(\IR^{d+1}_+): \bigg\| \bigg( \int\limits_0^\infty \bigg(\fint\limits_{\frac{t}{2}}^t \fint\limits_{B(\cdot,t)} |s^{-\beta}f(s,y)|^r \,\d y \d s\bigg)^\frac{q}{r}\,\frac{\d t}{t}\bigg)^\frac{1}{q}\bigg\|_{\L^p}<\infty\bigg\}.
\end{align*}
To complete Figure~\ref{fig: diagram2} and get a coherent theory, we answer the following two questions in the affirmative:
\vspace{0.5cm}
\begin{itemize}
    \item \textit{Can $\Z^{p,q,r}_\beta$ be recovered by real interpolation of $\T^{p,q,r}_\beta$-spaces}?

    \item \textit{Does the scale of function spaces $\T^{p,q,r}_\beta$ characterize homogeneous Triebel--Lizorkin spaces $\dot \F^\beta_{p,q}$}?
\end{itemize}
\vspace{0.5cm}
Lastly, we want to mention that in the case $r=q$ we obtain $\T^{p,q,r}_\beta = \T^{p,q}_\beta$ and that the spaces $\T^{p,q,r}_\beta$ are the weighted tent spaces with Whitney-averages of Huang \cite{Huang} with equivalent (quasi-)norm; see Section~\ref{subsec: basic prop}.

\subsection{Roadmap and methods}

Our paper subdivides into three parts: First, we characterize $\dot \B^\beta_{p,q}$ and $\dot \F^\beta_{p,q}$ by means of $\Z^{p,q,r}_\beta$ and $\T^{p,q,r}_\beta$, respectively. This is the content of Section~\ref{sec: characterizations}. Second, we investigate functional-analytic properties of $\Z^{p,q,r}_\beta$ in Section~\ref{sec:Z_spaces}. These properties include completeness, density and embeddings, as well as a full duality theory including the quasi-Banach range. Third, we develop a comprehensive interpolation theory for $\Z$-spaces. More precisely, we investigate their real interpolation behavior in Section~\ref{sec: real interpolation}, followed by the complex case in Section~\ref{sec: complex int}.

\subsubsection*{Characterization of Besov and Triebel--Lizorkin spaces}

This part relies heavily on the preliminary work by Ullrich~\cite{Ullrich1}. He has derived powerful characterizations for Besov and Triebel--Lizorkin spaces using maximal functions. With them at hand, it is fairly straightforward to show the inclusion of $\dot \B^\beta_{p,q}(\IR^d)$ and $\dot \F^\beta_{p,q}(\IR^d)$ into our four parameter $\Z$-space and tent space, respectively. The converse inclusion is harder as we have to dominate the pointwise maximal function by averages. The calculation is unfortunately lengthy and technical. The characterizations that we obtain in the end are in Theorem~\ref{thm: char of lifted Besov data} and Theorem~\ref{thm: char of lifted Triebel data}.

\subsubsection*{Functional-analytic properties of $\Z$-spaces}

To establish basic properties such as density, completeness and a \enquote{change of angle} formula, we mostly combine basic estimates for integrals with covering arguments adapted to the geometry of Whitney boxes. Also, arguments from~\cite{Amenta} are recycled. To investigate duality, the strategy is to embed our $\Z$-spaces into simpler spaces for which duality results are already known. In the Banach case (Theorem~\ref{thm: duality banach range}), we embed into a weighted vector-valued Lebesgue space. Next, as a general tool, we develop a dyadic characterization, which follows from straightforward covering arguments. From this description, embedding results of Hardy--Littlewood--Sobolev-type (Theorem~\ref{thm: weigthed embedding}) follow readily. Now, to obtain furthermore duality results in the quasi-Banach case (Theorem~\ref{thm: duality full range}), we can embed quasi-Banach $\Z$-spaces into $\Z$-spaces in the Banach range, for which we have already developed the duality theory.

\subsubsection*{Interpolation theory of $\Z$-spaces}

First, we develop real interpolation results. There are two of them: one (Theorem~\ref{thm: real int of Z-spaces}) with fixed parameter $p$, but a lot of flexibility for the parameters $q_i$, and another one in which $p$ can vary as well (Theorem~\ref{thm: real interpolation Z spaces II}). The proof of the first of them is based on direct estimates of the $K$-functional. The argument for the second one uses the dyadic characterization of $\Z$-spaces to embed them into weighted vector-valued sequence spaces. For them, real interpolation is known.

As a consequence, we deduce that in the case $p < \infty$, all our $\Z$-spaces can be obtained as real interpolants of tent spaces (Proposition~\ref{prop: real int tent space}). All of these results are inspired by \cite[Thm.\@ 2.4.2]{Triebel1}, where the author shows analogous results for Besov--Triebel--Lizorkin spaces.

Second, we develop a complex interpolation theory. In the \enquote{reflexive regime}, we use our embedding into a weighted vector-valued Lebesgue space to establish one inclusion of the interpolation identity, whereas the other one follows from a duality argument using the results from Section~\ref{sec:Z_spaces}.
This special case is stated in Lemma~\ref{lem: complex int 1<}.
To obtain an extension to the quasi-Banach range (Theorem~\ref{thm: complex int full range}), the complex interpolation task is first reformulated into a question of Calder\'{o}n products. To those, convexity arguments apply, and we can reduce matters to the \enquote{reflexive regime}, for which we have already established complex interpolation.

\subsection{Acknowledgments}
The authors would like to thank Emiel Lorist for pointing out a gap in one of our arguments. The third author thanks the Laboratoire de Mathématiques d’Orsay for hospitality and \textit{Studienstiftung des deutschen Volkes} for their support.

\subsection{Notation}
\label{sec: notation}

We use the following notation throughout the paper.
We denote by $\IN,\IN_0,\IZ, \IR$ the set of all positive integers, all non-negative integers, all integers, and all real numbers, respectively. Throughout, $d\geq 1$ denotes the dimension of the underlying Euclidean space. We write $\IR^{d+1}_+ = (0,\infty)\times \IR^{d}$ for the upper half-space. We denote the open ball centered at $x\in \IR^d$ with radius $t>0$ by $B(x,t)$. The Whitney box associated to $(t,x)\in\IR^{d+1}_+$ is denoted by $W(t,x) = (\frac{t}{2},t) \times B(x,t)$. We use the symbol $\|\cdot\|_X$ for the (quasi-)norm of a (quasi-)normed space $X$, and $X'$ as its anti-dual space. For a measure space $\Omega$ and a Banach space $X$ we denote by $\L^0(\Omega;X)$ the space of all strongly measurable functions with the usual identification (that is, functions equal almost everywhere are identified). Moreover, for $0<p\leq \infty$, we write
\begin{align*}
    \L^p(\Omega; X) &= \bigg\{ f\in \L^0(\Omega;X): \|f\|_{\L^p} = \bigg(\int\limits_{\Omega} \|f(x)\|_X^p\,\d \mu\bigg)^\frac{1}{p}<\infty \bigg\},
\end{align*}
with the usual modification if $p=\infty$. For $p\in (1 ,\infty)$ we denote by $p'\in (1,\infty)$ the unique number such that $1 = \frac{1}{p} + \frac{1}{p'}$. Moreover, for $p\in(0,1]$, put $p'\coloneqq \infty$. Finally, if $p=\infty$, then set $p'\coloneqq 1$. If $X$ is a quasi-Banach space, we also define for $0<p\leq \infty$ and $\beta\in \IR$ the weighted vector-valued sequence spaces $\ell^p_\beta(\IZ;X)$ as follows:
\begin{align*}
    \ell^p_\beta(\IZ;X)\coloneqq \Big\{ x: x=(x_k)_{k\in \IZ}\ \text{with}\ x_k\in X,\, \|x\|_{\ell^p_\beta(\IZ;X)} \coloneqq \Big(\sum\limits_{k\in\IZ} 2^{-k\beta p}\|x_k\|_X^p\Big)^\frac{1}{p}<\infty \Big\},
\end{align*}
again with the usual modification if $p=\infty$.
For two real numbers $a,b$, we write $a\lesssim b$ if there exists a constant $C>0$, whose precise value may vary from line to line but depends at most on the structural constants, such that $a\leq Cb$. Conversely, we write $a\gtrsim b$ if $Ca\geq b$. Furthermore, we write $a\simeq b$ if both $a\lesssim b$ and $a\gtrsim b$ hold.

\section{Characterization of Besov--Triebel--Lizorkin spaces}
\label{sec: characterizations}

\noindent In this section, we provide a continuous characterization of homogeneous Besov and Triebel--Lizorkin spaces with respect to general kernels with non-compact Fourier support. We have motivated the usage of such kernels in the introduction.

To do so, we need to introduce objects from harmonic analysis which will appear in the formulation of the characterization theorems as well as in their proofs.
Let $\cS(\IR^d)$ denote the Schwartz space and $\cS'(\IR^d)$ its topological dual space, the space of tempered distributions. Furthermore, for $N\in\IN_0 \cup \{\infty\}$ we define
\begin{align*}
    \cS_N(\IR^d) \coloneqq \{ f \in \cS(\IR^d): [\D^\alpha \cF(f)](0) = 0, \quad \text{for all $\alpha\in \IN_0^d$ with $|\alpha|\leq N$}\},
\end{align*}
where $\cF$ denotes the Fourier transform and $\D^\alpha = \partial_1^{\alpha_1} \dots \partial_d^{\alpha_d}$. In this context, we also define $\cS_{-1}(\IR^d) \coloneqq \cS(\IR^d)$. Moreover, let $\cS_N'(\IR^d)$ denote the topological dual space of $\cS_N(\IR^d)$, where $\cS_N(\IR^d)$ is equipped with the subspace topology of $\cS(\IR^d)$. It can be identified with $\cS'(\IR^d)/\cP_N(\IR^d)$, where $\cP_N(\IR^d)$ is the space of all polynomials on $\IR^d$ with complex coefficients of degree at most $N$; see \cite[Chap.\@ 5]{Triebel1}. If $\varphi$ and $\psi$ are integrable functions, then their convolution $\psi\ast \varphi$ is defined as
\begin{align*}
    (\psi\ast \varphi)(x) = \int\limits_{\IR^d} \psi(x-y)\varphi(y)\,\d y \quad (x\in\IR^d).
\end{align*}
If both $\psi$ and $\varphi$ belong to $\cS(\IR^d)$, then their convolution also belongs to $\cS(\IR^d)$. Furthermore, the convolution can be generalized to $(\psi, f ) \in \cS_N(\IR^d)\times \cS'_N(\IR^d)$ via $(\psi\ast f)(x) = f(\psi(x-\cdot))$, which is well-defined in $\IR^d$ and has at most polynomial growth.

Finally, we introduce the definition of the Peetre maximal function for dilations. For $\Psi\in \cS_N(\IR^d)$ and $t>0$ define $\Psi_t\in \cS_N(\IR^d)$ by $\Psi_t(\cdot) \coloneqq t^{-d}\Psi(\cdot/t)$. Then, for $f\in \cS'_N(\IR^d)$ and $t>0$, define $\delta_t(f)\in \cS'_N(\IR^d)$ by $\delta_t(f)(\Psi) \coloneqq f(\Psi_t)$ for $\Psi\in \cS_N(\IR^d)$  (to fix ideas, for good $f$, then $\delta_t(f)(\Psi)$ is just the integral $\int_{\IR^d} f(x) \Psi_t(x)\, \d x =\int_{\IR^d} f(tx) \Psi(x)\, \d x$). Moreover, we define for $a>0$ and $t>0$ the Peetre maximal function as
\begin{align}
    \label{eq:peetre_mf}
    (\Psi_t^*f)_a(x) = \sup\limits_{y\in\IR^d} \frac{|(\Psi_t\ast f) (x+y)|}{\Big(1+\frac{|y|}{t} \Big)^a} \quad (x\in \IR^d).
\end{align}
They will appear as a technical tool in the characterization of Besov--Triebel--Lizorkin spaces. The latter will be introduced next. For further properties of these spaces, we refer to \cite[Sec.\@ 5.1]{Triebel1}.

\begin{definition}
    Let $(\varphi_k)_{k\in\IZ}\subset \cS(\IR^d)$ be such that
    \begin{enumerate}
        \item $\varphi_k(x) = \varphi_0(2^{-k}x)$ for every $x\in\IR^d$ and $k\in\IZ$,

        \item $\supp \varphi_0 \subset \{x\in\IR^d: 2^{-1}\leq |x|\leq 2\}$,

        \item $\sum\limits_{k\in\IZ} \varphi_k(x) = 1$ for every $x\in\IR^d\setminus \{0\}$,
    \end{enumerate}
    and set $\Phi_k \coloneqq \mathcal{F}^{-1}(\varphi_k)\in \cS_\infty(\IR^d)$. For $0<p,q\leq \infty$ and $\beta\in \IR$ we define the homogeneous Besov space as
    \begin{align*}
        \dot\B^{\beta}_{p,q}(\IR^d) \coloneqq \Big\{ f\in \cS_\infty'(\IR^d): \|f\|_{\dot\B^{\beta}_{p,q}} \coloneqq \Big(\sum\limits_{k\in\IZ} 2^{k\beta q} \|\Phi_k\ast f\|_{\L^p}^q \Big)^\frac{1}{q}<\infty \Big\}.
    \end{align*}
    If in addition $0<p<\infty$, we define the homogeneous Triebel--Lizorkin spaces as
    \begin{align*}
        \dot\F^{\beta}_{p,q}(\IR^d) \coloneqq \Big\{ f\in \cS_\infty'(\IR^d): \|f\|_{\dot\F^{\beta}_{p,q}} \coloneqq \Big\|\Big(\sum\limits_{k\in\IZ} 2^{k\beta q} |(\Phi_k\ast f)(\cdot)|^q \Big)^\frac{1}{q} \Big\|_{\L^p}<\infty \Big\}.
    \end{align*}
\end{definition}
We would like to draw attention to the fact that the underlying building blocks $(\Phi_k)_{k\in\IZ}$ in the definition of these homogeneous spaces are a discrete family of convolution operators with a kernel that has compact support away from zero in Fourier-space. However, for applications in partial differential equations, one is interested in definitions with more general kernels. In particular, definitions with respect to a continuous family of convolution operators $(\Psi_t)_{t\in\IR}$ where the kernels decay fast enough in zero and infinity are of great interest. Characterizations for both $\dot\B^{\beta}_{p,q}(\IR^d)$ and $\dot\F^{\beta}_{p,q}(\IR^d)$ for such general kernels were established, for example, by Ullrich in \cite{Ullrich1}. There, for every Triebel--Lizorkin characterization but one, a corresponding Besov characterization  was found. The exceptional case is a characterization using tent spaces. While abstract arguments from interpolation theory allowed us to extend elements of the Besov space $\dot\B^{\beta}_{p,p}(\IR^d)$ to functions belonging to the $\Z$-spaces from the introduction, a proper characterization using $\Z$-spaces on the one hand, and the treatment of $\dot\B^{\beta}_{p,q}(\IR^d)$ for $p\neq q$ on the other hand, were still missing. The following theorem closes this gap in the literature.

    \begin{theorem}[Characterization of Besov spaces]
\label{thm: char of lifted Besov data}
    Let $\beta\in \IR$, $0<p,q,r\leq \infty$, and $R\in\IN_0\cup \{-1\}$ such that $R+1>\beta$. Furthermore, let $\Phi\in \cS(\IR^d)$ be such that
    \begin{align}
    \label{property 1}
        |\cF\Phi (\xi)|>0 \quad \text{on} \quad \Big\{\frac{\varepsilon}{2}<|\xi|<\varepsilon \Big\}
    \end{align}
    for some $\varepsilon>0$, and
    \begin{align}
    \label{property 2}
        [\D^\alpha \cF\Phi] (0) = 0 \quad \text{for all $\alpha\in \IN_0^d$ with $|\alpha|\leq R$.}
    \end{align}
    Then, the space $\dot \B^\beta_{p,q}(\IR^d)$ can be characterized as follows (with the usual modifications if $q$ and/or $r$ are infinite):
    \begin{align*}
        \dot \B^\beta_{p,q}(\IR^d) = \bigg\{f\in \cS'_R(\IR^d):\bigg(\int\limits_0^\infty \bigg\|\bigg(\fint\limits_{\frac{t}{2}}^t\fint\limits_{B(\cdot,t)} |s^{-\beta}\Phi_s\ast f(z)|^r\,\d z \d s\bigg)^\frac{1}{r}\bigg\|_{\L^p}^q\,\frac{\d t}{t}\bigg)^\frac{1}{q}<\infty \bigg\}.
    \end{align*}
\end{theorem}

\begin{remark}
    \label{rem:characterization}
    We clarify how the characterization identity has to be understood.
    \begin{itemize}
        \item In this theorem, the homogeneous space on the left-hand side is understood as a realization within
        $\cS_R'(\IR^d)$, where $R$ is as in the statement. This realization aligns with the framework presented in \cite[Sec.\@ 2.4.3]{Sawano} (see also \cites{Bourdaud,Moussai}) in the following sense. By \cite[Thm.\@ 2.31]{Sawano}, we can realize $\dot \B^\beta_{p,q}(\IR^d)$ within the space $\cS_M'(\IR^d)$, where $M=\max\{\lfloor \beta-\frac{d}{p}\rfloor,-1\}$ and $\lfloor x\rfloor$ denotes the integer part of $x\in\IR$. Because $M\leq R$ (see bullet point below), we have the embedding $\cS_M'(\IR^d)\subset \cS_R'(\IR^d)$ and thus can realize $\dot \B^\beta_{p,q}(\IR^d)$ within $\cS_R'(\IR^d)$. As a consequence, the identity in Theorem~\ref{thm: char of lifted Besov data} is understood as a set-theoretic equality in $\cS'_R(\IR^d)$ with corresponding equivalence of norms. The same conclusion applies to Triebel--Lizorkin spaces $\dot \F^\beta_{p,q}(\IR^d)$.

        \item We claim that $M\leq R$, which implies $\cS_M'(\IR^d)\subset \cS'_R(\IR^d)$. Indeed, if $M=-1$ then $M\leq R$ holds trivially. In the case $M=\lfloor \beta-\frac{d}{p}\rfloor$, assume for the sake of contradiction that $M>R$. Since $M,R\in\IZ$, this implies $R+1\leq M$ and thus
        \begin{align*}
            \beta < R+1\leq M  \leq \beta-\frac{d}{p},
        \end{align*}
        which is impossible.
    \end{itemize}

\end{remark}

\noindent Before we start with the proof of Theorem~\ref{thm: char of lifted Besov data}, we want to mention that the assumptions of the theorem are identical to those in \cite[Thm.\@ 2.9]{Ullrich1}. This allows us to use T.~Ullrich's techniques and derived identities as a black box for our proof. Nevertheless, the reader is advised to keep a copy of~\cite{Ullrich1} handy.

\begin{proof}[Proof of Theorem~\ref{thm: char of lifted Besov data}]

Denote by $X$ the space on the right-hand side in the statement to simplify notation. Taking Remark~\ref{rem:characterization} into account, we have to show the equivalence $\dot \B^\beta_{p,q}(\IR^d) = X$. To do so, we will divide the proof into two steps. The first step proves $\dot \B^\beta_{p,q}(\IR^d) \supset  X$ while the second step provides the reverse inclusion. For the rest of the proof, we fix $f \in \cS'_R(\IR^d)$. We suppose $q,r < \infty$ for notational convenience.

\noindent \textbf{Step 1:  $\dot \B^\beta_{p,q}(\IR^d) \supset  X$.}
Assume that $f\in X$ and fix $a>\frac{d}{\min\{1,p,q,r\}}$. Then we can use the inequality (2.89) of \cite{Ullrich1} in the homogeneous setting (see also the proof of \cite[Thm.\@ 2.8]{Ullrich1}), which reads as follows: for $\delta =R+1-\beta$, one has
\begin{align}
    \label{eq:wavelet_mf}
    2^{l\beta}(\Phi_{2^{-l}}^*f)_a(x) \lesssim \sum\limits_{k\in\IZ} 2^{-|k-l|\delta} 2^{k\beta} (\Phi^*_{2^{-k}t}f)_a(x)
\end{align}
for all $x\in \IR^d$, $t\in [1/2,2]$ and $l\in \IZ$. The slight enlargement of the constant $\delta$ as well as the interval of $t$ compared to \cite{Ullrich1} does not cause any difficulties. Fix $0<\gamma< \min\{1,p,q,r\}$ large enough, such that $ a\gamma >d$. Then, since $\gamma < 1$, the above identity implies
\begin{align*}
    2^{l\beta\gamma }|(\Phi_{2^{-l}}^*f)_a(x)|^\gamma \lesssim \sum\limits_{k\in\IZ} 2^{-|k-l|\delta \gamma} 2^{k\beta\gamma} |(\Phi^*_{2^{-k}t}f)_a(x)|^\gamma.
\end{align*}
Next, we use the inequality (2.66) of \cite{Ullrich1} on the right-hand side of the previous bound, to get
\begin{align}
\label{eq: BC1}
    2^{l\beta\gamma}|(\Phi_{2^{-l}}^*f)_a(x)|^\gamma \lesssim \sum\limits_{k\in\IZ} 2^{-|k-l|\delta\gamma} 2^{k\beta\gamma } \sum\limits_{m=0}^\infty 2^{-mN\gamma} 2^{(k+m)d}\int\limits_{\IR^d} \frac{|(\Phi_{2^{-(k+m)}t}\ast f)(y)|^\gamma}{(1+2^k|x-y|)^{a\gamma}}\,\d y,
\end{align}
for all $N\geq a$. Observe $2^{-m} t \leq 2$ as $m \geq 0$ and $t \leq 2$. If we take $|z|\leq 2^{-(k+m)}t$, then
\begin{align*}
    1+2^k|x+z-y|\leq 1+ 2^k|x-y|+2^k|z|\leq 1+ 2^k|x-y| + 2^{-m}t \lesssim1+ 2^k|x-y|.
\end{align*}
Hence, by translation, we have
\begin{align*}
    \int\limits_{\IR^d} \frac{|(\Phi_{2^{-(k+m)}t}\ast f)(y)|^\gamma}{(1+2^k|x-y|)^{a\gamma}}\,\d y \lesssim \int\limits_{\IR^d} \frac{|(\Phi_{2^{-(k+m)}t}\ast f)(y+z)|^\gamma}{(1+2^k|x-y|)^{a\gamma}}\,\d y.
\end{align*}
Since the left-hand side of this bound is independent of $z$, integrating both sides by
\begin{align*}
    \bigg(\fint\limits_{B(0 ,2^{-(k+m)}t)} |\cdot |^\frac{r}{\gamma} \,\d z\bigg)^\frac{\gamma}{r}
\end{align*}
and using Minkowski's integral inequality yield
\begin{align*}
    \int\limits_{\IR^d} \frac{|(\Phi_{2^{-(k+m)}t}\ast f)(y)|^\gamma}{(1+2^k|x-y|)^{a\gamma}}\,\d y \lesssim \int\limits_{\IR^d} \frac{\Big(\fint_{B(0 ,2^{-(k+m)}t)} |(\Phi_{2^{-(k+m)}t}\ast f)(y+z) |^r \,\d z\Big)^\frac{\gamma}{r}}{(1+2^k|x-y|)^{a\gamma}}\,\d y.
\end{align*}
Plugging this into \eqref{eq: BC1} and performing a translation give us
\begin{align*}
    &2^{l\beta\gamma}|(\Phi_{2^{-l}}^*f)_a(x)|^\gamma\\
    &\lesssim \sum\limits_{k\in\IZ} 2^{-|k-l|\delta\gamma} 2^{k\beta\gamma } \sum\limits_{m=0}^\infty 2^{-mN\gamma} 2^{(k+m)d}\int\limits_{\IR^d} \frac{\Big(\fint_{B(y ,2^{-(k+m)}t)} |(\Phi_{2^{-(k+m)}t}\ast f)(z) |^r \,\d z\Big)^\frac{\gamma}{r}}{(1+2^k|x-y|)^{a\gamma}}\,\d y.
\end{align*}
Lastly, for $\lambda \in [1 , 2]$ fixed but arbitrary, we integrate both sides by
\begin{align*}
    \bigg(\fint\limits_{\frac{\lambda}{2}}^\lambda  |\cdot |^\frac{r}{\gamma} \, \d t\bigg)^\frac{\gamma}{r},
\end{align*}
to get with Minkowski's integral inequality ($\gamma <r$) that
\begin{align}
\label{eq: B1}
    2^{l\beta\gamma}|(\Phi_{2^{-l}}^*f)_a(x)|^\gamma &\lesssim \sum\limits_{k\in\IZ} 2^{-|k-l|\delta\gamma} 2^{k\beta\gamma} \sum\limits_{m=0}^\infty 2^{-mN\gamma} 2^{(k+m)d}\int\limits_{\IR^d} \frac{|(\Phi^\lambda_{k+m}f)(y)|^\gamma}{(1+2^k|x-y|)^{a\gamma}}\,\d y \nonumber\\
    &=\sum\limits_{k\in\IZ} 2^{-|k-l|\delta\gamma} 2^{k\beta\gamma} \sum\limits_{m=0}^\infty 2^{-mN\gamma} 2^{md} [g_k\ast |\Phi^\lambda_{k+m}f|^\gamma](x),
\end{align}
where we define
\begin{align*}
    (\Phi^\lambda_{k+m}f)(y) \coloneqq \bigg(\fint\limits_{\frac{\lambda}{2}}^\lambda \fint\limits_{B(y,2^{-(k+m)}t)} |(\Phi_{2^{-(k+m)}t}\ast f)(z) |^r\,\d z \d t\bigg)^\frac{1}{r}
\end{align*}
as well as
\begin{align*}
    g_k(x) \coloneqq \frac{2^{kd}}{(1+2^k|x|)^{a\gamma}}.
\end{align*}
Since $a\gamma >d$, $g_k\in \L^1(\IR^d)$ with uniform $\L^1$-bound with respect to $k\in\IZ$. Next, take $\L^{\frac{p}{\gamma}}$-norms with respect to $x$ on both sides of \eqref{eq: B1}, and use
Minkowski's and Young's convolution inequalities, to obtain
\begin{align*}
   \|2^{l\beta}(\Phi_{2^{-l}}^*f)_a\|_{\L^p}^\gamma &\lesssim\sum\limits_{k\in\IZ} 2^{-|k-l|\delta\gamma} 2^{k\beta\gamma} \sum\limits_{m=0}^\infty 2^{-mN\gamma} 2^{md} \|g_k\ast|\Phi^\lambda_{k+m}f|^\gamma\|_{\L^\frac{p}{\gamma}}\\
   &\lesssim\sum\limits_{k\in\IZ} 2^{-|k-l|\delta\gamma} 2^{k\beta\gamma} \sum\limits_{m=0}^\infty 2^{-mN\gamma} 2^{md} \|\Phi^\lambda_{k+m}f\|^\gamma_{\L^p}.
\end{align*}
Now, integrate both sides by
\begin{align*}
    \bigg(\int\limits_1^2|\cdot |^\frac{q}{\gamma}\,\frac{\d \lambda}{\lambda}\bigg)^\frac{\gamma}{q}
\end{align*}
and use Minkowski's inequality once more, to get
\begin{align}
    \label{eq:hk_discrete_convolution}
   \|2^{l\beta}(\Phi_{2^{-l}}^*f)_a\|_{\L^p}^\gamma &\lesssim\sum\limits_{k\in\IZ} 2^{-|k-l|\delta\gamma} 2^{k\beta\gamma} \sum\limits_{m=0}^\infty 2^{-mN\gamma} 2^{md} \bigg(\int\limits_1^2\|\Phi^\lambda_{k+m}f\|^q_{\L^p}\,\frac{\d \lambda}{\lambda}\bigg)^\frac{\gamma}{q} \\
   &=\sum\limits_{k\in\IZ} 2^{-|k-l|\delta\gamma} h_k, \nonumber
\end{align}
where we define
\begin{align*}
    h_k \coloneqq 2^{k\beta\gamma} \sum\limits_{m=0}^\infty 2^{-mN\gamma} 2^{md} \bigg(\int\limits_1^2\|\Phi^\lambda_{k+m}f\|^q_{\L^p}\,\frac{\d \lambda}{\lambda}\bigg)^\frac{\gamma}{q}.
\end{align*}
Interpreting the right-hand side of~\eqref{eq:hk_discrete_convolution} as a discrete convolution, we take $\ell^\frac{q}{\gamma}$-norms on both sides and apply Young's (discrete) convolution inequality twice, as well as an index shift, to get
\begin{align*}
    \bigg(\sum\limits_{l\in\IZ} \|2^{l\beta}(\Phi_{2^{-l}}^*f)_a\|_{\L^p}^q\bigg)^\frac{\gamma}{q} &\lesssim \Big\|\sum\limits_{k\in\IZ} 2^{-|k-l|\delta\gamma} h_k \Big\|_{\ell^{\frac{q}{\gamma}}}\\
    &\lesssim \Big\|2^{k\beta\gamma}\sum\limits_{m=0}^\infty 2^{-mN\gamma} 2^{md} \bigg(\int\limits_1^2\|\Phi^\lambda_{k+m}f\|^q_{\L^p}\,\frac{\d \lambda}{\lambda}\bigg)^\frac{\gamma}{q} \Big\|_{\ell^{\frac{q}{\gamma}}}\\
    &= \bigg\| \sum\limits_{m=k}^\infty 2^{-(m-k)(N\gamma-d+\beta\gamma)}  \bigg(2^{m\beta q}\int\limits_1^2\|\Phi^\lambda_{m}f\|^q_{\L^p}\,\frac{\d \lambda}{\lambda}\bigg)^\frac{\gamma}{q} \bigg\|_{\ell^{\frac{q}{\gamma}}}\\
    &\lesssim \bigg\| \bigg(2^{k\beta q}\int\limits_1^2\|\Phi^\lambda_{k}f\|^q_{\L^p}\,\frac{\d \lambda}{\lambda}\bigg)^\frac{\gamma}{q} \bigg\|_{\ell^{\frac{q}{\gamma}}}\\
    &=\bigg\| \bigg(2^{k\beta q}\int\limits_1^2\|\Phi^\lambda_{k}f\|^q_{\L^p}\,\frac{\d \lambda}{\lambda}\bigg)^\frac{1}{q} \bigg\|^\gamma_{\ell^{q}},
\end{align*}
where we have chosen $N\in\IN$ large enough such that $N\gamma-d+\beta\gamma>0$. Taking the $\gamma$-root of the last estimate yields
\begin{align*}
    \bigg(\sum\limits_{l\in\IZ} \|2^{l\beta}(\Phi_{2^{-l}}^*f)_a\|_{\L^p}^q\bigg)^\frac{1}{q} \lesssim \bigg\| \bigg(2^{k\beta q}\int\limits_1^2\|\Phi^\lambda_{k}f\|^q_{\L^p}\,\frac{\d \lambda}{\lambda}\bigg)^\frac{1}{q} \bigg\|_{\ell^{q}}.
\end{align*}
The left-hand side in this bound is equivalent to $\|f\|_{\dot \B^\beta_{p,q}}$; see \cite[Thm.\@ 2.9]{Ullrich1}. For the right-hand side, we calculate
\begin{align*}
    \bigg(\sum\limits_{k\in\IZ} \int\limits_1^2 2^{k\beta q}\|\Phi^\lambda_{k}f\|_{\L^p}^q\,\frac{\d \lambda}{\lambda}\bigg)^\frac{1}{q}
    &=\bigg(\sum\limits_{k\in\IZ} \int\limits_1^2 2^{k\beta q}\bigg\|\bigg(\fint\limits_{\frac{\lambda}{2}}^\lambda \fint\limits_{B(\cdot,2^{-k}t)} |(\Phi_{2^{-k}t}\ast f)(z) |^r \,\d z \d t\bigg)^\frac{1}{r} \bigg\|_{\L^p}^q\,\frac{\d \lambda}{\lambda}\bigg)^\frac{1}{q}.
\intertext{First, substitute $2^{-k}t = s$, so that}
    &=\bigg(\sum\limits_{k\in\IZ} \int\limits_1^2 2^{k\beta q}\bigg\|\bigg(\fint\limits_{2^{-(k+1)}{\lambda}}^{2^{-k}\lambda} \fint\limits_{B(\cdot,s)} |(\Phi_{s}\ast f)(z) |^r \,\d z \d s \bigg)^\frac{1}{r} \bigg\|_{\L^p}^q\,\frac{\d \lambda}{\lambda}\bigg)^\frac{1}{q},
\intertext{and finally substitute $2^{-k}\lambda = t$ in the outer integral to get}
    &=\bigg(\sum\limits_{k\in\IZ} \int\limits_{2^{-k}}^{2^{-(k-1)}} 2^{k\beta q}\bigg\|\bigg(\fint\limits_{\frac{t}{2}}^{t} \fint\limits_{B(\cdot,s)} |(\Phi_{s}\ast f)(z) |^r \,\d z \d s \bigg)^\frac{1}{r} \bigg\|_{\L^p}^q\,\frac{\d t}{t}\bigg)^\frac{1}{q}\\
    &\lesssim \bigg( \int\limits_{0}^{\infty}  \bigg\|\bigg(\fint\limits_{\frac{t}{2}}^{t} \fint\limits_{B(\cdot,t)} |s^{-\beta}(\Phi_{s}\ast f)(z) |^r \,\d z \d s \bigg)^\frac{1}{r} \bigg\|_{\L^p}^q\,\frac{\d t}{t}\bigg)^\frac{1}{q},
\end{align*}
where we used $2^{-k}\simeq t$ and $t \simeq s$ in the last step.
Hence, we have the inequality
\begin{align*}
    \|f\|_{\dot\B^\beta_{p,q}}\lesssim \bigg( \int\limits_{0}^{\infty}\bigg\|\bigg(\fint\limits_{\frac{t}{2}}^{t} \fint\limits_{B(\cdot,t)} |s^{-\beta}(\Phi_{s}\ast f)(z) |^r \,\d z \d s \bigg)^\frac{1}{r} \bigg\|_{\L^p}^q\,\frac{\d t}{t}\bigg)^\frac{1}{q},
\end{align*}
which shows the first inclusion.\\

\noindent \textbf{Step 2: $\dot \B^\beta_{p,q}(\IR^d)\subset X$.}
To tackle this inclusion, recall the Peetre maximal function from~\eqref{eq:peetre_mf} and fix $t>0$ and $x\in \IR^d$. Using a substitution and introducing a supremum, we calculate
\begin{align*}
     \bigg(\fint\limits_{\frac{t}{2}}^{t} \fint\limits_{B(x,t)} |(\Phi_{s}\ast f)(z) |^r \,\d z \d s \bigg)^\frac{1}{r}&= \bigg(\fint\limits_{\frac{t}{2}}^{t} \fint\limits_{B(0,t)} |(\Phi_{s}\ast f)(x+z) |^r \,\d z \d s \bigg)^\frac{1}{r}\\
     & =   \bigg(\fint\limits_{\frac{t}{2}}^{t} \fint\limits_{B(0,t)} \Big(\frac{|(\Phi_{s}\ast f)(x+ z) |}{(1+\frac{|z|}{s})^a}\Big)^r \Big(1+\frac{|z|}{s}\Big)^{ar}\,\d z \d s \bigg)^\frac{1}{r}\\
     & \leq \bigg(\fint\limits_{\frac{t}{2}}^{t} \fint\limits_{B(0,t)} \Big(\sup\limits_{\tilde z\in\IR^d} \frac{|(\Phi_{s}\ast f)(x+\tilde z) |}{(1+\frac{|\tilde z|}{s})^a}\Big)^r \Big(1+\frac{|z|}{s}\Big)^{ar}\,\d z \d s \bigg)^\frac{1}{r}
\end{align*}
Note that $1+\frac{|z|}{s} \simeq 1$ on the Whitney box. Hence, introducing another supremum in $s$, we arrive at
\begin{align*}
     \bigg(\fint\limits_{\frac{t}{2}}^{t} \fint\limits_{B(x,t)} |(\Phi_{s}\ast f)(z) |^r \,\d z \d s \bigg)^\frac{1}{r} & \leq \bigg(\fint\limits_{\frac{t}{2}}^{t} \fint\limits_{B(0,t)} \Big(\sup\limits_{\tilde z\in\IR^d} \frac{|(\Phi_{s}\ast f)(x+\tilde z) |}{(1+\frac{|\tilde z|}{s})^a}\Big)^r \Big(1+\frac{|z|}{s}\Big)^{ar}\,\d z \d s \bigg)^\frac{1}{r}\\
     &\lesssim \sup\limits_{\frac{t}{2}<s \leq t}\sup\limits_{\tilde z\in\IR^d} \frac{|(\Phi_{s}\ast f)(x+\tilde z) |}{(1+\frac{|\tilde z|}{s})^a}=\sup\limits_{\frac{t}{2}<s \leq t} (\Phi_s^*f)_a(x).
\end{align*}
Thus, for $f\in \dot \B^\beta_{p,q}(\IR^d)$ it suffices to show
\begin{align}
\label{eq: B2}
     \bigg( \int\limits_{0}^{\infty} t^{-q\beta}\bigg\|\sup\limits_{\frac{t}{2}<s \leq t} (\Phi_s^*f)_a\bigg\|_{\L^p}^q\,\frac{\d t}{t}\bigg)^\frac{1}{q}\lesssim \|f\|_{\dot \B^\beta_{p,q}}.
\end{align}
As in Step~1, we take the inequality~\eqref{eq:wavelet_mf} as the starting point. In this bound, the discrete scale $2^{-l}$ is used on the left-hand side, and the \enquote{continuous} scale $2^{-k} t$ on the right-hand side. We claim that these roles can be swapped. Indeed, for $t > 0$, recall $\Phi_t(\cdot) = t^{-d} \Phi(\cdot/t)\in \cS_R(\IR^d)$ and $\delta_t(f)\in \cS'_R(\IR^d)$ defined by $\delta_t(f)(\varphi) = f(\varphi_t)$ for $\varphi\in \cS_R(\IR^d)$. From the elementary identity $(\Phi_t \ast f)(x) = (\Phi \ast \delta_t(f))(x/t)$ it follows
\begin{align}
    (\Phi_{st}^* f)_a(x) = (\Phi^*_s \delta_t(f))_a(x/t) \quad (s>0,x\in \IR^d).
\end{align}
Using this transformation twice in~\eqref{eq:wavelet_mf} applied with $t$ replaced by $t^{-1}$, keeping $\Phi_{2^{-l} t} = (\Phi_{2^{-l}})_t$ in mind, we find
\begin{align*}
    2^{l\beta}(\Phi_{2^{-l} t}^*f)_a(x) &= 2^{l\beta}(\Phi_{2^{-l}}^* \delta_t(f))_a(x/t) \\
    &\lesssim \sum\limits_{k\in\IZ} 2^{-|k-l|\delta} 2^{k\beta} (\Phi^*_{2^{-k} t^{-1}} \delta_t(f))_a(x/t) \\
    &= \sum\limits_{k\in\IZ} 2^{-|k-l|\delta} 2^{k\beta} (\Phi^*_{2^{-k}} f)_a(x)
\end{align*}
as desired.
Since the right-hand side of this inequality is independent of $t\in[1/2,2]$, we can fix $\lambda \in [1,2]$ and estimate
\begin{align*}
    2^{l\beta}\sup\limits_{\frac{\lambda}{2}<t\leq \lambda}|(\Phi^*_{2^{-l}t}f)_a(x) |\lesssim \sum\limits_{k\in\IZ} 2^{-|k-l|\delta} 2^{k\beta}(\Phi^*_{2^{-k}}f)_a(x).
\end{align*}
Because $\gamma < 1$, we also get
\begin{align}
\label{eq: identity continous less diskrete}
2^{l\beta\gamma}\Big(\sup\limits_{\frac{\lambda}{2}<t\leq \lambda}|(\Phi^*_{2^{-l}t}f)_a(x) |\Big)^\gamma\lesssim \sum\limits_{k\in\IZ} 2^{-|k-l|\delta\gamma} 2^{k\beta\gamma}|(\Phi^*_{2^{-k}}f)_a(x)|^\gamma.
\end{align}
Take $\L^\frac{p}{\gamma}$-norms on both sides of \eqref{eq: identity continous less diskrete} and use Minkowski's inequality to get
\begin{align*}
    2^{l\beta\gamma}\big\|\sup\limits_{\frac{\lambda}{2}<t\leq \lambda}|(\Phi^*_{2^{-l}t}f)_a| \big\|^\gamma_{\L^p} \lesssim \sum\limits_{k\in\IZ} 2^{-|k-l|\delta\gamma} 2^{k\beta\gamma}\|(\Phi^*_{2^{-k}}f)_a\|^\gamma_{\L^p}.
\end{align*}
Next, we integrate both sides by
\begin{align*}
    \bigg(\int\limits_1^2 |\cdot|^\frac{q}{\gamma}\,\frac{\d \lambda}{\lambda}\bigg)^\frac{\gamma}{q}
\end{align*}
and take $\ell^\frac{q}{\gamma}$-norms. With Young's (discrete) convolution inequality, this gives
\begin{align*}
    \bigg(\sum\limits_{l\in\IZ} \int\limits_1^2 2^{l\beta q}\big\|\sup\limits_{\frac{\lambda}{2}<t\leq \lambda}|(\Phi^*_{2^{-l}t}f)_a| \big\|_{\L^p}^q \,\frac{\d \lambda}{\lambda} \bigg)^\frac{\gamma}{q} &\lesssim \Big\|\Big(\sum\limits_{k\in\IZ} 2^{-|k-l|\delta\gamma} 2^{k\beta\gamma}\|(\Phi^*_{2^{-k}}f)_a\|^\gamma_{\L^p}\Big)_l\Big\|_{\ell^\frac{q}{\gamma}}\\
    &\lesssim \bigg(\sum\limits_{k\in\IZ} 2^{k\beta q}\|(\Phi^*_{2^{-k}}f)_a\|_{\L^p}^q\bigg)^\frac{\gamma}{q}\\
    &\lesssim \|f\|^\gamma_{\dot \B^\beta_{p,q}},
\end{align*}
where the last line follows from \cite[Theorem 2.9]{Ullrich1}. Finally, {applying several substitutions and using $t \simeq 2^{-l}$,} we have for the left-hand side of the last bound
\begin{align*}
    \bigg(\sum\limits_{l\in\IZ} \int\limits_1^2 2^{l\beta q}\big\|\sup\limits_{\frac{\lambda}{2}<t\leq \lambda}|(\Phi^*_{2^{-l}t}f)_a| \big\|_{\L^p}^q \,\frac{\d \lambda}{\lambda} \bigg)^\frac{1}{q} &= \bigg(\sum\limits_{l\in\IZ} \int\limits_1^2 2^{l\beta q}\big\|\sup\limits_{\frac{2^{-l}\lambda}{2}<s\leq 2^{-l}\lambda}|(\Phi^*_{s}f)_a| \big\|_{\L^p}^q \,\frac{\d \lambda}{\lambda} \bigg)^\frac{1}{q}\\
    &= \bigg(\sum\limits_{l\in\IZ} \int\limits_{2^{-l}}^{2^{-(l-1)}} 2^{l\beta q}\big\|\sup\limits_{\frac{t}{2}<s\leq t}|(\Phi^*_{s}f)_a| \big\|_{\L^p}^q \,\frac{\d t}{t} \bigg)^\frac{1}{q}\\
    &\simeq \bigg(\sum\limits_{l\in\IZ} \int\limits_{2^{-l}}^{2^{-(l-1)}} t^{-\beta q}\big\|\sup\limits_{\frac{t}{2}<s\leq t}|(\Phi^*_{s}f)_a| \big\|_{\L^p}^q \,\frac{\d t}{t} \bigg)^\frac{1}{q}\\
    &\simeq \bigg( \int\limits_0^\infty t^{-\beta q}\big\|\sup\limits_{\frac{t}{2}<s\leq t}|(\Phi^*_{s}f)_a| \big\|_{\L^p}^q \,\frac{\d t}{t} \bigg)^\frac{1}{q},
\end{align*}
which completes \eqref{eq: B2}, and hence the reverse inequality.
\end{proof}

We also present a similar characterization for Triebel--Lizorkin spaces. The proof is essentially as before with some minor changes. The meaning of the word \enquote{characterization} in the statement was clarified in Remark~\ref{rem:characterization}.

\begin{theorem}
\label{thm: char of lifted Triebel data}
    Let $\beta\in \IR$, $0<p<\infty$, $0<q,r\leq \infty$, and let $R\in\IN_0\cup\{-1\}$ such that $R+1>\beta$. Moreover, let $\Phi\in \cS(\IR^d)$ fulfill properties \eqref{property 1} and \eqref{property 2} for this choice of $R$ and some $\varepsilon>0$.
    Then, the space $\dot \F^\beta_{p,q}$ can be characterized as follows (with the usual modifications of $q$ and/or $r$ is infinite):
    \begin{align*}
        \dot \F^\beta_{p,q}(\IR^d) = \bigg\{ f\in \cS_R'(\IR^d): \bigg\| \bigg(\int\limits_0^\infty \bigg(\fint\limits_{\frac{t}{2}}^t\fint\limits_{B(\cdot,t)} |s^{-\beta}\Phi_s\ast f(z)|^r\,\d z \d s\bigg)^\frac{q}{r}\,\frac{\d t}{t}\bigg)^\frac{1}{q}\bigg\|_{\L^p}<\infty \bigg\}.
    \end{align*}
\end{theorem}

\begin{proof}
Let $a>\frac{d}{ \min\{1,r,q,p\}}$ and $0<\gamma< \min\{1,r,q,p\}$ such that $a\gamma >d$. By a similar procedure as in the proof of Theorem~\ref{thm: char of lifted Besov data}, but with an extra integration by
\begin{align*}
    \bigg(\int\limits_{1}^2  |\cdot |^\frac{q}{\gamma}\,\frac{\d \lambda }{\lambda}\bigg)^\frac{\gamma}{q},
\end{align*}
 on both sides of \eqref{eq: B1}, followed by Minkowski's integral inequality ($q<\gamma$), we get for all $x\in \IR^d$, $t\in [1/2,2]$ and $l\in \IZ$,
\begin{align}
\label{eq: help align :)}
    2^{l\beta\gamma}|(\Phi_{2^{-l}}^*f)_a(x)|^\gamma &\lesssim \sum\limits_{k\in\IZ} 2^{-|k-l|\delta\gamma} 2^{k\beta\gamma} \sum\limits_{m=0}^\infty 2^{-mN\gamma} 2^{(k+m)d}\int\limits_{\IR^d} \frac{|(\widetilde{\Phi}_{k+m}f)(y)|^\gamma}{(1+2^k|x-y|)^{a\gamma}}\,\d y \nonumber\\
    &=\sum\limits_{k\in\IZ} 2^{-|k-l|\delta\gamma} 2^{k\beta\gamma} \sum\limits_{m=0}^\infty 2^{-mN\gamma} 2^{md} [g_k\ast |\widetilde{\Phi}_{k+m}f|^\gamma](x),
\end{align}
where $\delta>0$ is the same as before, but we define
\begin{align*}
    (\widetilde{\Phi}_{k+m}f)(y) \coloneqq  \bigg(\int\limits_{1}^2\bigg(\fint\limits_{\frac{\lambda}{2}}^\lambda \fint\limits_{B(y,2^{-(k+m)}t)} |(\Phi_{2^{-(k+m)}t}\ast f)(z) |^r\,\d z \d t\bigg)^\frac{q}{r}\,\frac{\d \lambda }{\lambda}\bigg)^\frac{1}{q}
\end{align*}
and
\begin{align*}
    g_k(x) \coloneqq \frac{2^{kd}}{(1+2^k|x|)^{a\gamma}}.
\end{align*}
Again, the functions $g_k$ have a uniform $\L^1$-bound with respect to $k\in\IZ$, and we note that they are radial. Using the majorant property of the Hardy--Littlewood maximal function, see \cite[Thm.\@ 2.1.10]{Grafakos}, we have
\begin{align*}
    2^{l\beta\gamma}|(\Phi_{2^{-l}}^*f)_a(x)|^\gamma &\lesssim\sum\limits_{k\in\IZ} 2^{-|k-l|\delta\gamma} 2^{k\beta\gamma} \sum\limits_{m=0}^\infty 2^{-mN\gamma} 2^{md} \cM( |\widetilde{\Phi}_{k+m}f|^\gamma)(x)\\
    &=\sum\limits_{k\in\IZ} 2^{-|k-l|\delta\gamma} h_k(x),
\end{align*}
where we define
\begin{align*}
    h_k(x) \coloneqq 2^{k\beta\gamma} \sum\limits_{m=0}^\infty 2^{-mN\gamma} 2^{md} \cM( |\widetilde{\Phi}_{k+m}f|^\gamma)(x) \quad (x\in \IR^d).
\end{align*}
Next, take the $\L^{\frac{p}{\gamma}}(\IR^d;\ell^{\frac{q}{\gamma}}(\IZ))$-norm on both sides and use the convolution estimate \cite[Lemma 2.13]{Ullrich1}, to get
\begin{align*}
   \Big\| \| 2^{l\beta}|(\Phi_{2^{-l}}^*f)_a(\cdot)| \|_{\ell^{q}(\IZ)} \Big\|_{\L^{p}}^\gamma &=\Big\| \| 2^{l\beta\gamma}|(\Phi_{2^{-l}}^*f)_a(\cdot)|^\gamma \|_{\ell^{\frac{q}{\gamma}}(\IZ)} \Big\|_{\L^{\frac{p}{\gamma}}}\\
   &\lesssim \Big\| \| \sum\limits_{k\in\IZ} 2^{-|k-l|\delta\gamma} h_k(\cdot) \|_{\ell^{\frac{q}{\gamma}}(\IZ)} \Big\|_{\L^{\frac{p}{\gamma}}}\\
   &\lesssim \Big\| \| h_k(\cdot) \|_{\ell^{\frac{q}{\gamma}}(\IZ)} \Big\|_{\L^{\frac{p}{\gamma}}}.
\end{align*}
By an index shift, we find
\begin{align*}
    h_k(x) %
    &=  \sum\limits_{l=k}^\infty 2^{-(l-k)(N\gamma-d+\beta\gamma)}  2^{l\beta\gamma}\cM( |\widetilde{\Phi}_{l}f|^\gamma)(x)\\
    &\leq  \sum\limits_{l\in\IZ} 2^{-|l-k|\delta'}  2^{l\beta\gamma}\cM( |\widetilde{\Phi}_{l}f|^\gamma)(x),
\end{align*}
where $\delta' \coloneqq (N\gamma-d+\beta\gamma)>0$ for $N\in\IN$ large enough. Using again \cite[Lemma 2.13]{Ullrich1}, and the Fefferman--Stein inequality, see for example \cite[Thm.\@ 2.1]{Ullrich1}, deduce
\begin{align*}
    \Big\| \| 2^{l\beta}(\Phi_{2^{-l}}^*f)_a(\cdot) \|_{\ell^{q}(\IZ)} \Big\|_{\L^{p}}^\gamma &\lesssim \Big\| \| \sum\limits_{l\in\IZ} 2^{-|l-k|\delta'}  2^{l\beta\gamma}\cM( |\widetilde{\Phi}_{l}f|^\gamma)(\cdot) \|_{\ell^{\frac{q}{\gamma}}(\IZ)} \Big\|_{\L^{\frac{p}{\gamma}}}\\
    &\lesssim \Big\| \| 2^{l\beta\gamma}\cM( |\widetilde{\Phi}_{l}f|^\gamma)(\cdot) \|_{\ell^{\frac{q}{\gamma}}(\IZ)} \Big\|_{\L^{\frac{p}{\gamma}}}\\
    &\lesssim \Big\| \| 2^{l\beta\gamma} |(\widetilde{\Phi}_{l}f)(\cdot)|^\gamma \|_{\ell^{\frac{q}{\gamma}}(\IZ)} \Big\|_{\L^{\frac{p}{\gamma}}}\\
    &= \Big\| \| 2^{l\beta} (\widetilde{\Phi}_{l}f)(\cdot) \|_{\ell^{q}(\IZ)} \Big\|_{\L^{p}}^\gamma.
\end{align*}
The norm on the left-hand side is equivalent to $\|f\|^\gamma_{\dot \F^\beta_{p,q}}$; see \cite[Thm.\@ 2.8]{Ullrich1}. For the norm on the right-hand side, we calculate
\begin{align*}
    \Big\| \| 2^{l\beta} (\widetilde{\Phi}_{l}f)(x) \|_{\ell^{q}(\IZ)} \Big\|_{\L^{p}}&= \bigg(\int\limits_{\IR^d} \bigg(\sum\limits_{l\in\IZ} 2^{l\beta q} \int\limits_{1}^2\bigg(\fint\limits_{\frac{\lambda}{2}}^\lambda \fint\limits_{B(x,2^{-l}t)} |(\Phi_{2^{-l}t}\ast f)(z) |^r\,\d z \d t\bigg)^\frac{q}{r}\,\frac{\d \lambda }{\lambda} \bigg)^\frac{p}{q}\,\d x\bigg)^\frac{1}{p}\\
    &= \bigg(\int\limits_{\IR^d} \bigg(\sum\limits_{l\in\IZ} 2^{l\beta q} \int\limits_{1}^2\bigg(\fint\limits_{\frac{2^{-l}\lambda}{2}}^{2^{-l}\lambda} \fint\limits_{B(x,s)} |(\Phi_s\ast f)(z) |^r\,\d z \d s\bigg)^\frac{q}{r}\,\frac{\d \lambda }{\lambda} \bigg)^\frac{p}{q}\,\d x\bigg)^\frac{1}{p},
\intertext{where we substituted $2^{-l}t = s$. Finally, substitute $2^{-l}\lambda = t$ and use $B(x,s)\subset B(x,t)$ to get}
    &= \bigg(\int\limits_{\IR^d} \bigg(\sum\limits_{l\in\IZ} 2^{l\beta q} \int\limits_{2^{-l}}^{2^{-l+1}}\bigg(\fint\limits_{\frac{t}{2}}^{t} \fint\limits_{B(x,s)} |(\Phi_s\ast f)(z) |^r\,\d z \d s\bigg)^\frac{q}{r}\,\frac{\d t }{t} \bigg)^\frac{p}{q}\,\d x\bigg)^\frac{1}{p}\\
    &\lesssim \bigg(\int\limits_{\IR^d} \bigg(\sum\limits_{l\in\IZ}  \int\limits_{2^{-l}}^{2^{-l+1}}\bigg(\fint\limits_{\frac{t}{2}}^{t} \fint\limits_{B(x,t)} |s^{-\beta}(\Phi_s\ast f)(z) |^r\,\d z \d s\bigg)^\frac{q}{r}\,\frac{\d t }{t} \bigg)^\frac{p}{q}\,\d x\bigg)^\frac{1}{p}\\
    &=\bigg(\int\limits_{\IR^d} \bigg( \int\limits_{0}^\infty\bigg(\fint\limits_{\frac{t}{2}}^{t} \fint\limits_{B(x,t)} |s^{-\beta}(\Phi_s\ast f)(z) |^r\,\d z \d s\bigg)^\frac{q}{r}\,\frac{\d t }{t} \bigg)^\frac{p}{q}\,\d x\bigg)^\frac{1}{p},
\end{align*}
using also $2^{-l}\simeq t$ and $s \simeq t$ in the penultimate step.
Hence,
\begin{align*}
    \|f\|_{\dot\F^\beta_{p,q}}\lesssim\bigg(\int\limits_{\IR^d} \bigg( \int\limits_{0}^\infty\bigg(\fint\limits_{\frac{t}{2}}^{t} \fint\limits_{B(x,t)} |s^{-\beta}(\Phi_s\ast f)(z) |^r\,\d z \d s\bigg)^\frac{q}{r}\,\frac{\d t }{t} \bigg)^\frac{p}{q}\,\d x\bigg)^\frac{1}{p}.
\end{align*}
For the reverse inequality, as in the proof of Theorem~\ref{thm: char of lifted Besov data}, it suffices to show
\begin{align}
\label{reverse estimate}
     \bigg(\int\limits_{\IR^d} \bigg( \int\limits_{0}^\infty t^{-q\beta}\Big(\sup\limits_{\frac{t}{2}<s \leq t} (\Phi_s^*f)_a(x)\Big)^q\,\frac{\d t }{t} \bigg)^\frac{p}{q}\,\d x\bigg)^\frac{1}{p}\lesssim \|f\|_{\dot \F^\beta_{p,q}}.
\end{align}
We start with the inequality \eqref{eq: identity continous less diskrete}, which holds for any $f\in \cS'_R(\IR^d)$, namely
\begin{align*}
    2^{l\beta\gamma}\Big(\sup\limits_{\frac{\lambda}{2}<t\leq \lambda}|(\Phi^*_{2^{-l}t}f)_a(x) |\Big)^\gamma\lesssim \sum\limits_{k\in\IZ} 2^{-|k-l|\delta\gamma} 2^{k\beta\gamma}(\Phi^*_{2^{-k}}f)_a(x)^\gamma
\end{align*}
for $\lambda\in [1,2]$, $x\in\IR^d$ and $l\in\IZ$. Integrate both sides by
\begin{align*}
    \bigg(\int\limits_1^2 |\cdot|^\frac{q}{\gamma}\,\frac{\d \lambda}{\lambda}\bigg)^\frac{\gamma}{q},
\end{align*}
to get
\begin{align*}
    2^{l\beta \gamma}\bigg(\int\limits_1^2\Big(\sup\limits_{\frac{\lambda}{2}<t\leq \lambda}|(\Phi^*_{2^{-l}t}f)_a(x) |\Big)^q\,\frac{\d \lambda}{\lambda}\bigg)^\frac{\gamma}{q}\lesssim \sum\limits_{k\in\IZ} 2^{-|k-l|\delta\gamma} 2^{k\beta\gamma}(\Phi^*_{2^{-k}}f)_a(x)^\gamma.
\end{align*}
Take $\L^{\frac{p}{\gamma}}(\IR^d;\ell^\frac{q}{\gamma}(\IZ))$-norms on both sides and use the convolution estimate \cite[Lemma 2.13]{Ullrich1} as well as \cite[Thm.\@ 2.8]{Ullrich1} to get
\begin{align*}
    \Bigg\| \bigg\| 2^{l\beta }\bigg(\int\limits_1^2\Big(\sup\limits_{\frac{\lambda}{2}<t\leq \lambda}|(\Phi^*_{2^{-l}t}f)_a(\cdot) |\Big)^q\,\frac{\d \lambda}{\lambda}\bigg)^\frac{1}{q} \bigg\|_{\ell^q}\Bigg\|_{\L^p}^\gamma &\lesssim \Bigg\| \bigg\| \sum\limits_{k\in\IZ} 2^{-|k-l|\delta\gamma} 2^{k\beta\gamma}(\Phi^*_{2^{-k}}f)_a(\cdot)^\gamma\bigg\|_{\ell^\frac{q}{\gamma}}\Bigg\|_{\L^\frac{p}{\gamma}} \\
    &\lesssim \Big\| \big\|  2^{l\beta\gamma}(\Phi^*_{2^{-l}}f)_a(\cdot)^\gamma\big\|_{\ell^\frac{q}{\gamma}}\Big\|_{\L^\frac{p}{\gamma}} \\
    &= \Big\| \big\|  2^{l\beta}(\Phi^*_{2^{-l}}f)_a(\cdot)\big\|_{\ell^q}\Big\|_{\L^p}^\gamma \\
    &\simeq \|f\|_{\dot \F^\beta_{p,q}}^\gamma.
\end{align*}
Finally, for the $q$-th power of the inner $\ell^q$-(quasi-)norm on the left-hand side, we have
\begin{align*}
    \sum\limits_{l\in\IZ}2^{l\beta q}\int\limits_1^2\Big(\sup\limits_{\frac{\lambda}{2}<t\leq \lambda}|(\Phi^*_{2^{-l}t}f)_a(x) |\Big)^q\,\frac{\d \lambda}{\lambda}
    &=\sum\limits_{l\in\IZ}2^{l\beta q}\int\limits_1^2\Big(\sup\limits_{\frac{2^{-l}\lambda}{2}<s\leq 2^{-l}\lambda}|(\Phi^*_{s}f)_a(x) |\Big)^q\,\frac{\d \lambda}{\lambda} \\
    &=\sum\limits_{l\in\IZ}2^{l\beta q}\int\limits_{2^{-l}}^{2^{-l+1}}\Big(\sup\limits_{\frac{t}{2}<s\leq t}|(\Phi^*_{s}f)_a(x) |\Big)^q\,\frac{\d t}{t}\\
    &\simeq \sum\limits_{l\in\IZ}\int\limits_{2^{-l}}^{2^{-l+1}}t^{-\beta q}\Big(\sup\limits_{\frac{t}{2}<s\leq t}|(\Phi^*_{s}f)_a(x) |\Big)^q\,\frac{\d t}{t} \\
    &= \int\limits_0^\infty t^{-\beta q}\Big(\sup\limits_{\frac{t}{2}<s\leq t}|(\Phi^*_{s}f)_a(x) |\Big)^q\,\frac{\d t}{t},
\end{align*}
where we used the substitutions $2^{-l}t=s$, followed by $2^{-l}\lambda = t$. This proves \eqref{reverse estimate} and therefore the theorem.
\end{proof}

\begin{remark}
    Theorem~\ref{thm: char of lifted Triebel data} is generalized to the endpoint $p=\infty$ by the third-named author in \cite{Haardt}.
\end{remark}

As an illustration of Theorem~\ref{thm: char of lifted Besov data} and \ref{thm: char of lifted Triebel data}, we present the following Gauss--Weierstrass semigroup characterization of $\dot \B^\beta_{p,q}(\IR^d)$ and $\dot \F^\beta_{p,q}(\IR^d)$ (if $p<\infty$) in the case $\beta <0$. In this case, we can choose $R=-1$ such that the heat kernel $\Phi(x) = (4\pi)^{-\frac{d}{2}} e^{-\frac{|x|^2}{4}}$ satisfies \eqref{property 1} and \eqref{property 2} and the spaces $\dot \B^\beta_{p,q}(\IR^d)$ and $\dot \F^\beta_{p,q}(\IR^d)$ are realized within $\cS'(\IR^d)$; see Remark~\ref{rem:characterization}. Furthermore, we refer to Definition~\ref{def: function spaces} below for the function spaces $\T^{p,q,r}_\beta$ and $\Z^{p,q,r}_\beta$ that appear in the following statement.

\begin{proposition}[Gauss--Weierstrass semigroup characterization of $\dot \B^\beta_{p,q}$ and $\dot \F^\beta_{p,q}$]
\label{prop: Gauss--Weierstrass char}
Let $0<p,q\leq \infty$ and $\beta<0$. A tempered distribution $f\in\cS'(\IR^d)$ belongs to $\dot \B^\beta_{p,q}(\IR^d)$ if and only if $e^{t^2\Delta}f \in \Z^{p,q,r}_\beta$ for any $0<r\leq \infty$ with equivalence of norms as
\begin{align*}
    \|f\|_{\dot \B^\beta_{p,q}} \simeq \|e^{t^2\Delta}f \|_{\Z^{p,q,r}_\beta}.
\end{align*}
If in addition $0<p<\infty $, then $f\in\cS'(\IR^d)$ belongs to $\dot \F^\beta_{p,q}(\IR^d)$ if and only if $e^{t^2\Delta}f \in \T^{p,q,r}_\beta$ for any $0<r\leq \infty$ with equivalence of norms as
\begin{align*}
    \|f\|_{\dot \F^\beta_{p,q}} \simeq \|e^{t^2\Delta}f \|_{\T^{p,q,r}_\beta},
\end{align*}
where $\Delta$ denotes the Laplace operator.
\end{proposition}

\section{Functional analytic properties of $\Z$-spaces}
\label{sec:Z_spaces}

\noindent We saw in the previous section that homogeneous Besov--Triebel--Lizorkin spaces can be characterized via extensions to the upper half-space $\IR^{d+1}_+$ that belong to a certain extension space. In this section, we want to study these extension spaces from a function space theoretical point of view. First, we provide concrete definitions of these spaces and compare them to known spaces from the literature. Then, we study properties of these spaces, such as completeness and density (Section~\ref{subsec: basic prop}), duality (Section~\ref{subsec: duality >1} and \ref{subsec: duality p<1}), as well as dyadic descriptions of these spaces in conjunction with Hardy--Littlewood--Sobolev-type embeddings (Section~\ref{subsec: dyad and emb}).

\subsection{Definitions and basic properties}
\label{subsec: basic prop}

Here, we define the function spaces that appear in the characterizations of Section~\ref{sec: characterizations}, and collect basic properties such as completeness and density. Furthermore, we provide a link to existing function spaces from the literature to contextualize our new function spaces.

\begin{definition}[$\Z$-spaces and tent spaces]
\label{def: function spaces}
    For $0<p,q,r\leq \infty$ and $\beta\in\IR$ define the space $\Z^{p,q,r}_\beta$ as the set of all measurable functions $f$ on $\IR^{d+1}_+$ such that
    \begin{align*}
        \|f\|_{\Z^{p,q,r}_\beta} \coloneqq \bigg(\int\limits_{0}^\infty \bigg\|\bigg(\fint\limits_{\frac{t}{2}}^t \fint\limits_{B(\cdot,t)} |s^{-\beta}f(s,y)|^r \, \d y \d s\bigg)^\frac{1}{r} \bigg\|_{\L^p}^q\,\frac{\d t}{t}\bigg)^\frac{1}{q}<\infty,
    \end{align*}
    with the usual replacement of the integral by an $\esssup$ if the corresponding parameters $p,q,r$ are infinite.
    If in addition $0<p< \infty$, then define the space $\T^{p,q,r}_\beta$ as the set of all measurable functions $f$ on $\IR^{d+1}_+$ such that
    \begin{align*}
        \|f\|_{\T^{p,q,r}_\beta} \coloneqq \bigg\| \bigg( \int\limits_0^\infty \bigg(\fint\limits_{\frac{t}{2}}^t \fint\limits_{B(\cdot,t)} |s^{-\beta}f(s,y)|^r \,\d y \d s\bigg)^\frac{q}{r}\,\frac{\d t}{t}\bigg)^\frac{1}{q} \bigg\|_{\L^p} < \infty,
    \end{align*}
    with the same replacements in the case of infinite parameters as before. In the case $\beta = 0$ we usually drop the dependence from the notation.
\end{definition}

Let $0<a<b<\infty$ and $c>0$. Then, we define the Whitney box $W_{a,b,c}(t,x)$ associated to a point $(t,x)\in\IR^{d+1}_+$ as
\begin{align*}
    W_{a,b,c}(t,x)\coloneqq (at,bt)\times B(x,ct).
\end{align*}
If $a=\frac{1}{2}$, $b=c=1$ then we get $W_{a,b,c}(t,x) = W(t,x)$.
Even though our function spaces $\Z^{p,q,r}_\beta$ and $\T^{p,q,r}_\beta$ are defined by averages with respect to the specific Whitney box $W(t,x)$, it turns out that any other choice $(a,b,c)$ as above yields an equivalent norm. We will provide a proof for this fact in Proposition~\ref{prop: independent Whitney cube} and use it freely afterwards.

\begin{remark}
\label{rem: comparison of known spaces}
Let us compare with the definitions in \cites{Coifman_Meyer_Stein,Barton_Mayboroda}, see also \cite{Amenta2}, where authors introduced new scales of function spaces, so-called tent and Z-spaces. For $0 < p,q,r \leq \infty$ and $\beta \in \IR$, the tent space $\T^{p,q}_\beta$ (if in addition $p < \infty$) and the $\Z$-space $\Z^{p,r}_\beta$ defined in these references consist of all measurable functions $f$ on $\IR^{d+1}_+$ such that the corresponding (quasi-)norms
    \begin{align*}
        &\|f\|_{\Z^{p,r}_\beta} \coloneqq \bigg(\int\limits_{0}^\infty \int\limits_{\IR^d}\bigg(\fint\limits_{\frac{t}{2}}^t \fint\limits_{B(x,t)} |s^{-\beta}f(s,y)|^r \, \d y \d s\bigg)^\frac{p}{r}\,\frac{\d x\d t}{t}\bigg)^\frac{1}{p},\\
        &\|f\|_{\T^{p,q}_\beta} \coloneqq \bigg( \int\limits_{\IR^d}\bigg(\int\limits_{0}^\infty \fint\limits_{B(x,t)} |t^{-\beta}f(t,y)|^q \, \frac{\d y \d t}{t}\bigg)^\frac{p}{q}\,\d x\bigg)^\frac{1}{p}
    \end{align*}
    are finite, with the usual modifications in the case of infinite parameters. Our spaces are generalizations of these spaces with coincidence in the cases $\T^{p,q,q}_\beta = \T^{p,q}_\beta$ and $\Z^{p,p,r}_\beta =\Z^{p,r}_\beta$.
\end{remark}

Moreover, Huang has introduced in \cite{Huang} the Whitney-averaged tent space $\T^{p,r}_{q,\beta}$ as the set of all measurable functions $f$ on $\IR^{d+1}_+$ such that
\begin{align*}
    \|f\|_{\T^{p,r}_{q,\beta}} \coloneqq \bigg(\int\limits_{\IR^d} \bigg( \int\limits_0^\infty\fint\limits_{B(x,t)} \bigg(\fint\limits_{\frac{t}{2}}^{2t} \fint\limits_{B(y,t)} |s^{-\beta}f(s,z)|^r \,\d z \d s\bigg)^\frac{q}{r}\,\frac{\d y\d t}{t}\bigg)^\frac{p}{q}\,\d x \bigg)^\frac{1}{p} < \infty.
\end{align*}
Observe that they have an additional average compared to the norm of $\T^{p,q,r}_\beta$. Nevertheless, both spaces coincide with equivalence of norms due to the following lemma.

\begin{lemma}
    Let $0<p<\infty$, $0<q,r\leq \infty$ and $\beta\in\IR$. Then, $\T^{p,r}_{q,\beta}=\T^{p,q,r}_\beta$ with equivalent norms.
\end{lemma}

\begin{proof}
    To keep the notation simple, we suppose $q,r$ finite. The other cases are similar.

    \textbf{Step 1: $\T^{p,q,r}_\beta \subset \T^{p,r}_{q,\beta}$}.
    Using $B(y,t) \subset B(x, 2t)$, we calculate
    \begin{align*}
        \|f\|_{\T^{p,r}_{q,\beta}}^p &=\int\limits_{\IR^d}\bigg(\int\limits_0^\infty\fint\limits_{B(x,t)} \bigg(\fint\limits_{\frac{t}{2}}^{2t} \fint\limits_{B(y,t)} |s^{-\beta}f(s,z)|^r \,\d z \d s\bigg)^\frac{q}{r}\,\frac{\d y\d t}{t}\bigg)^\frac{p}{q}\,\d x\\
        &\lesssim \int\limits_{\IR^d}\bigg(\int\limits_0^\infty\fint\limits_{B(x,t)} \bigg(\fint\limits_{\frac{t}{2}}^{2t} \fint\limits_{B(x,2t)} |s^{-\beta}f(s,z)|^r \,\d z \d s\bigg)^\frac{q}{r}\,\frac{\d y\d t}{t}\bigg)^\frac{p}{q}\,\d x\\
        &=\int\limits_{\IR^d}\bigg(\int\limits_0^\infty \bigg(\fint\limits_{\frac{t}{2}}^{2t} \fint\limits_{B(x,2t)} |s^{-\beta}f(s,z)|^r \,\d z \d s\bigg)^\frac{q}{r}\,\frac{\d t}{t}\bigg)^\frac{p}{q}\,\d x.
    \end{align*}
    A straight forward change of Whitney parameters in the norm on the right-hand side implies
    \begin{align*}
        \|f\|_{\T^{p,r}_{q,\beta}}^p &\lesssim\int\limits_{\IR^d}\bigg(\int\limits_0^\infty \bigg(\fint\limits_{\frac{t}{2}}^{2t} \fint\limits_{B(x,2t)} |s^{-\beta}f(s,z)|^r \,\d z \d s\bigg)^\frac{q}{r}\,\frac{\d t}{t}\bigg)^\frac{p}{q}\,\d x\\
        &\lesssim \int\limits_{\IR^d}\bigg(\int\limits_0^\infty \bigg(\fint\limits_{\frac{t}{2}}^{t} \fint\limits_{B(x,t)} |s^{-\beta}f(s,z)|^r \,\d z \d s\bigg)^\frac{q}{r}\,\frac{\d t}{t}\bigg)^\frac{p}{q}\,\d x\\
        &= \|f\|_{\T^{p,q,r}_\beta}^p.
    \end{align*}
     \textbf{Step 2: $ \T^{p,r}_{q,\beta}\subset \T^{p,q,r}_\beta $.}
    Introducing an extra average, we get the identity
    \begin{align*}
        \|f\|_{\T^{p,q,r}_\beta}^p &=\int\limits_{\IR^d}\bigg(\int\limits_0^\infty \bigg(\fint\limits_{\frac{t}{2}}^{t} \fint\limits_{B(x,t)} |s^{-\beta}f(s,z)|^r \,\d z \d s\bigg)^\frac{q}{r}\,\frac{\d t}{t}\bigg)^\frac{p}{q}\,\d x\\
        &=\int\limits_{\IR^d}\bigg(\int\limits_0^\infty\fint\limits_{B(x,t)} \bigg(\fint\limits_{\frac{t}{2}}^{t} \fint\limits_{B(x,t)} |s^{-\beta}f(s,z)|^r \,\d z \d s\bigg)^\frac{q}{r}\,\frac{\d y\d t}{t}\bigg)^\frac{p}{q}\,\d x.
    \end{align*}
    For fixed $x\in\IR^d$ we use the covering $B(x,t)\subset B(y,2t)$ for $y\in B(x,t)$ to get
    \begin{align*}
        \|f\|_{\T^{p,q,r}_\beta}^p \lesssim \int\limits_{\IR^d}\bigg(\int\limits_0^\infty\fint\limits_{B(x,t)} \bigg(\fint\limits_{\frac{t}{2}}^{t} \fint\limits_{B(y,2t)} |s^{-\beta}f(s,z)|^r \,\d z \d s\bigg)^\frac{q}{r}\,\frac{\d y\d t}{t}\bigg)^\frac{p}{q}\,\d x.
    \end{align*}
    Since the norm
    of $\T^{p,r}_{q,\beta}$ is invariant under a change of Whitney parameters, see \cite[Obs.\@ 2.4]{Huang}, we get
    \begin{align*}
         \|f\|_{\T^{p,q,r}_\beta}^p \lesssim\int\limits_{\IR^d}\bigg(\int\limits_0^\infty\fint\limits_{B(x,t)} \bigg(\fint\limits_{\frac{t}{2}}^{2t} \fint\limits_{B(y,t)} |s^{-\beta}f(s,z)|^r \,\d z \d s\bigg)^\frac{q}{r}\,\frac{\d y\d t}{t}\bigg)^\frac{p}{q}\,\d x =\|f\|_{\T^{p,r}_{q,\beta}}^p,
    \end{align*}
    which proves the claim.
\end{proof}

\begin{remark}
    We point out that Huang also defined and studied such tent spaces in the case $p = \infty$.
\end{remark}

For Huang's tent spaces $\T^{p,r}_{q,\beta}$, many function space properties are already established. By virtue of the preceding lemma, many of them transfer to our tent spaces $\T^{p,q,r}_\beta$ as well. This is why we will mainly focus on the study of the $\Z^{p,q,r}_\beta$ spaces in this article. The $\T^{p,q,r}_\beta$ spaces are further investigated by the third-named author in the follow-up article~\cite{Haardt}.
A notable exception is Proposition~\ref{prop: real int tent space}, in which we link the real interpolation theories of tent and $\Z$-spaces.

\begin{lemma}
    For all $0<r\leq \infty$ and $f\in \L^0(\IR^{d+1}_+)$, the function
    \begin{align*}
        F(t,x) \coloneqq  \bigg(\fint\limits_{\frac{t}{2}}^{t} \fint\limits_{B(x,t)}|f(s,y)|^r\,\d y \d s\bigg)^\frac{1}{r}
    \end{align*}
    is lower semicontinuous. Consequently, if $p=\infty$ and/or $q=\infty$, one can replace the respective $\esssup$ by $\sup$ in the definition of $\Z^{p,q,r}_\beta$ for all $0<p,q,r\leq \infty$ and $\beta\in\IR$.
\end{lemma}

\begin{proof}
    Fix $M>0$ and assume that $F(t,x)>M$ for some $(t,x)\in\IR^{d+1}_+$. \\

    \noindent\textbf{Case 1: $0<r<\infty$}. By continuity of the Lebesgue measure, there exists $\varepsilon>0$ such that
    \begin{align*}
        \bigg(\frac{2}{(t+\varepsilon)|B(0,t+\varepsilon)|}\int\limits_{\frac{t}{2}+\frac{\varepsilon}{2}}^{t-\varepsilon} \int\limits_{B(x,t-\varepsilon)}|f(s,y)|^r\,\d y \d s\bigg)^\frac{1}{r} >M.
    \end{align*}
    Take $(\tilde t, \tilde x)\in\IR^{d+1}_+$ with $|x-\tilde x|<\varepsilon$ and $t-\varepsilon<\tilde t <t +\varepsilon$. Then, one has
    \begin{align*}
        \Big(\frac{t}{2}+\frac{\varepsilon}{2} , t-\varepsilon \Big)\times B(x, t-\varepsilon) \subset \Big(\frac{\tilde t}{2}, \tilde t \Big)\times B(x, \tilde t),
    \end{align*}
    which results in
    \begin{align*}
        \bigg(\frac{2}{(t+\varepsilon)|B(0,t+\varepsilon)|}\int\limits_{\frac{\tilde t}{2}}^{\tilde t} \int\limits_{B(\tilde x,\tilde t)}|f(s,y)|^r\,\d y \d s\bigg)^\frac{1}{r} >M.
    \end{align*}
    Finally, since $\tilde t <t + \varepsilon$, we get
    \begin{align*}
        F(\tilde t, \tilde x) = \bigg(\fint\limits_{\frac{\tilde t}{2}}^{\tilde t} \fint\limits_{B(\tilde x,\tilde t)}|f(s,y)|^r\,\d y \d s\bigg)^\frac{1}{r}> \bigg(\frac{2}{(t+\varepsilon)|B(0,t+\varepsilon)|}\int\limits_{\frac{\tilde t}{2}}^{\tilde t} \int\limits_{B(\tilde x,\tilde t)}|f(s,y)|^r\,\d y \d s\bigg)^\frac{1}{r} >M.
    \end{align*}
    This implies that $\{(t,x)\in\IR^{d+1}_+: F(t,x)>M\}$ is open. \\

    \noindent\textbf{Case 2: $r= \infty$}. We claim again that there exists $\varepsilon>0$ such that
    \begin{align*}
        \esssup_{\frac{t}{2}+\frac{\varepsilon}{2} <s<t - \varepsilon} \esssup_{y\in B(x,t-\varepsilon)} |f(s,y)| >M.
    \end{align*}
    Assume the contrary. Then, by continuity of the Lebesgue measure, it follows that
    \begin{align*}
        &|\{(s,y)\in W(t,x) : |f(s,y)|>M \}|\\
        &\quad= \lim\limits_{\varepsilon\to 0} \Big|\{(s,y)\in \Big(\frac{t}{2}+\frac{\varepsilon}{2}, t - \varepsilon\Big)\times B(x,t-\varepsilon) : |f(s,y)|>M \}\Big| =0.
    \end{align*}
    But this would imply $F(t,x) \leq M$,
    which is a contradiction. Now, the same procedure as in Case~1 implies the claim also for the case $r=\infty$.
\end{proof}

Another important property is the following quantitative \enquote{change of angle} formula. The name is motivated by the corresponding change of angle formula for tent spaces; see \cite{Auscher}.

\begin{lemma}
\label{lem: change of angle}
    For $\lambda >1$ we have
    \begin{align*}
         \bigg(\int\limits_{0}^\infty \bigg\|\bigg(\fint\limits_{\frac{t}{2}}^t t^{-d}\int\limits_{B(\cdot,\lambda t)} |s^{-\beta}f(s,y)|^r \, \d y \d s\bigg)^\frac{1}{r} \bigg\|_{\L^p}^q\,\frac{\d t}{t}\bigg)^\frac{1}{q} \lesssim \lambda^\frac{d}{\min\{p,r\}} \|f\|_{\Z^{p,q,r}_\beta},
    \end{align*}
    where the implicit constants depend only on $p,r$ and $d$.
\end{lemma}

\begin{proof}
    It is enough to look at $\beta = 0$. For notational convenience, we just consider $p,q,r$ finite. 
    By translation and dilation, there exist $J\subset \IN$ with $|J| \leq C\lambda^d$ for some constant $C>0$ depending only on $d$, and $(\xi_j)_{j\in J}\subset \IR^d$ such that for all $(t,x)\in\IR^{d+1}_+$, one has the covering by the balls
    \begin{align*}
        B(x,\lambda t) \subset \bigcup_{j\in J} B(x-t\xi_j, t).
    \end{align*}
    Then, we get
    \begin{align*}
        \fint\limits_{\frac{t}{2}}^t t^{-d}\int\limits_{B(x,\lambda t)} |f(s,y)|^r \, \d y \d s\lesssim \sum\limits_{j\in J} \fint\limits_{\frac{t}{2}}^t \fint\limits_{B(x-t\xi_j, t)} |f(s,y)|^r \, \d y \d s.
    \end{align*}
    If $p<r$ we have
    \begin{align*}
        &\bigg(\int\limits_{\IR^d}\bigg( \sum\limits_{j\in J} \fint\limits_{\frac{t}{2}}^t \fint\limits_{B(x-t\xi_j, t)} |f(s,y)|^r \, \d y \d s\bigg)^\frac{p}{r}\,\d x \bigg)^\frac{1}{p}\\
        &\leq \bigg(\sum\limits_{j\in J} \int\limits_{\IR^d}\bigg(  \fint\limits_{\frac{t}{2}}^t \fint\limits_{B(x-t\xi_j, t)} |f(s,y)|^r \, \d y \d s\bigg)^\frac{p}{r}\,\d x \bigg)^\frac{1}{p}\\
        &\leq C^{\frac{1}{p}}\lambda^\frac{d}{p}\bigg( \int\limits_{\IR^d}\bigg(  \fint\limits_{\frac{t}{2}}^t \fint\limits_{B(x, t)} |f(s,y)|^r \, \d y \d s\bigg)^\frac{p}{r}\,\d x \bigg)^\frac{1}{p}.
    \end{align*}
    On the contrary, if $p\geq r$ we use Minkowski's inequality to give
    \begin{align*}
         &\bigg(\int\limits_{\IR^d}\bigg( \sum\limits_{j\in J} \fint\limits_{\frac{t}{2}}^t \fint\limits_{B(x-t\xi_j, t)} |f(s,y)|^r \, \d y \d s\bigg)^\frac{p}{r}\,\d x \bigg)^\frac{1}{p}\\
         &= \Big\|  \sum\limits_{j\in J} \fint\limits_{\frac{t}{2}}^t \fint\limits_{B(\cdot-t\xi_j, t)} |f(s,y)|^r \, \d y \d s \Big\|_{\L^{\frac{p}{r}}}^\frac{1}{r}\\
         &\leq  \bigg(\sum\limits_{j\in J}   \Big\|\fint\limits_{\frac{t}{2}}^t \fint\limits_{B(\cdot-t\xi_j, t)} |f(s,y)|^r \, \d y \d s\Big\|_{\L^{\frac{p}{r}}} \bigg)^\frac{1}{r}\\
         &=  \bigg(\sum\limits_{j\in J}   \Big\|\fint\limits_{\frac{t}{2}}^t \fint\limits_{B(\cdot, t)} |f(s,y)|^r \, \d y \d s\Big\|_{\L^{\frac{p}{r}}} \bigg)^\frac{1}{r}\\
         &\leq C^{\frac{1}{r}}\lambda^\frac{d}{r}\bigg( \int\limits_{\IR^d}\bigg(  \fint\limits_{\frac{t}{2}}^t \fint\limits_{B(x, t)} |f(s,y)|^r \, \d y \d s\bigg)^\frac{p}{r}\,\d x \bigg)^\frac{1}{p}..
    \end{align*}
    Thus, both cases together yield the inequality of the claim upon applying the $\L^q$-norm in $t$.
\end{proof}

As announced earlier, we can now use Lemma~\ref{lem: change of angle} to show that the definition of the $\Z^{p,q,r}_\beta$-spaces is independent of the underlying Whitney box.

\begin{proposition}
\label{prop: independent Whitney cube}
    Let $0<p,q,r\leq \infty$ and $\beta\in\IR$. For $f\in \Z^{p,q,r}_\beta$, $0<a<b<\infty$ and $c>0$, we have
    \begin{align*}
        \bigg(\int\limits_0^\infty\bigg( \int\limits_{\IR^d}\bigg(  \fint\limits_{at}^{bt} \fint\limits_{B(x, ct)} |f(s,y)|^r \, \d y \d s\bigg)^\frac{p}{r}\,\d x \bigg)^\frac{q}{p}\,\frac{\d t}{t}\bigg)^\frac{1}{q}\simeq \|f\|_{\Z^{p,q,r}_\beta},
    \end{align*}
    where the implicit constants depend on $a,b,c$, $p$, $q$, $r$, and the dimension $d$.
\end{proposition}

\begin{proof}
    For simplicity we only consider $p,q,r$ finite. Moreover, we only provide the direction \enquote{$\lesssim$}. The reverse direction follows analogously. Let $n,m\in\IZ$ be such that $2^m<a<b<2^n$. Fix $(t,x)\in\IR^{d+1}_+$. Then, estimate for the average integral on the left-hand side
    \begin{align*}
         \fint\limits_{at}^{bt} \fint\limits_{B(x, ct)} |f(s,y)|^r \, \d y \d s\lesssim  \sum\limits_{k=m+1}^{n} \, \fint\limits_{2^{k-1}t}^{2^{k}t} \fint\limits_{B(x, ct)} |f(s,y)|^r \, \d y \d s.
    \end{align*}
    Using finiteness of the sum, we can estimate
    \begin{align*}
        &\bigg(\int\limits_0^\infty\bigg( \int\limits_{\IR^d}\bigg(  \fint\limits_{at}^{bt} \fint\limits_{B(x, ct)} |f(s,y)|^r \, \d y \d s\bigg)^\frac{p}{r}\,\d x \bigg)^\frac{q}{p}\,\frac{\d t}{t}\bigg)^\frac{1}{q}\\
        &\lesssim  \bigg(\int\limits_0^\infty\bigg( \int\limits_{\IR^d}\bigg(\sum\limits_{k=m+1}^{n} \, \fint\limits_{2^{k-1}t}^{2^{k}t} \fint\limits_{B(x, ct)} |f(s,y)|^r \, \d y \d s\bigg)^\frac{p}{r}\,\d x \bigg)^\frac{q}{p}\,\frac{\d t}{t}\bigg)^\frac{1}{q}\\
        &\lesssim  \sum\limits_{k=m+1}^{n} \bigg(\int\limits_0^\infty\bigg( \int\limits_{\IR^d}\bigg( \fint\limits_{2^{k-1}t}^{2^{k}t} \fint\limits_{B(x, ct)} |f(s,y)|^r \, \d y \d s\bigg)^\frac{p}{r}\,\d x \bigg)^\frac{q}{p}\,\frac{\d t}{t}\bigg)^\frac{1}{q}.
    \intertext{By the substitution $2^kt = \tau$, we obtain}
        &=  \sum\limits_{k=m+1}^{n} \bigg(\int\limits_0^\infty\bigg( \int\limits_{\IR^d}\bigg( \fint\limits_{\frac{\tau}{2}}^{\tau} \fint\limits_{B(x, \frac{c}{2^k}\tau)} |f(s,y)|^r \, \d y \d s\bigg)^\frac{p}{r}\,\d x \bigg)^\frac{q}{p}\,\frac{\d \tau}{\tau}\bigg)^\frac{1}{q}\\
        &\simeq  \sum\limits_{k=m+1}^{n}  \Big(\frac{c}{2^k}\Big)^{-\frac{d}{r}}\bigg(\int\limits_0^\infty\bigg( \int\limits_{\IR^d}\bigg( \fint\limits_{\frac{\tau}{2}}^{\tau} \tau^{-d}\int\limits_{B(x, \frac{c}{2^k}\tau)} |f(s,y)|^r \, \d y \d s\bigg)^\frac{p}{r}\,\d x \bigg)^\frac{q}{p}\,\frac{\d \tau}{\tau}\bigg)^\frac{1}{q}
    \intertext{Finally, using Lemma~\ref{lem: change of angle} in the case $\frac{c}{2^k}> 1$ or $B(x, \frac{c}{2^k}\tau)\subset B(x, \tau)$ when $\frac{c}{2^k}\leq 1$, we end up with}
        &\lesssim  \sum\limits_{k=m+1}^{n} \Big(\frac{c}{2^k}\Big)^{-\frac{d}{r}} \min\Big\{1 , \Big(\frac{c}{2^k}\Big)^{\frac{d}{\min\{p,r\}}}\Big\}\|f\|_{\Z^{p,q,r}_\beta}\\
        &\lesssim \|f\|_{\Z^{p,q,r}_\beta},
    \end{align*}
    where the implicit constant depends only on $p,q,r,d$ as well as $a,b,c$.
\end{proof}

\begin{lemma}
\label{lem: compare Z with Lr}
    Let $0<p,q,r\leq \infty$, $\beta\in \IR$ and $K\subset \IR^{d+1}_+$ be compact. Then, for a measurable function $f$ on $\IR^{d+1}_+$ we have
    \begin{align*}
        \|\mathbf{1}_K f \|_{\Z^{p,q,r}_\beta} \simeq \|f\|_{\L^{r}\big(K,\frac{\d y \d s}{s}\big)},
    \end{align*}
    where the implicit constants depend on $p,q,r,\beta,d$ and $K$.
\end{lemma}

\begin{proof}
    Using dependence of the implicit constants on $\beta$ and $K$, we can without loss of generality reduce to the case $\beta = 0$. Moreover, to simplify the notation, we only consider $p,q,r$ finite.
    Since $K$ is compact, there exist $0<a<b<\infty$, $R>0$ and $x_0\in \IR^{d}$ such that $K\subset (a,b)\times B(x_0,R)$.\\

    \noindent\textbf{Step 1: the direction \enquote{$\lesssim$}}. One has
    \begin{align*}
        &\bigg(\int\limits_{0}^\infty \bigg(\int\limits_{\IR^d}\bigg(\fint\limits_{\frac{t}{2}}^t \fint\limits_{B(x,t)} |\mathbf{1}_K(s,y)f(s,y)|^r \, \d y \d s\bigg)^\frac{p}{r}\,\d x \bigg)^\frac{q}{p}\,\frac{\d t}{t}\bigg)^\frac{1}{q}\\
        &\lesssim_{a,b,r} \bigg(\int\limits_{0}^\infty \bigg(\int\limits_{\IR^d}\bigg(\int\limits_{\frac{t}{2}}^t \int\limits_{B(x,t)} |\mathbf{1}_K(s,y)f(s,y)|^r \,\frac{\d y \d s}{s}\bigg)^\frac{p}{r}\,\d x \bigg)^\frac{q}{p}\,\frac{\d t}{t}\bigg)^\frac{1}{q}.
    \end{align*}
    Due to the observation
    \begin{align*}
        \mathbf{1}_{(t/2,t)\times B(x,t)}(s,y) \mathbf{1}_{(a,b)\times B(x_0,R)}(s,y)\leq \mathbf{1}_{(a,2b)}(t) \mathbf{1}_{B(x_0,2b+R)}(x),
    \end{align*}
    we end up with
    \begin{align*}
        &\bigg(\int\limits_{0}^\infty \bigg(\int\limits_{\IR^d}\bigg(\int\limits_{\frac{t}{2}}^t \int\limits_{B(x,t)} |\mathbf{1}_K(s,y)f(s,y)|^r \, \frac{\d y \d s}{s}\bigg)^\frac{p}{r}\,\d x \bigg)^\frac{q}{p}\,\frac{\d t}{t}\bigg)^\frac{1}{q} \\
        &\leq \bigg(\int\limits_{a}^{2b}\bigg(\int\limits_{B(x_0,2b+R)} \| f\|_{\L^r\big(K,\frac{\d y \d s}{s}\big)}^p\,\d x \bigg)^\frac{q}{p}\,\frac{\d t}{t}\bigg)^\frac{1}{q}\\
        &\lesssim_{a,b,R,p,q}  \|f\|_{\L^r\big(K,\frac{\d y \d s}{s}\big)}.
    \end{align*}

    \noindent \textbf{Step 2: the reverse inequality \enquote{$\gtrsim$}}. We estimate
    \begin{align*}
        \|f\|_{\L^r\big(K,\frac{\d y \d s}{s}\big)} &\lesssim_{a,b,r,d} \bigg(\int\limits_{a}^{b} \int\limits_{B(x_0,R)} |(\mathbf{1}_Kf)(s,y)|^r \,\frac{\d y \d s}{s^{d+1}}\bigg)^\frac{1}{r}\\
        &\leq \sum\limits_{j\in J}\bigg(\int\limits_{a}^{b} \int\limits_{B(x_j,\frac{a}{2})} |(\mathbf{1}_Kf)(s,y)|^r \,\frac{\d y \d s}{s^{d+1}}\bigg)^\frac{1}{r},
    \end{align*}
    where $(x_j)_{j\in J}\subset \IR^d$ for some finite set $J\subset \IN$ and $|J|$ depends only on $d,a,R$. An averaging argument yields
    \begin{align*}
        \sum\limits_{j\in J}\bigg(\int\limits_{a}^{b} \int\limits_{B(x_j,\frac{a}{2})} |(\mathbf{1}_Kf)(s,y)|^r \,\frac{\d y \d s}{s^{d+1}}\bigg)^\frac{1}{r}
        &= \sum\limits_{j\in J}\bigg(\fint\limits_{B(x_j,\frac{a}{2})}\bigg(\int\limits_{a}^{b} \int\limits_{B(x_j,\frac{a}{2})} |(\mathbf{1}_Kf)(s,y)|^r \,\frac{\d y \d s}{s^{d+1}}\bigg)^\frac{p}{r}\,\d x\bigg)^\frac{1}{p}.
    \end{align*}
    Now, for $x,y\in B(x_j,\frac{a}{2})$, one has $|x-y|<a$. Therefore,
    \begin{align*}
        &\sum\limits_{j\in J}\bigg(\fint\limits_{B(x_j,\frac{a}{2})}\bigg(\int\limits_{a}^{b} \int\limits_{B(x_j,\frac{a}{2})} |(\mathbf{1}_Kf)(s,y)|^r \,\frac{\d y \d s}{s^{d+1}}\bigg)^\frac{p}{r}\,\d x\bigg)^\frac{1}{p}\\
        &\leq \sum\limits_{j\in J}\bigg(\fint\limits_{B(x_j,\frac{a}{2})}\bigg(\int\limits_{a}^{b} \int\limits_{B(x,a)} |(\mathbf{1}_Kf)(s,y)|^r \,\frac{\d y \d s}{s^{d+1}}\bigg)^\frac{p}{r}\,\d x\bigg)^\frac{1}{p}\\
        &\lesssim_{a,R,d,p} \bigg(\int\limits_{\IR^d}\bigg(\int\limits_{a}^{b} \int\limits_{B(x,a)} |(\mathbf{1}_Kf)(s,y)|^r \,\frac{\d y \d s}{s^{d+1}}\bigg)^\frac{p}{r}\,\d x\bigg)^\frac{1}{p}.
    \end{align*}
    To introduce in addition a time integral, we do a similar trick. Let $N\in\IN$ be such that $b<2^Na$. Then,
    \begin{align*}
        &\bigg(\int\limits_{\IR^d}\bigg(\int\limits_{a}^{b} \int\limits_{B(x,a)} |(\mathbf{1}_Kf)(s,y)|^r \,\frac{\d y \d s}{s^{d+1}}\bigg)^\frac{p}{r}\,\d x\bigg)^\frac{1}{p}\\
        &\lesssim\sum\limits_{l=0}^{N-1} \bigg(\int\limits_{\IR^d}\bigg(\int\limits_{2^la}^{2^{l+1}a} \int\limits_{B(x,a)} |(\mathbf{1}_Kf)(s,y)|^r \,\frac{\d y \d s}{s^{d+1}}\bigg)^\frac{p}{r}\,\d x\bigg)^\frac{1}{p}\\
        &\lesssim \sum\limits_{l=0}^{N-1} \bigg(\fint\limits_{2^{l+1}a}^{2^{l+2}a}\bigg(\int\limits_{\IR^d}\bigg(\int\limits_{2^la}^{2^{l+1}a} \int\limits_{B(x,a)} |(\mathbf{1}_Kf)(s,y)|^r \,\frac{\d y \d s}{s^{d+1}}\bigg)^\frac{p}{r}\,\d x\bigg)^\frac{q}{p}\,\d t\bigg)^\frac{1}{q}\\
        &\lesssim\sum\limits_{l=0}^{N-1} \bigg(\int\limits_{2^{l+1}a}^{2^{l+2}a}\bigg(\int\limits_{\IR^d}\bigg(\int\limits_{\frac{t}{4}}^{t} \int\limits_{B(x,t)} |(\mathbf{1}_Kf)(s,y)|^r \,\frac{\d y \d s}{s^{d+1}}\bigg)^\frac{p}{r}\,\d x\bigg)^\frac{q}{p}\,\frac{\d t}{t}\bigg)^\frac{1}{q}\\
        &\lesssim  \bigg(\int\limits_0^\infty\bigg(\int\limits_{\IR^d}\bigg(\fint\limits_{\frac{t}{4}}^{t} \fint\limits_{B(x,t)} |(\mathbf{1}_Kf)(s,y)|^r \,\d y \d s\bigg)^\frac{p}{r}\,\d x\bigg)^\frac{q}{p}\,\frac{\d t}{t}\bigg)^\frac{1}{q}.
    \end{align*}
    Since the $\Z$-space (quasi-)norms  with respect to different Whitney boxes are equivalent, see Proposition~\ref{prop: independent Whitney cube}, the claim follows.
\end{proof}

\begin{proposition}
\label{prop: complete and dense subspace}
    The space $\Z^{p,q,r}_\beta$ is a (quasi-)Banach space. Moreover, the space of compactly supported functions in $\L^r(\IR^{d+1}_+)$ is a dense subspace of $\Z^{p,q,r}_\beta$ if $\max\{p,q,r\}<\infty$.
\end{proposition}

\begin{proof}
    Using Lemma~\ref{lem: compare Z with Lr}, the proof is the same as for~\cite[Proposition 3.5]{Amenta} and will be omitted.
\end{proof}

\subsection{Duality: Banach range}
\label{subsec: duality >1}

A very useful tool to analyze our spaces is an embedding into vector-valued Lebesgue spaces; see \cites{Amenta, HTV} for the same approach for tent spaces. More precisely, denote by $\L^q_{t;\beta}\L^p_{\vphantom{t,\beta}x}\L^r_{\vphantom{t,\beta}s,y}$ the vector-valued Lebesgue space
\begin{align*}
    \L^q\bigg((0,\infty), \frac{\d t}{t^{1+\beta q}} ; \L^p\Big(\IR^d; \L^r\Big(\IR^{d+1}_+; \frac{\d y \d s}{s^{d+1}}\Big) \Big) \bigg).
\end{align*}
Then, we have the following lemma.

\begin{lemma}
\label{lem: vv-embedding}
    The linear mapping
    \begin{align*}
        i : {\Z^{p,q,r}_\beta } \to \L^q_{t;\beta}\L^p_{\vphantom{t,\beta}x}\L^r_{\vphantom{t,\beta}s,y},\quad f(s,y)\mapsto \mathbf{1}_{W(t,x)}(s,y)f(s,y)
    \end{align*}
    satisfies $\|i(f)\|_{\L^q_{t;\beta}\L^p_{\vphantom{t,\beta}x}\L^r_{\vphantom{t,\beta}s,y}}\simeq \|f\|_{{\Z^{p,q,r}_\beta }}$. In particular, it is bounded, injective and has closed range.
\end{lemma}

\begin{proof}
    It is clear that
    \begin{align*}
        \|i(f)\|_{\L^q_{t;\beta}\L^p_{\vphantom{t,\beta}x}\L^r_{\vphantom{t,\beta}s,y}}
        \simeq \bigg(\int\limits_{0}^\infty \bigg(\int\limits_{\IR^d}\bigg(\fint\limits_{\frac{t}{2}}^t \fint\limits_{B(x,t)} |s^{-\beta}f(s,y)|^r \, \d y \d s\bigg)^\frac{p}{r}\,\d x \bigg)^\frac{q}{p}\,\frac{\d t}{t}\bigg)^\frac{1}{q}= \|f\|_{{\Z^{p,q,r}_\beta }}.
    \end{align*}
    Therefore, the map is bounded, injective and has closed range due to Proposition~\ref{prop: complete and dense subspace}.
\end{proof}

\begin{theorem}[Duality -- Banach range]
\label{thm: duality banach range}
    Let $1\leq p,q,r<\infty$ and $\beta\in \IR$. Then, we have
    \begin{align*}
        \int\limits_0^\infty \int\limits_{\IR^d} |f(s,y)| \cdot |g(s,y)| \,\frac{\d y \d s}{s} \lesssim  \|f\|_{\Z^{\vphantom{p',q',r'}p,q,r}_{\beta} } \cdot \|g\|_{\Z^{p',q',r'}_{-\beta} }
    \end{align*}
    for any measurable functions $f,g$. Moreover, we can identify $(\Z^{p,q,r}_\beta)' \simeq \Z^{p',q',r'}_{-\beta}$ with equivalent norms via the $\L^2$ duality pairing.
\end{theorem}

\begin{proof}
    Let $f$ and $g$ be measurable functions. Then,
    \begin{align*}
        \int\limits_0^\infty \int\limits_{\IR^d} |f(s,y)| \cdot |g(s,y)| \,\frac{\d y \d s}{s} &= \int\limits_0^\infty \int\limits_{\IR^d} \fint\limits_{s}^{2s}\fint\limits_{B(y,s)} |f(s,y)| \cdot |g(s,y)| \, \d x \d t\,\frac{\d y \d s}{s}\\
        &\lesssim \int\limits_0^\infty \int\limits_{\IR^d} \fint\limits_{\frac{t}{2}}^{t}\fint\limits_{B(x,t)} |s^{-\beta}f(s,y)| \cdot |s^{\beta}g(s,y)| \, \d y \d s\,\frac{\d x \d t}{t}\\
        &\leq \|f\|_{\Z^{\vphantom{p',q',r'}p,q,r}_\beta } \cdot \|g\|_{\Z^{p',q',r'}_{-\beta} },
    \end{align*}
    where we used Hölder's inequality three times in the last step. In particular, $\Z^{p',q',r'}_{-\beta} \subset ({\Z^{p,q,r}_\beta })'$. For the reverse inclusion, pick $L \in ({\Z^{p,q,r}_\beta })'$. Recall the mapping $i$ from Lemma~\ref{lem: vv-embedding} and let $\Rg(i)$ denote its range. Then, we have
    \begin{align*}
        L\circ i^{-1} \in \cL (\Rg(i), \IC),
    \end{align*}
     which can be extended to some functional on $\L^q_{t;\beta}\L^p_{\vphantom{t,\beta}x}\L^r_{\vphantom{t,\beta}s,y}$. Let $\tilde L$ denote this extension and identify it with $g\in \L^{q'}_{t;\beta}\L^{p'}_{\vphantom{t,\beta}x}\L^{r'}_{\vphantom{t,\beta}s,y}$ by duality of weighted vector-valued Lebesgue spaces. Then, for all $f \in \Z^{p,q,r}_{\beta}$ one has
    \begin{align*}
        L(f) = L\circ i^{-1}(i(f)) &= \int\limits_0^\infty\int\limits_{\IR^d}\int\limits_0^\infty\int\limits_{\IR^d} g(\tau,z,s,y) \cdot \overline{i(f)(\tau,z,s,y)}\,\frac{\d y \d s}{s^{d+1}} \,\frac{\d z \d \tau}{\tau}\\
        &=\int\limits_0^\infty \int\limits_{\IR^d}\int\limits_0^\infty\int\limits_{\IR^d} g(\tau,z,s,y) \cdot \overline{\mathbf{1}_{W(\tau,z)}(s,y)f(s,y)}\,\frac{\d y \d s}{s^{d+1}} \,\frac{\d z \d \tau}{\tau}\\
        &=\int\limits_0^\infty \int\limits_{\IR^d}\bigg(\int\limits_s^{2s}\int\limits_{B(y,\tau)}  g(\tau,z,s,y) \,\frac{\d z \d \tau}{\tau s^{d}}\bigg)\cdot \overline{f(s,y)}\,\frac{\d y \d s}{s}.
    \end{align*}
    Observe that
    \begin{align*}
        \bigg|\int\limits_s^{2s}\int\limits_{B(y,\tau)} g(\tau,z,s,y) \,\frac{\d z \d \tau}{\tau s^{d}}\bigg|\lesssim \fint\limits_s^{2s}\fint\limits_{B(y,\tau)} |g(\tau,z,s,y)| \,\d z \d \tau\eqqcolon G(s,y).
    \end{align*}
    Hence, it suffices to show that $G \in \Z^{p',q',r'}_{-\beta}$. Indeed, if $G \in \Z^{p',q',r'}_{-\beta}$ then we can identify $L$ with a $\Z^{p',q',r'}_{-\beta}$-function and thus get the reverse inclusion $({\Z^{p,q,r}_\beta })'\subset \Z^{p',q',r'}_{-\beta}$. To prove $G \in \Z^{p',q',r'}_{-\beta}$, we fix $(t,x)\in \IR^{d+1}_+$. For $\frac{t}{2}< s <t$ and $s<\tau <2s$ we have $(s,2s) \times B(y,\tau) \subset (\frac{t}{2}, 2t) \times B(x,3t)$ for all $y\in B(x,t)$. This, together with Minkowski's integral inequality, yields
    \begin{align*}
        \bigg(\fint\limits_{\frac{t}{2}}^{t} \fint\limits_{B(x,t)} |s^{\beta}G(s,y)|^{r'} \,\d y\d s \bigg)^\frac{1}{{r'}} &\lesssim \bigg(\fint\limits_{\frac{t}{2}}^{t} \fint\limits_{B(x,t)} \bigg|s^{\beta}\fint\limits_{\frac{t}{2}}^{2t} \fint\limits_{B(x,3t)} |g(\tau, z, s, y)| \, \d z \d \tau\bigg|^{r'} \,\d y\d s \bigg)^\frac{1}{r'}\\
        &\lesssim  \fint\limits_{\frac{t}{2}}^{2t} \fint\limits_{B(x,3t)}\bigg(\fint\limits_{\frac{t}{2}}^{t} \fint\limits_{B(x,t)} |\tau^\beta g(\tau, z, s, y)|^{r'} \,\d y\d s  \,\bigg)^\frac{1}{r'} \d z \d \tau\\
        &\lesssim \fint\limits_{\frac{t}{2}}^{2t} \fint\limits_{B(x,3t)}\bigg(\fint\limits_{\frac{\tau}{4}}^{2\tau} \fint\limits_{B(z,8\tau)} |\tau^\beta g(\tau, z, s, y)|^{r'} \,\d y\d s  \,\bigg)^\frac{1}{r'} \d z \d \tau\\
        &= \fint\limits_{\frac{t}{2}}^{2t} \fint\limits_{B(x,3t)} h(\tau,z)\, \d z \d \tau,
    \end{align*}
    where
    \begin{align}
        \label{eq:duality_def_h}
        h(\tau,z) \coloneqq \bigg(\fint\limits_{\frac{\tau}{4}}^{2\tau} \fint\limits_{B(z,8\tau)} |\tau^\beta g(\tau, z, s, y)|^{r'} \,\d y\d s  \,\bigg)^\frac{1}{r'}.
    \end{align}
    Taking the $p'$-th power and integrating both sides in $x$ in the last estimate yields
    \begin{align*}
        \int\limits_{\IR^d}\bigg(\fint\limits_{\frac{t}{2}}^{t} \fint\limits_{B(x,t)} |s^\beta G(s,y)|^{r'} \,\d y\d s \bigg)^\frac{p'}{r'}\,\d x &\lesssim \int\limits_{\IR^d} \bigg|\fint\limits_{\frac{t}{2}}^{2t} \fint\limits_{B(x,3t)} h(\tau,z) \,\d z \d \tau\bigg|^{p'}\, \d x.
    \end{align*}
    Using Fubini's theorem, Jensen's inequality and an averaging trick gives us
    \begin{align*}
        \int\limits_{\IR^d}\bigg(\fint\limits_{\frac{t}{2}}^{t} \fint\limits_{B(x,t)} |s^\beta G(s,y)|^{r'} \,\d y\d s \bigg)^\frac{p'}{r'}\,\d x \lesssim \int\limits_{\IR^d}  \fint\limits_{B(x,3t)} \bigg|\fint\limits_{\frac{t}{2}}^{2t}h(\tau,z)\, \d \tau\bigg|^{p'} \d z  \d x
        &=\int\limits_{\IR^d} \bigg|\fint\limits_{\frac{t}{2}}^{2t}h(\tau,z)\, \d \tau\bigg|^{p'} \d z .
    \end{align*}
    Taking the $p'$-th root and taking $\L^{q'}$-norms in $t$ on both sides yields
    \begin{align*}
        \|G\|_{\Z^{p',q',r'}_{-\beta}} &\lesssim \bigg(\int\limits_0^\infty \bigg\| \fint\limits_{\frac{t}{2}}^{2t}h(\tau,\cdot)\, \d \tau \bigg\|_{\L^{p'}}^{q'}\,\frac{\d t}{t}\bigg)^\frac{1}{q'}.
    \end{align*}
    Again, Minkowski's integral inequality, Jensen's inequality and an averaging trick gives us
    \begin{align*}
        \|G\|_{\Z^{p',q',r'}_{-\beta}}
        \lesssim \bigg(\int\limits_0^\infty \bigg( \fint\limits_{\frac{t}{2}}^{2t} \|h(\tau,\cdot)\|_{\L^{p'}} \d \tau \bigg)^{q'}\,\frac{\d t}{t}\bigg)^\frac{1}{q'}
        &\leq \bigg(\int\limits_0^\infty \fint\limits_{\frac{t}{2}}^{2t} \|h(\tau,\cdot)\|^{q'}_{\L^{p'}} \, \d \tau \,\frac{\d t}{t}\bigg)^\frac{1}{q'}\\
        &\simeq  \bigg(\int\limits_0^\infty  \|h(\tau,\cdot)\|^{q'}_{\L^{p'}}  \,\frac{\d \tau}{\tau}\bigg)^\frac{1}{q'}.
    \end{align*}
    Finally, by the definition of $h$ in~\eqref{eq:duality_def_h} we get
    \begin{align*}
        \|G\|_{\Z^{p',q',r'}_{-\beta}}
        &\lesssim \bigg(\int\limits_0^\infty  \bigg\|\bigg(\fint\limits_{\frac{\tau}{4}}^{2\tau} \fint\limits_{B(z,8\tau)} |\tau^\beta g(\tau, z, s, y)|^{r'} \,\d y\d s  \,\bigg)^\frac{1}{r'} \bigg\|^{q'}_{\L^{p'}}  \,\frac{\d \tau}{\tau}\bigg)^\frac{1}{q'}\\
        &\lesssim \bigg(\int\limits_0^\infty  \bigg\|\bigg(\int\limits_{0}^\infty \int\limits_{\IR^d} |\tau^\beta g(\tau, z, s, y)|^{r'} \,\frac{\d y\d s}{s^{d+1}}  \,\bigg)^\frac{1}{r'} \bigg\|^{q'}_{\L^{p'}}  \,\frac{\d \tau}{\tau}\bigg)^\frac{1}{q'},
    \end{align*}
    which is bounded by the definition of $g$.
\end{proof}

\subsection{Dyadic description and embeddings}
\label{subsec: dyad and emb}
In this section, we give a dyadic characterization of our $\Z^{p,q,r}_\beta$-spaces. This can be seen as an extension of the dyadic characterization for the $\Z^{p,r}_\beta$-spaces of Amenta; see \cites{Amenta2, AA} and Remark~\ref{rem: comparison of known spaces}.

For $k\in \IZ$ we define the set of all dyadic half-open cubes of generation $k$ as
\begin{align*}
    \square_k \coloneqq \{ 2^kx + [0,2^k)^d : x\in \IZ^d\},
\end{align*}
and denote by $\square = \cup_{k\in\IZ} \square_k$ the family of all dyadic cubes. If $Q\in\square_k$ we let $\ell(Q) = 2^k$. Moreover, we denote by
\begin{align*}
    \bar Q \coloneqq \Big[\frac{\ell(Q)}{2} , \ell(Q) \Big) \times Q
\end{align*}
the associated Whitney box of a cube $Q$. Finally, set
\begin{align*}
    G \coloneqq \{ \bar Q : Q\in \square \},
\end{align*}
which is the collection of all Whitney boxes associated to $\square$. It forms a partition of $\IR^{d+1}_+$. Lastly, we define the set
\begin{align*}
    G(\bar Q) \coloneqq \{ \bar R \in G : \bar R \cap W(t,x) \neq \emptyset \ \text{for some} \ (t,x)\in \bar Q \},
\end{align*}
which contains only finitely many elements due to \cite[Lemma 2.19]{AA}. More precisely,the number of elements, denoted by $|G(\bar Q)|$, is bounded by a constant depending only on the dimension $d$.

\begin{proposition}
\label{prop: dyadic char}
    Let $0<p,q,r\leq \infty$ and $\beta\in \IR$. Then, for $f \in \Z^{p,q,r}_\beta$ there holds
    \begin{align*}
        \|f\|_{\Z^{p,q,r}_\beta} \simeq \Big( \sum\limits_{k\in\IZ} 2^{-k\beta q} \Big(\sum\limits_{Q\in \square_k} \ell(Q)^d \|f\|_{\L^r(\bar Q , \frac{\d y \d s}{s^{d+1}})}^p \Big)^\frac{q}{p} \Big)^\frac{1}{q}
    \end{align*}
    with implicit constants depending only on $d,p,q,r,\beta$.
\end{proposition}

\begin{proof}
    The proof follows the ideas of \cite[Prop.\@ 2.20]{AA}. To simplify notation, we concentrate on the case $p,q,r < \infty$.

    \noindent\textbf{Step 1: We show \enquote{$\lesssim$}.}
    We decompose the $\L^q$ and $\L^p$ integrals dyadically to get
    \begin{align*}
        \|f\|_{\Z^{p,q,r}_\beta}^q &\simeq \int\limits_0^\infty \bigg(\int\limits_{\IR^d} \bigg( \int\limits_{\frac{t}{2}}^t \int\limits_{B(x,t)} |s^{-\beta} f(s,y) |^r \, \frac{\d y \d s}{s^{d+1}} \bigg)^\frac{p}{r} \, \d x \bigg)^\frac{q}{p} \frac{\d t}{t}\\
        &= \sum\limits_{k\in \IZ} \, \int\limits_{2^{k-1}}^{2^{k}} \bigg( \sum\limits_{Q\in \square_k} \int\limits_{Q} \bigg( \int\limits_{\frac{t}{2}}^t \int\limits_{B(x,t)} |s^{-\beta} f(s,y) |^r \, \frac{\d y \d s}{s^{d+1}} \bigg)^\frac{p}{r} \, \d x \bigg)^\frac{q}{p} \frac{\d t}{t}.
    \intertext{Observe that for $Q\in \square_k$, $t\in[2^{k-1},2^k)$ and $x\in Q$ we have $(t,x)\in \bar Q$. Therefore, according to the definition of $G(\bar Q)$, we can cover $W(t,x)$ by all Whitney boxes of $G(\bar Q)$, to give}
        &\lesssim  \sum\limits_{k\in \IZ} 2^{-k\beta q}\int\limits_{2^{k-1}}^{2^{k}} \bigg( \sum\limits_{Q\in \square_k} \int\limits_{Q} \bigg(\sum\limits_{\bar R\in G(\bar Q)} \iint\limits_{\bar R} | f(s,y) |^r \, \frac{\d y \d s}{s^{d+1}} \bigg)^\frac{p}{r} \, \d x \bigg)^\frac{q}{p} \frac{\d t}{t}\\
        &\simeq \sum\limits_{k\in \IZ} 2^{-k\beta q}\bigg( \sum\limits_{Q\in \square_k} \ell(Q)^d \bigg(\sum\limits_{\bar R\in G(\bar Q)} \iint\limits_{\bar R} | f(s,y) |^r \, \frac{\d y \d s}{s^{d+1}} \bigg)^\frac{p}{r}   \bigg)^\frac{q}{p}.
    \intertext{Since $|G(\bar Q)|$ is bounded by a dimensional constant, we can pull out the sum to get}
        &\lesssim \sum\limits_{k\in \IZ} 2^{-k\beta q}\bigg( \sum\limits_{Q\in \square_k}  \sum\limits_{\bar R\in G(\bar Q)} \ell(Q)^d \|f\|_{\L^r(\bar R , \frac{\d y \d s}{s^{d+1}})}^p \bigg)^\frac{q}{p}.
    \intertext{Finally, for $\bar R \in G(\bar Q)$ it holds $\ell(R) \simeq \ell(Q)$. Indeed, for $Q\in \square_k$ only dyadic cubes $R\in \square_{k-1}\cup \square_k$ can satisfy $\bar R \cap W(t,x) \neq \emptyset$ for some $(t,x)\in \bar Q$. Therefore,}
    &\simeq \sum\limits_{k\in \IZ} 2^{-k\beta q}\bigg( \sum\limits_{Q\in \square_k}  \sum\limits_{\bar R\in G(\bar Q)} \ell(R)^d \|f\|_{\L^r(\bar R , \frac{\d y \d s}{s^{d+1}})}^p \bigg)^\frac{q}{p}.
    \intertext{Moreover, for a fixed $R \in \square_{k-1}\cup \square_k $ the number of cubes $Q\in \square_k$ with $\bar R \in G(\bar Q)$ is bounded by a dimensional constant. Thus,}
        &\lesssim \sum\limits_{k\in \IZ} 2^{-k\beta q}\bigg( \sum\limits_{R\in \square_{k-1}\cup \square_k}  \ell(R)^d \|f\|_{\L^r(\bar R , \frac{\d y \d s}{s^{d+1}})}^p \bigg)^\frac{q}{p}\\
        &\lesssim \sum\limits_{k\in \IZ} 2^{-k\beta q}\bigg( \sum\limits_{R\in \square_k}  \ell(R)^d \|f\|_{\L^r(\bar R , \frac{\d y \d s}{s^{d+1}})}^p \bigg)^\frac{q}{p},
    \end{align*}
    which proves one direction of our claim.\\

    \noindent\textbf{Step 2: We show \enquote{$\gtrsim$}.} Observe that, if $t \in (2^{k-1}, 2^k)$, $R \in \Box_k$ and $x\in R$, then $R \subset B(x,2\sqrt{d}t)$, and hence $\bar{R} \subset [\tfrac{t}{2},2t) \times B(x,ct)$, where $c \coloneqq 2\sqrt{d}$.
    Now, estimate
    \begin{align*}
        &\sum\limits_{k\in \IZ}2^{-k\beta q}\bigg( \sum\limits_{R\in \square_k}  \ell(R)^d \|f\|_{\L^r(\bar R , \frac{\d y \d s}{s^{d+1}})}^p \bigg)^\frac{q}{p}\\
        &\simeq \sum\limits_{k\in \IZ} \, \int\limits_{2^{k-1}}^{2^k} 2^{-k\beta q}\bigg( \sum\limits_{R\in \square_k}  \int\limits_R \|f\|_{\L^r(\bar R , \frac{\d y \d s}{s^{d+1}})}^p \,\d x \bigg)^\frac{q}{p}\,\frac{\d t}{t}\\
        &\leq \sum\limits_{k\in \IZ} \, \int\limits_{2^{k-1}}^{2^k} 2^{-k\beta q}\bigg( \sum\limits_{R\in \square_k} \int\limits_R \bigg( \int\limits_{\frac{t}{2}}^{2t}\int\limits_{B(x,ct)} |f(s,y)|^r \, \frac{\d y \d s}{s^{d+1}} \bigg)^\frac{p}{r} \,\d x \bigg)^\frac{q}{p}\,\frac{\d t}{t}\\
        &\lesssim \int\limits_{0}^\infty \bigg( \int\limits_{\IR^d}\bigg( \int\limits_{\frac{t}{2}}^{2t}\int\limits_{B(x,ct)} |s^{-\beta}f(s,y)|^r \, \frac{\d y \d s}{s^{d+1}} \bigg)^\frac{p}{r} \,\d x \bigg)^\frac{q}{p}\,\frac{\d t}{t}.
    \end{align*}
    Eventually, by a change of Whitney parameters (Proposition~\ref{prop: independent Whitney cube}) we can
    conclude the reverse inequality.
\end{proof}

\begin{remark}
\label{rem: trivial dyadic decomp}
    Let $0<p,q,r< \infty$ and $\beta\in\IR$. Given a function $f\in \Z^{p,q,r}_\beta$, Proposition~\ref{prop: dyadic char} shows that the decomposition
    \begin{align*}
        f = \sum\limits_{k\in\IZ} \sum\limits_{Q\in\square_k} \mathbf{1}_{\bar Q}f
    \end{align*}
    converges in $\Z^{p,q,r}_\beta$.
\end{remark}

Using the dyadic characterization, we derive the following embeddings, which can be seen as Hardy--Littlewood--Sobolev-type embeddings, in analogy to the corresponding embeddings for homogeneous Besov spaces; see \cite[Theorem 2.7.1]{Triebel1}.

\begin{theorem}[Weighted $\Z$-space embeddings]
\label{thm: weigthed embedding}
    Let $0<p_0\leq p_1\leq \infty$, $0<q_0\leq q_1\leq \infty$,  $0<r_1\leq r_0\leq \infty$ and $\beta_0,\beta_1\in \IR$ such that $\beta_0-\beta_1 = d(\frac{1}{p_0}-\frac{1}{p_1})$. Then, we have
    \begin{align*}
        \Z^{p_0,q_0,r_0}_{\beta_0} \hookrightarrow \Z^{p_1,q_1,r_1}_{\beta_1}.
    \end{align*}
\end{theorem}

\begin{proof}
    We use the dyadic characterization from Proposition~\ref{prop: dyadic char}. Since $p_0\leq p_1$, we get
    \begin{align*}
        \|f\|_{ \Z^{p_1,q_1,r_1}_{\beta_1} } &\simeq \Big(\sum\limits_{k\in\IZ} 2^{-k\beta_1 q_1}\Big( \sum\limits_{Q\in \square_k} \ell(Q)^d \|f\|_{\L^{r_1}(\bar Q, \frac{\d y \d s}{s^{d+1}})}^{p_1}  \Big)^\frac{q_1}{p_1} \Big)^\frac{1}{q_1}\\
        &\leq \Big(\sum\limits_{k\in\IZ} 2^{-k\beta_1 q_1}\Big( \sum\limits_{Q\in \square_k} \ell(Q)^{d\frac{p_0}{p_1}}\|f\|_{\L^{r_1}(\bar Q, \frac{\d y \d s}{s^{d+1}})}^{p_0}  \Big)^\frac{q_1}{p_0} \Big)^\frac{1}{q_1}\\
        &=\Big(\sum\limits_{k\in\IZ} 2^{-k(\beta_1 +d(\frac{1}{p_0}-\frac{1}{p_1}))q_1}\Big( \sum\limits_{Q\in \square_k} \ell(Q)^{d}\|f\|_{\L^{r_1}(\bar Q, \frac{\d y \d s}{s^{d+1}})}^{p_0}  \Big)^\frac{q_1}{p_0} \Big)^\frac{1}{q_1},
    \intertext{where we used $\ell(Q) = 2^k$ in the last step. Since $q_0\leq q_1$, $r_1\leq r_0$ and $\beta_0= \beta_1 +  d(\frac{1}{p_0}-\frac{1}{p_1})$, we get}
        &=\Big(\sum\limits_{k\in\IZ} 2^{-k\beta_0q_1}\Big( \sum\limits_{Q\in \square_k} \ell(Q)^{d}\|f\|_{\L^{r_1}(\bar Q, \frac{\d y \d s}{s^{d+1}})}^{p_0}  \Big)^\frac{q_1}{p_0} \Big)^\frac{1}{q_1}\\
        &\leq \Big(\sum\limits_{k\in\IZ} 2^{-k\beta_0q_0}\Big( \sum\limits_{Q\in \square_k} \ell(Q)^{d}\|f\|_{\L^{r_1}(\bar Q, \frac{\d y \d s}{s^{d+1}})}^{p_0}  \Big)^\frac{q_0}{p_0} \Big)^\frac{1}{q_0}\\
        &\lesssim \Big(\sum\limits_{k\in\IZ} 2^{-k\beta_0q_0}\Big( \sum\limits_{Q\in \square_k} \ell(Q)^{d}\|f\|_{\L^{r_0}(\bar Q, \frac{\d y \d s}{s^{d+1}})}^{p_0}  \Big)^\frac{q_0}{p_0} \Big)^\frac{1}{q_0}\\
        &\simeq \|f\|_{ \Z^{p_0,q_0,{r_0}}_{\beta_0}}.
    \end{align*}
    This proves the theorem.
\end{proof}

\begin{remark}
    We would like to point out that there are also mixed-type embeddings for tent and Z-spaces; see for example \cite[Thm.\@ 2.34]{AA}. Corresponding embeddings for $\T^{p,q,r}_\beta$ and $\Z^{p,q,r}_\beta$ are established in \cite[Thm.~3.23]{Haardt}.
\end{remark}

\subsection{Duality: full range}
\label{subsec: duality p<1}
We use the dyadic description to prove a duality theorem in the quasi-Banach range. Recall the convention $p'=\infty$ for $p\in(0,1]$.

\begin{theorem}[Duality -- quasi-Banach range]
\label{thm: duality full range}
    Let $1\leq r <\infty$, $0<p,q<\infty$ and $\beta\in \IR$. Furthermore, define $\beta'(p)\coloneqq -\beta+\max\{0,d(\frac{1}{p}-1)\}$.
    Then we have
    \begin{align*}
        \int\limits_0^\infty \int\limits_{\IR^d} |f(s,y)| \cdot |g(s,y)| \,\frac{\d y \d s}{s} \leq  \|f\|_{\Z^{p',q',r'}_{\beta'(p)} }\|g\|_{\Z^{\vphantom{p',q',r'}p,q,r}_{\beta} }
    \end{align*}
    for any measurable functions $f,g$. Moreover, we can identify $(\Z^{p,q,r}_\beta)' \simeq \Z^{p',q',r'}_{\beta'(p)}$ with equivalent norms via the $\L^2$ duality pairing.

\end{theorem}

\begin{proof}
    The case $1\leq p<\infty$ and $1\leq q<\infty$ is the Banach range case and was done in Theorem~\ref{thm: duality banach range}. Hence, we focus on the other cases by dividing the proof into the following cases:\\

    \noindent \textbf{Case 1: $1\leq p<\infty$ and $0< q<1$}. Notice that $\beta'(p)=-\beta$ in this case. For $f$ and $g$ measurable,
    we get with a similar argumentation as in Theorem~\ref{thm: duality banach range} and Theorem~\ref{thm: weigthed embedding} that
    \begin{align*}
        \int\limits_0^\infty \int\limits_{\IR^d} |f(s,y)|\cdot |g(s,y)| \ \frac{\d y \d s }{s} &\lesssim  \|f\|_{\Z^{p', \infty,r'}_{-\beta}}\cdot\|g\|_{\Z^{\vphantom{p',q',r'}p,1,r}_\beta} \leq \|f\|_{\Z^{p', \infty,r'}_{-\beta}}\cdot \|g\|_{\Z^{\vphantom{p',q',r'}p,q,r}_\beta} .
    \end{align*}
    This shows the inclusion $\Z^{p', \infty,r'}_{-\beta}\subset (\Z^{p,q,r}_\beta)'$. To prove the reverse inclusion, we take $\Phi \in (\Z^{p,q,r}_\beta)'$. For $g\in \Z^{p,1,r}_\beta$ we define
    \begin{align*}
        \tilde \Phi (g) \coloneqq \sum\limits_{m\in\IZ} \Phi\Big(\sum\limits_{R\in\square_m} \mathbf{1}_{\bar R}g \Big).
    \end{align*}
    Notice that this is well defined and bounded because we have, using the dyadic characterization and support properties of the cubes,
    \begin{align*}
        \Big\|\sum\limits_{R\in\square_m} \mathbf{1}_{\bar R}g  \Big\|_{\Z^{p,q,r}_\beta} &\simeq \Big( \sum\limits_{k\in\IZ}2^{-k\beta q} \Big(\sum\limits_{Q\in\square_k} \ell(Q)^d \Big\|\sum\limits_{R\in\square_m} \mathbf{1}_{\bar R}g\Big\|_{\L^r\big(\bar Q, \frac{\d y \d s}{s^{d+1}}\big)}^p\Big)^\frac{q}{p}\Big)^\frac{1}{q}\\
        &\leq\Big( \sum\limits_{k\in\IZ}2^{-k\beta q} \Big(\sum\limits_{Q\in\square_k} \ell(Q)^d \Big(\sum\limits_{R\in\square_m} \|\mathbf{1}_{\bar R}g\|_{\L^r\big(\bar Q, \frac{\d y \d s}{s^{d+1}}\big)}\Big)^p\Big)^\frac{q}{p}\Big)^\frac{1}{q}\\
        &=\Big(2^{-m\beta q} \Big(\sum\limits_{Q\in\square_m} \ell(Q)^d \|g\|_{\L^r\big(\bar Q, \frac{\d y \d s}{s^{d+1}}\big)}^p\Big)^\frac{q}{p}\Big)^\frac{1}{q}\\
        &=2^{-m\beta } \Big(\sum\limits_{Q\in\square_m} \ell(Q)^d \|g\|_{\L^r\big(\bar Q, \frac{\d y \d s}{s^{d+1}}\big)}^p\Big)^\frac{1}{p}\\
        &\leq \|g\|_{\Z^{p,1,r}_\beta}
    \end{align*}
    on the one hand, and on the other hand, re-using the previous calculation,
    \begin{align*}
        |\tilde \Phi (g)|&\leq  \sum\limits_{m\in\IZ} \Big|\Phi\Big(\sum\limits_{R\in\square_m} \mathbf{1}_{\bar R}g \Big)\Big|\\
        &\leq \|\Phi\| \times \sum\limits_{m\in\IZ}
        \Big\|\sum\limits_{R\in\square_m} \mathbf{1}_{\bar R}g  \Big\|_{\Z^{p,q,r}_\beta}\\
        &\leq  \|\Phi\| \times \sum\limits_{m\in\IZ} 2^{-m\beta } \Big(\sum\limits_{Q\in\square_m} \ell(Q)^d \|g\|_{\L^r\big(\bar Q, \frac{\d y \d s}{s^{d+1}}\big)}^p\Big)^\frac{1}{p}\\
        &= \|\Phi\| \times\|g\|_{\Z^{p,1,r}_\beta}.
    \end{align*}
    Thus, we conclude that $\tilde \Phi \in (\Z^{p,1,r}_\beta)'$ and that $\tilde \Phi = \Phi$ on $\Z^{p,q,r}_\beta$ due to Remark~\ref{rem: trivial dyadic decomp}. By Theorem~\ref{thm: duality banach range}, there exists $f \in \Z^{p',\infty,r'}_{-\beta}$ such that
    \begin{align*}
        \int\limits_{0}^\infty\int\limits_{\IR^d} f(s,y) \cdot \overline{g(s,y)} \,\frac{\d y \d s}{s} = \tilde \Phi(g) = \Phi(g),\quad \text{for all $g \in \Z^{p,q,r}_\beta$.}
    \end{align*}
    Hence, we can identify $\Phi$ with $f$, which gives $(\Z^{p,q,r}_\beta)' \subset  \Z^{p',\infty,r'}_{-\beta}$.\\

    \noindent \textbf{Case 2: $0<p<1$ and $0<q<\infty$}.
    Notice that $\beta'(p)=-\beta +d(\frac{1}{p}-1)$ in this case. By Theorem~\ref{thm: weigthed embedding}, we have the continuous embedding
    \begin{align*}
         \Z^{p,q,r}_\beta \subset \Z^{1,q,r}_{\beta-d(\frac{1}{p}-1)}.
    \end{align*}
    Thus, for $f\in \Z^{\infty,q',r'}_{-\beta+d(\frac{1}{p}-1)}$ and $g\in \Z^{p,q,r}_\beta$, we have by Theorem~\ref{thm: duality banach range} (if $q\geq 1$) or the previous case (if $q<1$) that
    \begin{align*}
        \int\limits_{0}^\infty\int\limits_{\IR^d} |f(s,y)|\cdot |g(s,y)|\,\frac{\d y \d s}{s} \leq \|f\|_{ \Z^{\infty,q',r'}_{-\beta+d(\frac{1}{p}-1)}}  \|g\|_{ \Z^{\vphantom{p',q',r'}1,q,r}_{\beta-d(\frac{1}{p}-1)}}\lesssim \|f\|_{ \Z^{\infty,q',r'}_{-\beta+d(\frac{1}{p}-1)}}  \|g\|_{ \Z^{\vphantom{p',q',r'}p,q,r}_\beta }.
    \end{align*}
    Hence, $ \Z^{\infty,q',r'}_{-\beta+d(\frac{1}{p}-1)}\subset ( \Z^{p,q,r}_\beta )'$. For the converse inclusion, we distinguish cases. First, assume $q\leq 1$ and pick $\Phi \in ( \Z^{p,q,r}_\beta )'$. For $g\in  \Z^{1,q,r}_{\beta-d(\frac{1}{p}-1)}$, define
    \begin{align*}
        \tilde \Phi (g) \coloneqq \sum\limits_{m\in\IZ} \sum\limits_{R\in\square_m} \Phi(\mathbf{1}_{\bar R}g).
    \end{align*}
    Then, one has
    \begin{align*}
        |\tilde \Phi (g)| &\leq \sum\limits_{m\in\IZ} \sum\limits_{R\in\square_m} |\Phi(\mathbf{1}_{\bar R}g)|\\
        &\leq \|\Phi\|\sum\limits_{m\in\IZ} \sum\limits_{R\in\square_m} \|\mathbf{1}_{\bar R} g\|_{\Z^{p,q,r}_\beta}\\
        &\simeq  \|\Phi\|\sum\limits_{m\in\IZ} \sum\limits_{R\in\square_m} 2^{-m\beta}\ell(R)^\frac{d}{p}\| g\|_{\L^r\big(\bar R, \frac{\d y \d s}{s^{d+1}}\big)}\\
        &\simeq  \|\Phi\|\sum\limits_{m\in\IZ} \sum\limits_{R\in\square_m} 2^{-m(\beta-d(\frac{1}{p}-1))}\ell(R)^d\| g\|_{\L^r\big(\bar R, \frac{\d y \d s}{s^{d+1}}\big)}\\
        &\leq \|\Phi\|\cdot \|g\|_{\Z^{1,q,r}_{\beta-d(\frac{1}{p}-1)}},
    \end{align*}
    where we used $\ell^q \subset \ell^1$ in the last step, which uses the assumption $q \leq 1$ of the current case.
    Thus, $\tilde \Phi \in (\Z^{1,q,r}_{\beta-d(\frac{1}{p}-1)})' \simeq \Z^{\infty,\infty,r'}_{-\beta+d(\frac{1}{p}-1)}$ due to the first case. Since $\tilde \Phi = \Phi$ on $\Z^{p,q,r}_\beta$, this means that we can identify $\Phi$ with a $ \Z^{\infty,\infty,r'}_{-\beta+d(\frac{1}{p}-1)}$-function, which proves $(\Z^{p,q,r}_\beta)'\subset  \Z^{\infty,\infty,r'}_{-\beta+d(\frac{1}{p}-1)}$.\\

    \noindent Second, let $1< q <\infty$. We have $\Z^{p,1,r}_\beta\subset \Z^{p,q,r}_\beta$, therefore
    \begin{align*}
        (\Z^{p,q,r}_\beta)'\subset(\Z^{p,1,r}_\beta)' \simeq \Z^{\infty,\infty,r'}_{-\beta+d(\frac{1}{p}-1)}.
    \end{align*}
    For some $\Phi \in (\Z^{p,q,r}_\beta)'$, let $f\in \Z^{\infty,\infty,r'}_{-\beta+d(\frac{1}{p}-1)}$ be such that
    \begin{align}
    \label{eq: phi identity}
        \Phi(g) = \iint\limits_{\IR^{d+1}_+} f(s,y) \cdot \overline{g(s,y)} \,\frac{\d y \d s}{s},\quad \text{for all $g \in \Z^{p,1,r}_\beta$},
    \end{align}
    and
    \begin{align}
    \label{eq: r>1}
        \sup\limits_{k\in\IZ}2^{(\beta-d(\frac{1}{p}-\frac{1}{r}))k} \sup\limits_{R\in\square_k} \|f\|_{\L^{r'}\big(\bar R, \frac{\d y \d s}{s}\big)} &\simeq \sup\limits_{k\in\IZ}2^{(\beta-d(\frac{1}{p}-1))k} \sup\limits_{R\in\square_k} \|f\|_{\L^{r'}\big(\bar R, \frac{\d y \d s}{s^{d+1}}\big)} \nonumber\\
        &\simeq \|f\|_{\Z^{\infty,\infty,r'}_{-\beta+d(\frac{1}{p}-1)}}\\
        &= \|\Phi\|\nonumber
    \end{align}
    if $r>1$, or
    \begin{align}
    \label{eq: r=1}
        \sup\limits_{k\in\IZ}2^{(\beta-d(\frac{1}{p}-1))k} \sup\limits_{R\in\square_k} \|f\|_{\L^{\infty}(\bar R)} = \|f\|_{\Z^{\infty,\infty,r'}_{-\beta+d(\frac{1}{p}-1)}} = \|\Phi\|
    \end{align}
    in the case $r=1$. For simplicity, we continue the argument in the case $r>1$. The other case can be treated analogously using \eqref{eq: r=1} instead of \eqref{eq: r>1}.
    Using the estimate \eqref{eq: r>1}, one can find for each $k\in\IZ$ a cube $R_k\in \square_k$ such that
    \begin{align}
    \label{eq: compare sup with pointwise}
         2^{(\beta-d(\frac{1}{p}-\frac{1}{r}))k} \sup\limits_{R\in\square_k} \|f\|_{\L^{r'}\big(\bar R, \frac{\d y \d s}{s}\big)}
        &\simeq 2^{(\beta-d(\frac{1}{p}-\frac{1}{r}))k} \|f\|_{\L^{r'}\big(\bar R_k, \frac{\d y \d s}{s}\big)}.
    \end{align}
    As a consequence of the Hahn-Banach theorem and the duality of $\L^r$-spaces, it is possible to find for each $k\in\IZ$ a function $\varphi_k\in \L^{r}\big(\bar R_k, \frac{\d y \d s}{s}\big)$ with the properties that
    \begin{align*}
        \|\varphi_k\|_{\L^{r'}\big(\bar R_k, \frac{\d y \d s}{s}\big)} \leq 1 \quad \text{and} \quad  \|f\|_{\L^{r'}\big(\bar R_k, \frac{\d y \d s}{s}\big)} \leq 2 \langle f , \varphi_k\rangle_{\L^{r'},\L^{r}}.
    \end{align*}
    Note that we implicitly assume $\langle f , \varphi_k\rangle_{\L^{r'},\L^{r}}$ to be real.
    Since all $\bar R_k$ are disjoint, we can define a global function $\varphi\in \L^{r}_{\loc}\big(\IR^{d+1}_+, \frac{\d y \d s}{s}\big)$ with $\varphi = \varphi_k$ on $\bar R_k$ for each $k\in \IZ$, and therefore
    \begin{align*}
        \|\varphi\|_{\L^{r'}\big(\bar R_k, \frac{\d y \d s}{s}\big)} \leq 1 \quad \text{and} \quad  \|f\|_{\L^{r'}\big(\bar R_k, \frac{\d y \d s}{s}\big)} \leq 2 \langle f , \varphi \ind_{\bar R_k} \rangle_{\L^{r'},\L^{r}} \quad \text{for every $k\in\IZ$.}
    \end{align*}
    Let $(a_k)_{k\in\IZ}$ be an arbitrary real, finite sequence, say supported on $\{ 0, \dots, N \}$ for some $N\in\IN_0$. By the properties of $\varphi$, we have
    \begin{align}
    \label{eq: back to}
        \Big|\sum\limits_{|k|\leq N} a_k2^{(\beta-d(\frac{1}{p}-\frac{1}{r}))k} \|f\|_{\L^{r'}\big(\bar R_k, \frac{\d y \d s}{s}\big)}\big| &\lesssim  \sum\limits_{|k|\leq N} |a_k| 2^{(\beta-d(\frac{1}{p}-\frac{1}{r}))k} \langle f , \varphi \ind_{\bar R_k} \rangle_{\L^{r'},\L^r}\nonumber \\
        &=\langle f , \sum\limits_{|k|\leq N} |a_k| 2^{(\beta-d(\frac{1}{p}-\frac{1}{r}))k}\varphi \ind_{\bar R_k} \rangle_{\L^{r'},\L^r}
    \end{align}
    Because the test function is a compactly supported $\L^{r}$-function, it belongs to $\Z^{p,1,r}_\beta$ by Lemma~\ref{lem: compare Z with Lr}. Thus, \eqref{eq: phi identity} yields
    \begin{align*}
        \langle f , \sum\limits_{|k|\leq N} |a_k|2^{(\beta-d(\frac{1}{p}-\frac{1}{r}))k}\varphi \ind_{\bar R_k} \rangle_{\L^{r'},\L^r} &= \Phi\Big(\sum\limits_{|k|\leq N} |a_k|2^{(\beta-d(\frac{1}{p}-\frac{1}{r}))k}\varphi \ind_{\bar R_k} \Big)\\
        &\leq \|\Phi\|\times \bigg\|\sum\limits_{|k|\leq N} |a_k|2^{(\beta-d(\frac{1}{p}-\frac{1}{r}))k}\varphi \ind_{\bar R_k} \bigg\|_{\Z^{p,q,r}_\beta}.
    \end{align*}
    Finally, keeping the location of $\bar R_k$ in mind, we compute
    \begin{align*}
       &\bigg\|\sum\limits_{|k|\leq N} |a_k|2^{(\beta-d(\frac{1}{p}-\frac{1}{r}))k}\varphi \ind_{\bar R_k} \bigg\|_{\Z^{p,q,r}_\beta} \\
       &\lesssim \Big(\sum\limits_{m\in\IZ} 2^{-m\beta q} \Big(\sum\limits_{Q\in\square_m} \ell(Q)^d \Big( \sum\limits_{|k|\leq N} |a_k|2^{(\beta-d(\frac{1}{p}-\frac{1}{r}))k} \|\varphi \ind_{\bar R_k} \|_{\L^r(\bar Q,\frac{\d y \d s}{s^{d+1}})} \Big)^p \Big)^\frac{q}{p} \Big)^\frac{1}{q}\\
       &\simeq \Big(\sum\limits_{|k|\leq N} 2^{-k\beta q} \ell(R_k)^\frac{dq}{p}  |a_k|^q2^{(\beta-\frac{d}{p})qk} \|\varphi_k\|^q_{\L^r(\bar R_k,\frac{\d y \d s}{s})}  \Big)^\frac{1}{q}\\
       &\leq \Big(\sum\limits_{|k|\leq N}  |a_k|^q \Big)^\frac{1}{q}.
    \end{align*}
    Going back to \eqref{eq: back to}, we hence have
    \begin{align*}
        \Big|\sum\limits_{|k|\leq N} a_k2^{(\beta-d(\frac{1}{p}-\frac{1}{r}))k}  \|f\|_{\L^{r'}\big(\bar R_k,\frac{\d y \d s}{s}\big)}\Big|\lesssim \|\Phi\|\times \Big(\sum\limits_{|k|\leq N}  |a_k|^q \Big)^\frac{1}{q},
    \end{align*}
    where the implicit constant is independent of $f$, $N$ and $(a_k)_{k\in\IZ}$. This shows that the sequence of real numbers
    \begin{align*}
        b_k = 2^{(\beta-d(\frac{1}{p}-\frac{1}{r}))k}  \|f\|_{\L^{r'}\big(\bar R_k,\frac{\d y \d s}{s}\big)}, \quad k\in\IZ,
    \end{align*}
    belongs to $\ell^{q'}(\IZ)$ with norm bound
    \begin{align*}
        \|b_k\|_{\ell^{q'}} = \sup\limits_{\substack{(a_k)\in \ell^{q}\ \text{finite}\\
        \|a_k\|_{\ell^q}\leq 1}}\Big|\sum\limits_{k\in\IZ} a_k b_k\Big|\lesssim \|\Phi\|.
    \end{align*}
    Finally, combining this with \eqref{eq: compare sup with pointwise}, we obtain
    \begin{align*}
        \|f\|_{\Z^{\infty,q',r'}_{-\beta+d(\frac{1}{p}-1)}} &\simeq\Big(\sum\limits_{k\in\IZ} 2^{(\beta-d(\frac{1}{p}-1))kq'}  \sup\limits_{R\in\square_k}\|f\|^{q'}_{\L^{r'}\big(\bar R,\frac{\d y \d s}{s^{d+1}}\big)}\Big)^\frac{1}{q'}\\
        &\lesssim\Big(\sum\limits_{k\in\IZ} 2^{(\beta-d(\frac{1}{p}-1))kq'}  \|f\|^{q'}_{\L^{r'}\big(\bar R_k,\frac{\d y \d s}{s^{d+1}}\big)}\Big)^\frac{1}{q'}\\
        &\lesssim \|\Phi\|,
    \end{align*}
    which shows $(\Z^{p,q,r}_\beta)'\subset  \Z^{\infty,q',r'}_{-\beta+d(\frac{1}{p}-1)}$.
\end{proof}

\section{Real interpolation of $\Z$-spaces}
\label{sec: real interpolation}
\noindent In this section, we give a complete picture of the real interpolation theory of $\Z^{p, q,r}_{\beta}$ spaces. Before we dive into our statements, we give a brief introduction to real interpolation theory, thereby also introducing relevant notation. For more background, we refer to the monograph \cite[Chp.\@ 3]{Bergh_Loefstroem} for a detailed excursion.

Let $(X_0,X_1)$ be an interpolation couple, that is, a pair of (quasi-)Banach spaces $X_0$ and $X_1$ that are continuous embedded into a Hausdorff topological vector space $Y$. Then we define the (quasi-)Banach space $X_0+X_1 = \{x_0+x_1 \in Y: x_0\in X_0, x_1\in X_1\}$  equipped with the (quasi-)norm
\begin{align*}
    \|x\|_{X_0+X_1} \coloneqq \inf\limits_{\substack{x_0\in X_0, x_1\in X_1\\ x=x_0+x_1}} \|x_0\|_{X_0} + \|x_1\|_{X_1}.
\end{align*}
For every $x\in X_0 + X_1$ and $t>0$ we define the $K$-functional
\begin{align*}
    K(x,t,X_0,X_1) \coloneqq \inf\limits_{\substack{x_0\in X_0, x_1\in X_1\\ x=x_0+x_1}} \|x_0\|_{X_0} + t\|x_1\|_{X_1}.
\end{align*}
Since we will often work with $q$-powers of this functional in the following, we introduce the more readable abbreviation
\begin{align*}
    K^q(x,t,X_0,X_1) \coloneqq \big(
    K(x,t,X_0,X_1)
    \big)^q.
\end{align*}
Finally, for $0<\theta<1$ and $0<q\leq \infty$, we define the real interpolation space $(X_0,X_1)_{\theta,q}$ as the set of all $x\in X_0+X_1$ such that
\begin{align*}
    \|x\|_{\theta,q} \coloneqq \bigg(\int\limits_0^\infty t^{-\theta q}  K^q(x,t,X_0,X_1) \, \frac{\d t}{t}\bigg)^\frac{1}{q}<\infty.
\end{align*}
It is known that $(X_0,X_1)_{\theta,q}$ is a (quasi-)Banach space. Furthermore, splitting the integral dyadically and using monotonicity in $t$ of the $K$-functional, one calculates
\begin{align}
\label{eq: dyadic int}
    \|x\|_{\theta,q} \simeq \Big(\sum\limits_{k\in\IZ} 2^{k\alpha\theta q} K^q(x,2^{-k\alpha},X_0,X_1)\Big)^\frac{1}{q}
\end{align}
for every $\alpha>0$, where the implicit constant depends only on $\alpha, \theta,q$.\\

Now that we have set the notation, we will begin with our first result. Its proof is similar to that of \cite[Thm.\@ 2.4.2]{Triebel1}.

\begin{theorem}[Real interpolation of $\Z$-spaces I]
\label{thm: real int of Z-spaces}
    Let $0< p,q, q_0,q_1, r\leq \infty$. Furthermore, let $\beta_0,\beta_1\in\IR$ with $\beta_0\neq \beta_1$ and $\theta\in (0,1)$. Then, we have
    \begin{align*}
        (\Z^{p,q_0,r}_{\beta_0} , \Z^{p,q_1,r}_{\beta_1} )_{\theta, q} = \Z^{p, q,r}_{\beta}
    \end{align*}
    where $\beta = (1-\theta)\beta_0 + \theta \beta_1$.
\end{theorem}

\noindent We stress that $q$ is not related to $q_0$ and $q_1$ but arbitrary.

\begin{proof}
    To simplify notation, we present the proof only for finite exponents. Moreover, Lebesgue norms over a Whitney box are always with respect to the measure $\frac{\d y \d s}{s^{d+1}}$ without explicitly mentioning this. The case of infinite exponents follows with the usual modification. Moreover, we assume without loss of generality that $\beta_0>\beta_1$. In the proof, we use without mentioning that the $\Z^{p,q,r}_\beta$ (quasi-)norm is equivalent to the dyadic description given in Proposition~\ref{prop: dyadic char}. We divide the proof into three steps. \\

    \noindent \textbf{Step 1 : We show $(\Z^{p,\infty,r}_{\beta_0} , \Z^{p,\infty,r}_{\beta_1} )_{\theta, q} \subset  \Z^{p, q,r}_{\beta}$}.\\
    Denote by $ \|\cdot\|^q_{\theta,q}$ the norm of the interpolation space $(\Z^{p,\infty,r}_{\beta_0} , \Z^{p,\infty,r}_{\beta_1} )_{\theta, q}$. Since $\beta_1\neq \beta_0$, we use \eqref{eq: dyadic int} with $\alpha = \beta_0-\beta_1$ to get
    \begin{align}
    \label{eq: real int char cont to dyad}
        \|f\|^q_{\theta,q}\simeq  \sum\limits_{k\in\IZ} 2^{k(\beta_0-\beta_1)\theta q} K^q(f,2^{-k(\beta_0-\beta_1)},\Z^{p,\infty,r}_{\beta_0} , \Z^{p,\infty,r}_{\beta_1}).
    \end{align}
    Now, let $f_0\in \Z^{p,\infty,r}_{\beta_0}$ and $f_1\in \Z^{p,\infty,r}_{\beta_1}$ be such that $f=f_0 + f_1$. Then, for each $k\in\IZ$, we have
    \begin{align*}
        &2^{- k\beta_0}\bigg(\sum\limits_{Q\in\square_k} \ell(Q)^d \|f\|_{\L^r(\bar Q)}^p\bigg)^\frac{1}{p}\\
        &\lesssim 2^{-k\beta_0}\bigg(\sum\limits_{Q\in\square_k} \ell(Q)^d \|f_0\|_{\L^r(\bar Q)}^p\bigg)^\frac{1}{p} + 2^{-k\beta_0}\bigg(\sum\limits_{Q\in\square_k} \ell(Q)^d \|f_1\|_{\L^r(\bar Q)}^p\bigg)^\frac{1}{p} \\
        &\leq \sup\limits_{m\in\IZ}2^{-m\beta_0}\bigg(\sum\limits_{Q\in\square_{m}} \ell(Q)^d \|f_0\|_{\L^r(\bar Q)}^p\bigg)^\frac{1}{p}\\
        &\qquad+ 2^{-k(\beta_0-\beta_1)} \sup\limits_{m\in\IZ}2^{-m\beta_1}\bigg(\sum\limits_{Q\in\square_{m}} \ell(Q)^d \|f_1\|_{\L^r(\bar Q)}^p\bigg)^\frac{1}{p} \\
        &\simeq \|f_0\|_{\Z^{p,\infty,r}_{\beta_0}} + 2^{-k(\beta_0-\beta_1)}\|f_1\|_{\Z^{p,\infty,r}_{\beta_1}}.
    \end{align*}
    Since this holds for all decompositions of $f$ in these spaces, we get that
    \begin{align*}
        \|f\|^q_{\theta,q}&\simeq\sum\limits_{k\in\IZ} 2^{k(\beta_0-\beta_1)\theta q} K^q(f,2^{-k(\beta_0-\beta_1)},\Z^{p,\infty,r}_{\beta_0} , \Z^{p,\infty,r}_{\beta_1})\\
        &\gtrsim \sum\limits_{k\in\IZ} 2^{k(\beta_0-\beta_1)\theta q} 2^{- k\beta_0q}\bigg(\sum\limits_{Q\in\square_k} \ell(Q)^d \|f\|_{\L^r(\bar Q)}^p\bigg)^\frac{q}{p}\\
        &=\sum\limits_{k\in\IZ}  2^{- k\beta q}\bigg(\sum\limits_{Q\in\square_k} \ell(Q)^d \|f\|_{\L^r(\bar Q)}^p\bigg)^\frac{q}{p}\\
        &\simeq \|f\|_{\Z^{p,q,r}_\beta}^q,
    \end{align*}
    where we used $\theta(\beta_0-\beta_1) = \beta_0-\beta$ in the penultimate step.\\

    \noindent \textbf{Step 2: $\Z^{p, q,r}_{\beta}\subset (\Z^{p,\tilde q,r}_{\beta_0} , \Z^{p,\tilde q,r}_{\beta_1} )_{\theta, q} $ for $0<\tilde q<q$}.\\
    Now, denote by $ \|\cdot \|_{\theta,q}$ the (quasi-)norm of the interpolation space $(\Z^{p,\tilde q,r}_{\beta_0} , \Z^{p,\tilde q,r}_{\beta_1} )_{\theta, q}$. Take $k\in \IZ$ and define
    \begin{align*}
        f_0(s,y) = \mathbf{1}_{[2^{k},\infty)}(s) f(s,y)  \quad \text{and} \quad f_1(s,y) =\mathbf{1}_{(0,2^{k})}(s) f(s,y).
    \end{align*}
    Then, $f = f_0 + f_1$. To bound $\|f_0\|_{\Z^{p,\tilde q,r}_{\beta_0}}$, expand
    \begin{align*}
        \|f_0\|^{\tilde q}_{\Z^{p,\tilde q,r}_{\beta_0}} \simeq \sum\limits_{n\in\IZ} 2^{-n\beta_0  \tilde q} \Big(\sum\limits_{Q\in\square_n} \ell(Q)^d \big\|\mathbf{1}_{[2^{k},\infty)} f \big\|_{\L^r(\bar Q)}^p \Big)^\frac{\tilde q}{p} .
    \end{align*}
    We only have to consider $n\geq k+1$ in the outer sum, since otherwise the support of the integrand in the $\L^r$-norm is disjoint to $\bar Q$. Moreover, define
    \begin{align}
        \label{eq: xin}
        \xi_n \coloneqq \Big(\sum\limits_{Q\in\square_n} \ell(Q)^d \|f \|_{\L^r(\bar Q)}^p \Big)^\frac{1}{p}
    \end{align}
    for all $n\in\IZ$. Then,
    \begin{align*}
        \|f_0\|^{\tilde q}_{\Z^{p,\tilde q,r}_{\beta_0}} \simeq \sum\limits_{n = k+1}^\infty 2^{-n\beta_0  \tilde q} \Big(\sum\limits_{Q\in\square_n} \ell(Q)^d \|f \|_{\L^r(\bar Q)}^p \Big)^\frac{\tilde q}{p} = \sum\limits_{n = k+1}^\infty 2^{-n\beta_0 \tilde q} \xi_n^{\tilde q}.
    \end{align*}
    Analogously, we get
    \begin{align*}
        \|f_1\|^{\tilde q}_{\Z^{p,\tilde q,r}_{\beta_1}} \simeq \sum\limits_{n = -\infty}^k 2^{-n\beta_1 \tilde q} \Big(\sum\limits_{Q\in\square_n} \ell(Q)^d \|f \|_{\L^r(\bar Q)}^p \Big)^\frac{\tilde q}{p} = \sum\limits_{n = -\infty}^k 2^{-n\beta_1 \tilde q} \xi_n^{\tilde q}.
    \end{align*}
    Taking the infimum over all decompositions of $f$ yields
    \begin{align*}
        K^q(f,2^{-k(\beta_0-\beta_1)},\Z^{p,\tilde q,r}_{\beta_0} , \Z^{p,\tilde q,r}_{\beta_1}) \lesssim \bigg(\Big[\sum\limits_{n = k+1}^\infty 2^{-n\beta_0  \tilde q} \xi_n^{\tilde q} \Big]^\frac{1}{\tilde q} + 2^{-k(\beta_0-\beta_1)}\Big[\sum\limits_{n = -\infty}^k 2^{-n\beta_1\tilde q} \xi_n^{\tilde q} \Big]^\frac{1}{\tilde q}\bigg)^q
    \end{align*}
    for every $k\in\IZ$. Thus, we have by \eqref{eq: dyadic int} that
    \begin{align}
    \label{eq: RI 1}
         \|f\|^q_{\theta,q} \lesssim  \sum\limits_{k\in\IZ} 2^{k(\beta_0-\beta_1)\theta q} \bigg(\Big[\sum\limits_{n = k+1}^\infty 2^{-n\beta_0 \tilde q} \xi_n^{\tilde q} \Big]^\frac{1}{\tilde q} + 2^{-k(\beta_0-\beta_1)}\Big[\sum\limits_{n = -\infty}^k 2^{-n\beta_1 \tilde q} \xi_n^{\tilde q} \Big]^\frac{1}{\tilde q}\bigg)^q.
    \end{align}
    Now, pick $\beta_0>\alpha_0 >\beta >\alpha_1>\beta_1$ and $\sigma>0$ such that $\frac{\tilde q}{\sigma}+\frac{\tilde q}{q}  = 1$. Then, with Hölder's inequality we get
    \begin{align*}
        \Big[\sum\limits_{n = k+1}^\infty 2^{-n\beta_0  \tilde q} \xi_n^{\tilde q} \Big]^\frac{1}{\tilde q} &= \Big[\sum\limits_{n = k+1}^\infty 2^{-n(\beta_0-\alpha_0) \tilde q} \cdot 2^{-n\alpha_0  \tilde q}\xi_n^{\tilde q} \Big]^\frac{1}{\tilde q}\\
        &\leq \Big[\sum\limits_{n = k+1}^\infty 2^{-n(\beta_0-\alpha_0) \sigma}\Big]^\frac{1}{\sigma}\cdot \Big[\sum\limits_{n = k+1}^\infty 2^{-n\alpha_0  q}\xi^{q}_n \Big]^\frac{1}{ q}\\
        &\lesssim 2^{-k(\beta_0-\alpha_0)}\Big[\sum\limits_{n = k+1}^\infty  2^{-n\alpha_0  q}\xi^{q}_n \Big]^\frac{1}{ q},
    \end{align*}
    as well as
    \begin{align*}
        \Big[\sum\limits_{n = -\infty}^k 2^{-n\beta_1 \tilde q} \xi_n^{\tilde q} \Big]^\frac{1}{\tilde q} &= \Big[\sum\limits_{n = -\infty}^k 2^{-n(\beta_1-\alpha_1) \tilde q} \cdot 2^{-n\alpha_1 \tilde q}\xi_n^{\tilde q} \Big]^\frac{1}{\tilde q}\\
        &\leq \Big[\sum\limits_{n = -\infty}^k 2^{-n(\beta_1-\alpha_1) \sigma}\Big]^\frac{1}{\sigma} \cdot \Big[\sum\limits_{n = -\infty}^k  2^{-n\alpha_1 q}\xi_n^{ q} \Big]^\frac{1}{q}\\
        &\lesssim 2^{-k(\beta_1-\alpha_1)}\Big[\sum\limits_{n = -\infty}^k  2^{-n\alpha_1  q}\xi_n^{ q} \Big]^\frac{1}{q}.
    \end{align*}
    Plug both bounds back into \eqref{eq: RI 1}. Using $\beta_0 - \beta = \theta(\beta_0-\beta_1)$, we get
    \begin{align*}
        \|f\|^q_{\theta,q} &\lesssim \sum\limits_{k\in\IZ} 2^{k(\beta_0-\beta_1)\theta q} \bigg(2^{-k(\beta_0-\alpha_0)}\Big[\sum\limits_{n = k+1}^\infty  2^{-n\alpha_0 q}\xi^{q}_n \Big]^\frac{1}{ q} + 2^{-k(\beta_0-\beta_1)}2^{-k(\beta_1-\alpha_1)}\Big[\sum\limits_{n = -\infty}^k  2^{-n\alpha_1   q}\xi_n^{ q} \Big]^\frac{1}{q}\bigg)^q\\
        &\lesssim \sum\limits_{k\in\IZ} 2^{k(\beta_0-\beta)q} \bigg(2^{-k(\beta_0-\alpha_0)q}\sum\limits_{n = k+1}^\infty  2^{-n\alpha_0  q}\xi^{q}_n + 2^{-k(\beta_0-\alpha_1)q}\sum\limits_{n = -\infty}^k  2^{-n\alpha_1   q}\xi_n^{ q}\bigg)\\
        &= \sum\limits_{k\in\IZ} 2^{-k(\beta-\alpha_0)q}\sum\limits_{n = k+1}^\infty  2^{-n\alpha_0 q}\xi^{q}_n + \sum\limits_{k\in\IZ} 2^{-k(\beta-\alpha_1)q}\sum\limits_{n = -\infty}^k  2^{-n\alpha_1 q}\xi_n^{ q}.
    \intertext{Finally, with Fubini's theorem we get}
        &= \sum\limits_{n\in\IZ} \Big( \sum\limits_{k=-\infty}^{n-1}2^{-k(\beta-\alpha_0)q}\Big) 2^{-n\alpha_0  q}\xi^{q}_n + \sum\limits_{n\in\IZ}\Big(\sum\limits_{k=n}^\infty 2^{-k(\beta-\alpha_1)q}\Big)  2^{-n\alpha_1  q}\xi_n^{ q}\\
        &\simeq  \sum\limits_{n\in\IZ} 2^{-n(\beta-\alpha_0)q}\cdot 2^{-n\alpha_0 q}\xi^{q}_n + \sum\limits_{n\in\IZ} 2^{-n(\beta-\alpha_1)q} \cdot  2^{-n\alpha_1  q}\xi_n^{ q}\\
        &=\sum\limits_{n\in\IZ} 2^{-n\beta q}\xi^{q}_n\\
        &\simeq \|f\|_{\Z^{p,q,r}_\beta}^q,
    \end{align*}
    where we used the definition of $\xi_n$ in \eqref{eq: xin}. This proves the second step.\\

    \noindent \textbf{Step 3: Conclusion.}\\
    Take $0<\tilde q < \min\{q,q_1,q_2\}$ and use twice the embedding of $\Z$-spaces (Theorem~\ref{thm: weigthed embedding}), as well as the previous steps, to conclude
    \begin{align*}
         \Z^{p, q,r}_{\beta} &\subset (\Z^{p,\tilde q,r}_{\beta_0} , \Z^{p,\tilde q,r}_{\beta_1} )_{\theta, q} \\
         &\subset (\Z^{p, q_0,r}_{\beta_0} , \Z^{p, q_1,r}_{\beta_1} )_{\theta, q} \\
         &\subset (\Z^{p,\infty,r}_{\beta_0} , \Z^{p,\infty,r}_{\beta_1} )_{\theta, q} \subset  \Z^{p, q,r}_{\beta},
    \end{align*}
    which proves the claim.
\end{proof}

This interpolation result is for fixed parameter $p$. If we want to interpolate the parameter $p$, we have the following result.

\begin{theorem}[Real interpolation of $\Z$-spaces II]
    \label{thm: real interpolation Z spaces II}
    Let $0< p_0,q_0,p_1,q_1,r\leq \infty$. Furthermore, let $\beta_0,\beta_1\in\IR$ and $\theta\in (0,1)$. Then, we have
    \begin{align*}
        (\Z^{p_0,q_0,r}_{\beta_0} , \Z^{p_1,q_1,r}_{\beta_1} )_{\theta, p} = \Z^{p, p,r}_{\beta},
    \end{align*}
    where $\beta = (1-\theta)\beta_0 + \theta \beta_1$ and $\frac{1}{p} = \frac{1-\theta}{p_0} + \frac{\theta}{p_1} = \frac{1-\theta}{q_0} + \frac{\theta}{q_1}$.
\end{theorem}

\begin{proof}
    We use the dyadic description of $\Z^{p,q,r}_{\beta}$ (Proposition~\ref{prop: dyadic char}) to reduce matters to real interpolation of vector-valued sequence spaces. We divide the proof into two steps.\\

    \noindent \textbf{Step 1: Embed $\Z^{p,q,r}_\beta$ into vector-valued sequence space.}
    Fix the Whitney box $\bar Q_0 = (1/2 , 1) \times Q_0$ associated to the unit cube $Q_0 = [0,1)^d$. Let $\bar Q \in G$ with $Q =2^kx+[0,2^k)^d$ for some $x\in \IZ^d$ and $k\in \IZ$. Then, for a measurable function $f$ on $\IR^{d+1}_+$, we define the affine reparamerization
    \begin{align*}
        f_{k,x}(s,y) \coloneqq f(2^ks , 2^ky +2^kx).
    \end{align*}
    By Proposition~\ref{prop: dyadic char}, we have
    \begin{align}
    \label{eq: isoC}
        \|f\|_{\Z^{p,q,r}_\beta}
        &\simeq \Big(\sum\limits_{k\in\IZ} 2^{-k\beta q}\Big(\sum\limits_{x\in\IZ^d} \|2^{\frac{kd}{p}} f_{k,x}\|_{\L^r(\bar Q_0; \frac{\d y \d s}{s^{d+1}})}^p\Big)^\frac{q}{p}\Big)^\frac{1}{q}
        =\big\| 2^\frac{kd}{p}f_{k,x} \big\|_{\ell^q_\beta(\IZ; \ell^p(\IZ^d; \L^r(\bar Q_0; \frac{\d y \d s}{s^{d+1}})))}.
    \end{align}
    In the following, we use the shorter notation
    \begin{align*}
        \ell^q_\beta\ell^p\L^r \coloneqq \ell^q_\beta(\IZ; \ell^p(\IZ^d; \L^r(\bar Q_0; \frac{\d y \d s}{s^{d+1}}))).
    \end{align*}
    Now, we define (formally) the map
    \begin{align}
    \label{eq: def I}
        I: f(s,y) \mapsto 2^{\frac{kd}{p}}f_{k,x}(s,y).
    \end{align}
    By~\eqref{eq: isoC}, we find that $I$ restricts to a well-defined bounded map $\Z^{p,q,r}_\beta \to\ell^q_\beta\ell^p\L^r$ satisfying $\|I(f)\|_{\ell^q_\beta\ell^p\L^r}\simeq \|f\|_{\Z^{p,q,r}_\beta}$.
    Next, we construct an inverse map to $I$. To this end, we take some arbitrary but fixed sequence
    \begin{align*}
        g =((g_{k,x})_{x\in\IZ^d})_{k\in\IZ}
    \end{align*}
    of measurable functions on $\bar Q_0$, and define the function
    \begin{align*}
        f(s,y) \coloneqq \sum\limits_{k\in\IZ} 2^{-\frac{dk}{p}}\sum\limits_{x\in\IZ^d} \mathbf{1}_{\bar Q_0}(2^{-k}s,2^{-k}y-x) \, g_{k,x}(2^{-k}s,2^{-k}y- x).
    \end{align*}
    Note that owing to the indicator function, convergence of the sums is trivial.
    As a limit of measurable functions, $f$ is again measurable. Moreover, for $m\in\IZ$ and $z\in\IZ^d$, one has for $(s,y)\in \bar Q_0$ that
    \begin{align*}
        f_{m,z}(s,y) &=f(2^ms , 2^m(y+z))\\
        &=\sum\limits_{k\in\IZ} 2^{-\frac{dk}{p}}\sum\limits_{x\in\IZ^d} \mathbf{1}_{\bar Q_0}(2^{m-k}s,2^{m-k}(y+z)-x) \, g_{k,x}(2^{m-k}s , 2^{m-k}(y+z)-x) \\
        &=2^{-\frac{dm}{p}} g_{m,z}(s,y),
    \end{align*}
    by an elementary geometric consideration.
    Hence, by setting $$I^{-1}\bigl[((g_{k,x})_{x\in \IZ^d})_{k\in \IZ} \bigr] \coloneqq f,$$ we have constructed an operator that is well-defined on all sequence spaces of measurable functions and satisfies $I^{-1} \colon \ell^q_\beta\ell^p\L^r \to \Z^{p,q,r}_\beta$ as well as $I \circ I^{-1} = \Id = I^{-1} \circ I$. \\

    \noindent \textbf{Step 2: Conclusion.}
    By Step~1, we know that $I$ and $I^{-1}$ are morphisms between the interpolation couples $(\Z^{p_0,q_0,r}_{\beta_0}, \Z^{p_1,q_1,r}_{\beta_1})$ and $(\ell^{q_0}_{\beta_0}\ell^{p_0}\L^r, \ell^{q_1}_{\beta_1}\ell^{p_1}\L^r)$ which restrict to bounded maps $\Z^{p_i,q_i,r}_{\beta_i} \to \ell^{q_i}_{\beta_i}\ell^{p_i}\L^r$ and $\ell^{q_i}_{\beta_i}\ell^{p_i}\L^r \to \Z^{p_i,q_i,r}_{\beta_i}$ for $i \in \{0,1\}$, respectively. Hence, $(\theta,p)$-real interpolation yields
    \begin{align}
        I &\colon (\Z^{p_0,q_0,r}_{\beta_0} , \Z^{p_1,q_1,r}_{\beta_1} )_{\theta, p} \to (\ell^{q_0}_{\beta_0}\ell^{p_0}\L^r , \ell^{q_1}_{\beta_1}\ell^{p_1}\L^r)_{\theta, p}, \\
        I^{-1} &\colon (\ell^{q_0}_{\beta_0}\ell^{p_0}\L^r , \ell^{q_1}_{\beta_1}\ell^{p_1}\L^r)_{\theta, p} \to (\Z^{p_0,q_0,r}_{\beta_0} , \Z^{p_1,q_1,r}_{\beta_1} )_{\theta, p}.
    \end{align}
    Since $I \circ I^{-1} = \Id = I^{-1} \circ I$, this shows the ismorphism property $$(\Z^{p_0,q_0,r}_{\beta_0} , \Z^{p_1,q_1,r}_{\beta_1} )_{\theta, p} \simeq (\ell^{q_0}_{\beta_0}\ell^{p_0}\L^r , \ell^{q_1}_{\beta_1}\ell^{p_1}\L^r)_{\theta, p}.$$
    We claim that we can compute the space on the right-hand side of this identity as
    \begin{align}
    \label{eq: int ident l}
        (\ell^{q_0}_{\beta_0}\ell^{p_0}\L^r , \ell^{q_1}_{\beta_1}\ell^{p_1}\L^r)_{\theta, p} = \ell^{p}_{\beta}\ell^{p}\L^r.
    \end{align}
    Admitting this, the proof is complete, as $I^{-1}$ again identifies $\ell^{p}_{\beta}\ell^{p}\L^r$ with $\Z^{p,p,r}$.

    Finally, to prove \eqref{eq: int ident l}, we apply Peetre's interpolation result \cite[Chp.\@ 5, Thm.\@ 4, p.\@ 98]{Peetre} twice, to get
    \begin{align*}
        (\ell^{q_0}_{\beta_0}\ell^{p_0}\L^r , \ell^{q_1}_{\beta_1}\ell^{p_1}\L^r)_{\theta, p}  = \ell^p_\beta\big((\ell^{p_0}\L^r , \ell^{p_1}\L^r)_{\theta, p}\big) =  \ell^p_\beta\big( \ell^{p}\big((\L^r,\L^r)_{\theta, p}\big)\big) = \ell^{p}_{\beta}\ell^{p}\L^r,
    \end{align*}
    where all parameters are as in the claim.

\end{proof}

Lastly, we want to show that our $\Z^{p,q,r}_\beta$-spaces can be derived from real interpolation of $\T^{p,q,r}_\beta$-spaces, in the same way as Besov spaces can be derived from Triebel--Lizorkin spaces; see \cite[Thm.\@ 2.4.2]{Triebel1}. Recall that we did not define the scale of spaces $\T^{p,q,r}_\beta$ for $p=\infty$. In particular, we will not provide a characterization of $\Z^{\infty,q,r}_\beta$ by means of real interpolation of tent spaces here. This is done in \cite[Prop.~3.31]{Haardt}.\\

The characterization of $\Z^{p,q,r}_\beta$-spaces by means of tent spaces relies on the following nesting property.

\begin{lemma}
    For $0<p<\infty$ and $0< q, r\leq \infty$ we have the continuous embeddings
    \begin{align*}
         \Z^{p,\min\{p,q\},r}_\beta\hookrightarrow \T^{p,q,r}_\beta \hookrightarrow \Z^{p,\max\{p,q\},r}_\beta.
    \end{align*}
\end{lemma}

\begin{proof}
    We only showcase the first inclusion, the second one follows by a similar calculation. For notational simplicity, we only treat the case $q,r$ finite. Set $\alpha = \min\{p,q\}$ and fix $x\in\IR^d$. Using the embedding $\ell^\alpha(\IZ) \subset \ell^q(\IZ)$ and removing/introducing averaging integrals, we get
    \begin{align*}
        \bigg\| \bigg(\fint\limits_{\frac{t}{2}}^t \fint\limits_{B(x,t)} |s^{-\beta}f(s,y)|^r \,\d y \d s\bigg)^\frac{1}{r}\bigg\|_{\L^q_t}
        &=\bigg(\sum\limits_{k\in\IZ} \int\limits_{2^k}^{2^{k+1}} \bigg(\fint\limits_{\frac{t}{2}}^t \fint\limits_{B(x,t)} |s^{-\beta}f(s,y)|^r \,\d y \d s\bigg)^\frac{q}{r}\,\frac{\d t}{t}\bigg)^\frac{1}{q}\\
        &\lesssim \bigg(\sum\limits_{k\in\IZ}  \bigg(\fint\limits_{2^{k-1}}^{2^{k+1}} \fint\limits_{B(x,2^{k+1})} |s^{-\beta}f(s,y)|^r \,\d y \d s\bigg)^\frac{q}{r}\bigg)^\frac{1}{q}\\
        &\leq \bigg(\sum\limits_{k\in\IZ}  \bigg(\fint\limits_{2^{k-1}}^{2^{k+1}} \fint\limits_{B(x,2^{k+1})} |s^{-\beta}f(s,y)|^r \,\d y \d s\bigg)^\frac{\alpha}{r}\bigg)^\frac{1}{\alpha}\\
        &\simeq \bigg(\sum\limits_{k\in\IZ}  \int\limits_{2^k}^{2^{k+1}} \bigg(\fint\limits_{2^{k-1}}^{2^{k+1}} \fint\limits_{B(x,2^{k+1})} |s^{-\beta}f(s,y)|^r \,\d y \d s\bigg)^\frac{\alpha}{r} \,\frac{\d t}{t}\bigg)^\frac{1}{\alpha}\\
        &\lesssim \bigg(\sum\limits_{k\in\IZ}  \int\limits_{2^k}^{2^{k+1}} \bigg(\fint\limits_{\frac{t}{4}}^{2t} \fint\limits_{B(x,2t)} |s^{-\beta}f(s,y)|^r \,\d y \d s\bigg)^\frac{\alpha}{r} \,\frac{\d t}{t}\bigg)^\frac{1}{\alpha}\\
        &= \bigg(\int\limits_0^\infty \bigg(\fint\limits_{\frac{t}{4}}^{2t} \fint\limits_{B(x,2t)} |s^{-\beta}f(s,y)|^r \,\d y \d s\bigg)^\frac{\alpha}{r} \,\frac{\d t}{t}\bigg)^\frac{1}{\alpha}.
    \end{align*}
    Thus, applying the $\L^p$-norm in $x$ on both sides of the upper inequality and using Minkowski's integral (using $\alpha \leq p$) inequality yields
    \begin{align*}
        \|f\|_{\T^{p,q,r}_\beta} &\leq \bigg\| \bigg\| \bigg(\fint\limits_{\frac{t}{4}}^{2t} \fint\limits_{B(x,2t)} |s^{-\beta}f(s,y)|^r \,\d y \d s\bigg)^\frac{1}{r} \bigg\|_{\L^\alpha_t} \bigg\|_{\L^p_{\vphantom{t}x}}\\
        &\leq \bigg\| \bigg\| \bigg(\fint\limits_{\frac{t}{4}}^{2t} \fint\limits_{B(x,2t)} |s^{-\beta}f(s,y)|^r \,\d y \d s\bigg)^\frac{1}{r} \bigg\|_{\L^p_{\vphantom{t}x}} \bigg\|_{\L^\alpha_t}\\
        &\lesssim \|f\|_{\Z^{p,\alpha,r}_\beta},
    \end{align*}
    where we used a change of Whitney parameter in the last inequality, see Proposition~\ref{prop: independent Whitney cube}.
\end{proof}

\begin{proposition}[Real interpolation of tent spaces]
\label{prop: real int tent space}
    Let $0<p<\infty$ and $0< q_0,q_1,q, r\leq \infty$. Furthermore, let $\beta_0,\beta_1\in\IR$ with $\beta_0\neq \beta_1$ and $\theta\in (0,1)$. Then, we have
    \begin{align*}
        (\T^{p,q_0,r}_{\beta_0} , \T^{p,q_1,r}_{\beta_1} )_{\theta, q} = \Z^{p,q,r}_\beta,
    \end{align*}
    where $\beta = (1-\theta)\beta_0 + \theta \beta_1$.
\end{proposition}

\begin{proof}
    By the previous lemma, we have the embedding
    \begin{align*}
        (\Z^{p,\min\{p,q_0\},r}_{\beta_0} , \Z^{p,\min\{p,q_1\},r}_{\beta_1} )_{\theta, q}\subset (\T^{p,q_0,r}_{\beta_0} , \T^{p,q_1,r}_{\beta_1} )_{\theta, q} \subset  (\Z^{p,\max\{p,q_0\},r}_{\beta_0} , \Z^{p,\max\{p,q_1\},r}_{\beta_1} )_{\theta, q}.
    \end{align*}
    Hence, calculating the interpolation spaces on the left- and right-hand inclusions by Theorem~\ref{thm: real int of Z-spaces}, it follows
    \begin{align*}
        \Z^{p,q,r}_\beta \subset(\T^{p,q_0,r}_{\beta_0} , \T^{p,q_1,r}_{\beta_1} )_{\theta, q} \subset \Z^{p,q,r}_\beta,
    \end{align*}
    which proves the claim.
\end{proof}
\section{Complex interpolation of $\Z$-spaces}
\label{sec: complex int}

\noindent The goal of this section is to show that the $\Z^{p,q,r}_{\beta}$-spaces form a complex interpolation scale. To do so, we use a refinement of the complex interpolation method for quasi-Banach spaces of Kalton--Mitrea \cite{Kalton_Mitrea} due to Egert and Kosmala \cite{Egert_Kosmala}. If both of the underlying spaces are Banach spaces, then this method coincides with Calderón's complex method \cite{Calderon} (see also \cite[Chp.\@ 4]{Bergh_Loefstroem}). Because a precise construction of complex interpolation spaces is not necessary to understand this section, we use them as a blackbox and refer to the aforementioned references for details instead. Our main result reads as follows.

\begin{theorem}[Complex interpolation]
\label{thm: complex int full range}
    Let $0< p_0,p_1,q_0,q_1, r_0,r_1\leq \infty$ be such that $\max\{q_0 , p_0, r_0\}<\infty$ or $\max\{q_1,p_1,r_1\}<\infty$. Furthermore, let $\beta_0,\beta_1\in\IR$ and $\theta\in (0,1)$. Define
    \begin{align*}
        \frac{1}{p_\theta} = \frac{1-\theta}{p_0} + \frac{\theta}{p_1},\quad \frac{1}{q_\theta} = \frac{1-\theta}{q_0} + \frac{\theta}{q_1},\quad \frac{1}{r_\theta} = \frac{1-\theta}{r_0} + \frac{\theta}{r_1}, \quad \beta_\theta = (1-\theta)\beta_0 + \theta \beta_1.
    \end{align*}
    Then, we have
    \begin{align*}
        [\Z^{p_0,q_0,r_0}_{\beta_0} , \Z^{p_1,q_1,r_1}_{\beta_1} ]_{\theta} = \Z^{p_\theta, q_\theta,r_\theta}_{\beta_\theta}
    \end{align*}
    with equivalent norms.
\end{theorem}

Before proving this result in the full quasi-Banach range, we start with the corresponding result in the Banach range, which is based on the vector-valued embedding (Lemma~\ref{lem: vv-embedding}) and duality (Theorem~\ref{thm: duality banach range}).

\begin{lemma}
\label{lem: complex int 1<}
    Let $1< p_0,p_1,q_0,q_1, r_0,r_1\leq \infty$ be such that $\max\{q_0, p_0, r_0\}<\infty$ or $\max\{q_1, p_1, r_1\}<\infty$. Furthermore, let $\beta_0,\beta_1\in\IR$ and $\theta\in (0,1)$. Define
    \begin{align*}
        \frac{1}{p_\theta} = \frac{1-\theta}{p_0} + \frac{\theta}{p_1},\quad \frac{1}{q_\theta} = \frac{1-\theta}{q_0} + \frac{\theta}{q_1},\quad \frac{1}{r_\theta} = \frac{1-\theta}{r_0} + \frac{\theta}{r_1}, \quad \beta_\theta = (1-\theta)\beta_0 + \theta \beta_1.
    \end{align*}
    Then we have
    \begin{align*}
        [\Z^{p_0,q_0,r_0}_{\beta_0} , \Z^{p_1,q_1,r_1}_{\beta_1} ]_{\theta} = \Z^{p_\theta, q_\theta,r_\theta}_{\beta_\theta}
    \end{align*}
    with equivalent norms.
\end{lemma}

\begin{proof}
    We show both inclusions separately.\\

    \noindent \textbf{Step 1: $ [\Z^{p_0,q_0,r_0}_{\beta_0} , \Z^{p_1,q_1,r_1}_{\beta_1} ]_{\theta} \subset  \Z^{p_\theta, q_\theta,r_\theta}_{\beta_\theta}$.}
    We use the embedding $i$ of Lemma~\ref{lem: vv-embedding} and \cite[Thm.\@ 14.3.1]{AIB} to get boundedness of
    \begin{align*}
        i: [\Z^{p_0,q_0,r_0}_{\beta_0} , \Z^{p_1,q_1,r_1}_{\beta_1} ]_{\theta}\to [\L^{q_0}_{t;\beta_0}\L^{p_0}_{\vphantom{t;\beta_0}x}\L^{r_0}_{\vphantom{t;\beta_0}s,y} , \L^{q_1}_{t;\beta_1}\L^{p_1}_{\vphantom{t;\beta_0}x}\L^{r_1}_{\vphantom{t;\beta_0}s,y}]_{\theta} \simeq \L^{q_\theta}_{t;\beta_\theta}\L^{p_\theta}_{\vphantom{t;\beta_\theta}x}\L^{r_\theta}_{\vphantom{t;\beta_\theta}s,y}.
    \end{align*}
    Thus, we have
    \begin{align*}
        \|f\|_{\Z^{p_\theta, q_\theta,r_\theta}_{\beta_\theta}} &\simeq \|i(f)\|_{\L^{q_\theta}_{t;\beta_\theta}\L^{p_\theta}_{\vphantom{t;\beta_\theta}x}\L^{r_\theta}_{\vphantom{t;\beta_\theta}s,y}}\lesssim \|f\|_{[\Z^{p_0,q_0,r_0}_{\beta_0} , \Z^{p_1,q_1,r_1}_{\beta_1} ]_{\theta}}.
    \end{align*}

    \noindent \textbf{Step 2: $ \Z^{p_\theta, q_\theta,r_\theta}_{\beta_\theta}\subset [\Z^{p_0,q_0,r_0}_{\beta_0} , \Z^{p_1,q_1,r_1}_{\beta_1} ]_{\theta} $.}
    This inclusion can be obtained from Step~1 by virtue of a duality argument. First, using Theorem~\ref{thm: duality banach range} and dualizing the inclusion of Step~1, we deduce
    \begin{align*}
        \Z^{p_\theta, q_\theta,r_\theta}_{\beta_\theta} = (\Z^{p_\theta', q_\theta',r_\theta'}_{-\beta_\theta})' &\subset ([\Z^{p_0',q_0',r_0'}_{-\beta_0} , \Z^{p_1',q_1',r_1'}_{-\beta_1} ]_{\theta})',
    \end{align*}
    where we used that $p_\theta, q_\theta,r_\theta$ are finite by construction. Second, to calculate the interpolation space on the right-hand side of the last inclusion,
    we observe that at least one of the spaces $\Z^{p_0',q_0',r_0'}_{-\beta_0} $ or $\Z^{p_1',q_1',r_1'}_{-\beta_1}$ is reflexive by Theorem~\ref{thm: duality banach range} (keep in mind that $p_i, q_i, r_i > 1$ in both cases). Moreover, their intersection is dense in both spaces due to Proposition~\ref{prop: complete and dense subspace} and $p_i, q_i, r_i > 1$. Thus, we can apply \cite[Cor.\@ 4.5.2]{Bergh_Loefstroem} to get
    \begin{align*}
         \Z^{p_\theta, q_\theta,r_\theta}_{\beta_\theta} \subset([\Z^{p_0',q_0',r_0'}_{-\beta_0} , \Z^{p_1',q_1',r_1'}_{-\beta_1} ]_{\theta})'
        &= [(\Z^{p_0',q_0',r_0'}_{-\beta_0})' , (\Z^{p_1',q_1',r_1'}_{-\beta_1})']_{\theta}.
    \end{align*}
    Third, still using that $p_i',q_i',r_i'$ are finite because $p_i,q_i,r_i>1$, Theorem~\ref{thm: duality banach range} yields
    \begin{align*}
        \Z^{p_\theta, q_\theta,r_\theta}_{\beta_\theta} \subset  [(\Z^{p_0',q_0',r_0'}_{-\beta_0})' , (\Z^{p_1',q_1',r_1'}_{-\beta_1})' ]_{\theta}
        =[\Z^{p_0,q_0,r_0}_{\beta_0} , \Z^{p_1,q_1,r_1}_{\beta_1} ]_{\theta},
    \end{align*}
    which proves the claim.
\end{proof}

To get an interpolation result also in the quasi-Banach setting, we use on the one hand an abstract result to identify the interpolation spaces with Calderón products, and on the other hand convexity arguments to reduce matters to interpolation in the Banach setting. More precisely, we use the following abstract result by Egert and Kosmala \cite[Thm.\@ 1.1]{Egert_Kosmala}, which only needs one of the two interpolants to be separable. This way, endpoint spaces can be treated as well.\\

We state the result first and clarify the notions appearing in it right after.

\begin{proposition}
\label{prop: egert-kosmala}
    Let $X_0, X_1$ be $A$-convex quasi-Banach function spaces over a separable and $\sigma$-finite measure space $\Omega$, one of which is separable. Then, $X_0 + X_1$ is $A$-convex, and for every $\theta\in(0,1)$ one has
    \begin{align*}
        [X_0,X_1]_\theta = (X_0)^{1-\theta}(X_1)^\theta
    \end{align*}
    up to equivalent norms, where $(X_0)^{1-\theta}(X_1)^\theta$ is called the Calderón product of $X_0$ and $X_1$. It consists of all $f\in \L^0(\Omega)$ such that
    \begin{align*}
        \|f\|_{(X_0)^{1-\theta}(X_1)^\theta} \coloneqq \inf \Big\{\|g\|_{X_0}^{1-\theta} \|h\|_{X_1}^{\theta}: |f|\leq |g|^{1-\theta} |h|^\theta, \, g\in X_0 , h\in X_1 \Big\}<\infty.
    \end{align*}
\end{proposition}
Recall that a quasi-Banach space $X\subset \L^0(\Omega)$ is called a quasi-Banach function space, if
    \begin{enumerate}
        \item there exists $f\in X$ such that $f>0$ almost everywhere,

        \item $f\in X$ and $g\in \L^0(\Omega)$ with $|g|\leq |f|$ almost everywhere implies $g\in X$.
    \end{enumerate}
It is known that a quasi-Banach function space $X$ is $A$-convex if and only if it is \emph{lattice $\alpha$-convex} for some $\alpha>0$ (see \cite[Thm.\@ 7.8]{Kalton_Mayboroda_Mitrea}), where the later means that
\begin{align*}
    \Big\|\Big(\sum\limits_{i=1}^m |f_i|^\alpha \Big)^\frac{1}{\alpha} \Big\|_{X} \leq \Big(\sum\limits_{i=1}^m \|f_i\|^\alpha \Big)^\frac{1}{\alpha}
\end{align*}
for all finite families $(f_i)_{i=1,\dots, m} \subset X$. Moreover, this property implies that
\begin{align*}
    [X]^\alpha \coloneqq \{f\in \L^0(\Omega): |f|^\frac{1}{\alpha}\in X\}\quad \text{equipped with} \quad \|f\|_{[X]^\alpha} \coloneqq \||f|^\frac{1}{\alpha} \|_X^\alpha
\end{align*}
is a Banach function space as well, called the \emph{$\alpha$-convexification} of $X$; see \cite[Prop.\@ 3.23]{Kosmala}. Furthermore, by a straightforward computation, one can show that the Calderón product commutes with the process of convexification. More precisely, for two lattice $\alpha$-convex quasi-Banach function spaces $X_0,X_1$ we have for all $\theta\in (0,1)$ that
\begin{align}
\label{eq: commutation}
    \big[(X_0)^{1-\theta}(X_1)^\theta \big]^\alpha = \big([X_0]^\alpha\big)^{1-\theta} \big([X_1]^\alpha \big)^\theta.
\end{align}
Now, we check that our $\Z$-spaces fulfill all necessary properties so that the just explained machinery applies to them.

\begin{lemma}
\label{lem: Z-space convex}
    Let  $0< p,q, r \leq \infty$ and $\beta \in \IR$. Then, the space $\Z^{p,q,r}_\beta$ is a quasi-Banach function space, $\alpha$-convex for all $0<\alpha<\min\{p,q,r\}$ and $\Z^{p,q,r}_\beta = [\Z^{p/\alpha,q/\alpha,r/\alpha}_{\alpha\beta}]^\frac{1}{\alpha}$. Furthermore, the space $\Z^{p,q,r}_\beta$ is separable if $\max\{p,q,r\}<\infty$.
\end{lemma}

\begin{proof}
    The facts that $\Z^{p,q,r}_\beta$ is a quasi-Banach function space and that $\Z^{p,q,r}_\beta = [\Z^{p/\alpha,q/\alpha,r/\alpha}_{\alpha\beta}]^\frac{1}{\alpha}$ holds are clear. If $\max\{p,q,r\}<\infty$, then separability follows by Lemma~\ref{lem: vv-embedding} in conjunction with separability of weighted vector-valued Lebesgue spaces. Finally, for $0<\alpha<\min\{p,q,r\}$, we have
    \begin{align*}
        \Big\|\Big(\sum\limits_{i=1}^m |f_i|^\alpha \bigg)^\frac{1}{\alpha}\Big\|_{\Z^{p,q,r}_\beta} &= \bigg(\int\limits_0^\infty \bigg\|\bigg(\fint\limits_{\frac{t}{2}}^t \fint\limits_{B(\cdot,t)} \Big|s^{-\beta}\Big(\sum\limits_{i=1}^m |f_i|^\alpha\bigg)^\frac{1}{\alpha} \Big|^r\,\d y \d s \bigg)^\frac{1}{r} \bigg\|_{\L^p}^q\,\frac{\d t}{t}\bigg)^\frac{1}{q}\\
        &= \bigg(\int\limits_0^\infty \bigg\|\bigg(\fint\limits_{\frac{t}{2}}^t \fint\limits_{B(\cdot,t)} \Big(\sum\limits_{i=1}^m s^{-\alpha\beta }|f_i|^\alpha\bigg)^\frac{r}{\alpha}\,\d y \d s \bigg)^\frac{\alpha}{r} \bigg\|_{\L^{\frac{p}{\alpha}}}^\frac{q}{\alpha}\,\frac{\d t}{t}\bigg)^{\frac{\alpha}{ q} \frac{1}{\alpha}}\\
        &\leq  \Big(\sum\limits_{i=1}^m \| |f_i|^\alpha \|_{\Z^{p/\alpha, q/\alpha , r/\alpha}_{\alpha \beta}} \Big)^\frac{1}{\alpha}\\
        &=\Big(\sum\limits_{i=1}^m \| f_i\|_{\Z^{p,q,r}_{\beta}}^\alpha \Big)^\frac{1}{\alpha},
    \end{align*}
    which proves the claim.
\end{proof}

\begin{proof}[Proof of Theorem~\ref{thm: complex int full range}]
    Set $0<\alpha<\min\{p_0,p_1,q_0,q_1, r_0,r_1\}$. Then, owing to Lemma~\ref{lem: Z-space convex}, we get by Proposition~\ref{prop: egert-kosmala} that
    \begin{align*}
         [\Z^{p_0,q_0,r_0}_{\beta_0} , \Z^{p_1,q_1,r_1}_{\beta_1} ]_{\theta} = (\Z^{p_0,q_0,r_0}_{\beta_0} )^{1-\theta}   (\Z^{p_1,q_1,r_1}_{\beta_1} )^\theta.
    \end{align*}
    Using property~\eqref{eq: commutation}, we get again with Lemma~\ref{lem: Z-space convex} that
    \begin{align*}
        (\Z^{p_0,q_0,r_0}_{\beta_0} )^{1-\theta}  (\Z^{p_1,q_1,r_1}_{\beta_1} )^\theta &=  \Big( \big[\Z^{p_0/\alpha,q_0/\alpha,r_0/\alpha}_{\alpha\beta_0} \big]^{\frac{1}{\alpha}} \Big)^{1-\theta}   \Big(  \big[\Z^{p_1/\alpha,q_1/\alpha,r_1/\alpha}_{\alpha\beta_1} \big]^{\frac{1}{\alpha}} \Big)^{\theta} \\
        &=\Big[ \big(\Z^{p_0/\alpha,q_0/\alpha,r_0/\alpha}_{\alpha\beta_0}  \big)^{1-\theta}   \big(\Z^{p_1/\alpha,q_1/\alpha,r_1/\alpha}_{\alpha\beta_1} \big)^{\theta} \Big]^{\frac{1}{\alpha}}.
    \end{align*}
    We make the crucial observation $\tfrac{p_i}{\alpha}, \tfrac{q_i}{\alpha}, \tfrac{r_i}{\alpha}>1$. Thus, by Proposition~\ref{prop: egert-kosmala} followed by Lemma~\ref{lem: complex int 1<} and Lemma~\ref{lem: Z-space convex},
    we get
    \begin{align*}
        \Big[\big(\Z^{p_0/\alpha,q_0/\alpha,r_0/\alpha}_{\alpha\beta_0} \big)^{1-\theta}  \big(\Z^{p_1/\alpha,q_1/\alpha,r_1/\alpha}_{\alpha\beta_1}\big)^{\theta} \Big]^{\frac{1}{\alpha}}  &= \Big[\big[\Z^{p_0/\alpha,q_0/\alpha,r_0/\alpha}_{\alpha\beta_0} , \Z^{p_1/\alpha,q_1/\alpha,r_1/\alpha}_{\alpha\beta_1}\big]_\theta \Big]^{\frac{1}{\alpha}}\\
         &= \Big[ \Z^{p_\theta/\alpha,q_\theta/\alpha,r_\theta/\alpha}_{\alpha\beta_\theta}\Big]^{\frac{1}{\alpha}}\\
         &=  \Z^{p_\theta, q_\theta,r_\theta}_{\beta_\theta}.
    \end{align*}
    Concatenating the three established identities proves the claim.
\end{proof}

\begin{remark}
    We would like to point out that the theory of strong factorization provides an alternative approach to complex interpolation due to its connection to Calderón products. For example, see \cite{Huang} in the case of weighted tent spaces with Whitney averages $\T^{p,r}_{q,\beta}$. However, it is unclear whether such  factorizations hold for $\Z^{p,q,r}_\beta$-spaces. This is why we took a different approach to complex interpolation theory, following the method in \cite{Kalton_Mayboroda_Mitrea}.
\end{remark}

\begin{remark}
    Another important feature of tent spaces is the atomic theory. We have not developed this concept for $\Z$-spaces. However, a good substitute is the decomposition observed in Remark~\ref{rem: trivial dyadic decomp} and valid for all parameters $p,q,r$ and $\beta$.
\end{remark}


\begin{bibdiv}
\begin{biblist}

\bibitem{Amenta}
A.\@ Amenta.
\newblock Tent Spaces over Metric Measure Spaces under Doubling and Related Assumptions.
\newblock {\em Oper.\@ Theory: Adv.\@ Appl.\@}~\textbf{240} (2014), 1--29.

\bibitem{Amenta2}
A.\@ Amenta.
\newblock Interpolation and Embeddings of Weighted Tent Spaces.
\newblock {\em J.\@ Fourier Anal.\@ Appl.\@}~\textbf{24} (2018), 108--140.

\bibitem{AA}
A.\@ Amenta and P.\@ Auscher.
\newblock {\em Elliptic boundary value problems with fractional regularity data. The first order approach.}
\newblock Providence, RI: American Mathematical Society (AMS), 2018.

\bibitem{Auscher}
P.\@ Auscher.
\newblock Change of angle in tent spaces.
\newblock {\em  C.\@ R.\@, Math.\@, Acad.\@ Sci.\@ Paris}~\textbf{349} (2011), 297--301.

\bibitem{Auscher_Bechtel}
P.\@ Auscher and S.\@ Bechtel.
\newblock Non-linear parabolic {PDEs} with rough coefficients and critical data: existence, uniqueness and regularity of weak solutions.
\newblock Preprint, available under \url{https://arxiv.org/abs/2601.05080}, 2026.

\bibitem{Auscher_Egert}
P.\@ Auscher and M.\@ Egert.
\newblock {\em Boundary value problems and Hardy spaces for elliptic systems with block structure.}
\newblock Cham: Birkhäuser, 2023.

\bibitem{Barton_Mayboroda}
A.\@ Barton and S.\@ Mayboroda.
\newblock {\em Layer potentials and boundary-value problems for second order elliptic operators with data in Besov spaces.}
\newblock  Providence, RI: American Mathematical Society (AMS), 2016.

\bibitem{Bergh_Loefstroem}
J.\@ Bergh and J.\@ Löfström.
\newblock {\em Interpolation spaces. An introduction.}
\newblock Cham: Springer, 1976.

\bibitem{Bourdaud}
G.\@ Bourdaud.
\newblock Realizations of homogeneous Besov and Lizorkin--Triebel spaces.
\newblock {\em Math.\@ Nachr.\@}~\textbf{286} (2013), 476--491.

\bibitem{Bui_Taibleson}
H.\@ Bui and M.\@ Taibleson.
\newblock The characterization of the Triebel--Lizorkin spaces for $p=\infty$.
\newblock {\em J.\@ Fourier Anal.\@ Appl.\@}~\textbf{6} (2000), 537--550.

\bibitem{Calderon}
A.\@ Calderón.
\newblock Intermediate spaces and interpolation, the complex method.
\newblock {\em Stud.\@ Math.\@}~\textbf{24} (1964), 113--190.

\bibitem{Coifman_Meyer_Stein}
R.\@ Coifman, Y.\@ Meyer and E.\@ Stein.
\newblock Some new function spaces and their applications to Harmonic analysis.
\newblock {\em J.\@ Funct.\@ Anal.\@}~\textbf{62} (1985), 304--335.

\bibitem{Egert_Kosmala}
M.\@ Egert and B.\@ Kosmala.
\newblock On complex interpolation for pairs of quasi-Banach function spaces with a non-separable component.
\newblock {\em Stud.\@ Math.}~\textbf{288} (2026), 27--40.

\bibitem{Grafakos}
L.\@ Grafakos.
\newblock {\em Classical Fourier Analysis. 3rd ed.}
\newblock New York, NY: Springer, 2014.

\bibitem{Haardt}
L.~Haardt.
\newblock A coherent theory of tent spaces and homogeneous Triebel--Lizorkin spaces.
\newblock Preprint, available under \url{https://arxiv.org/abs/2602.20305}, 2026.

\bibitem{HTV}
E.\@ Harboure, J.\@ Torrea and B.\@ Viviani.
\newblock A vector-valued approach to tent spaces.
\newblock {\em J.\@ Anal.\@ Math.\@}~\textbf{56} (1991), 125--140.

\bibitem{Hofmann_Mayboroda_McIntosh}
S.\@ Hofmann, S.\@ Mayboroda and A.\@ McIntosh.
\newblock Second order elliptic operators with complex bounded measurable coefficients in $\L^p$, {Sobolev} and {Hardy} spaces.
\newblock {\em Ann.\@ Sci.\@ {\'E}c.\@ Norm.\@ Sup{\'e}r.\@}~\textbf{44} (2011), 723--800.

\bibitem{Huang}
Y.\@ Huang.
\newblock Weighted tent spaces with Whitney averages:
factorization, interpolation and duality.
\newblock {\em Math.\@ Z.\@}~\textbf{282} (2016), 913--933.

\bibitem{AIB}
T.\@ Hytönen, J.\@ van Neerven, M.\@ Veraar and L.\@ Weis.
\newblock {\em Analysis in Banach Spaces.\@ Volume III: Harmonic Analysis and Spectral Theory}.
\newblock Cham: Springer, 2023.

\bibitem{Kalton_Mayboroda_Mitrea}
N.\@ Kalton, S.\@ Mayboroda and M.\@ Mitrea.
\newblock Interpolation of Hardy--Sobolev--Besov--Triebel--Lizorkin Spaces and Applications to Problems in Partial Differential Equations.
\newblock {\em Contemp.\@ Math.\@}~\textbf{445} (2007),  121--177.

\bibitem{Kalton_Mitrea}
N.\@ Kalton and M.\@ Mitrea.
\newblock Stability results on interpolation scales of quasi-Banach spaces and applications.
\newblock {\em Trans.\@ Am.\@ Math.\@ Soc.\@}~\textbf{350} (1998), 3903--3922.

\bibitem{Kempka}
H.~Kempka, M.~Sch{\"a}fer and T.~Ullrich.
\newblock General coorbit space theory for quasi-{Banach} spaces and inhomogeneous function spaces with variable smoothness and integrability.
\newblock {\em J.\@ Fourier Anal.\@ Appl.}~\textbf{23} (2017), 1348--1407.

\bibitem{Kosmala}
B.\@ Kosmala.
\newblock Products of Quasi-Banach Lattices.
\newblock Master's thesis at Technische
Universität Darmstadt, available under \url{https://tuprints.ulb.tu-darmstadt.de/id/eprint/28710}, 2023.


\bibitem{Ullrich2}
Y.~Liang, Y.~ Sawano, T.~Ullrich, D.~Yang and W.~Yuan.
\newblock New characterizations of Besov--Triebel--Lizorkin--Hausdorff spaces including coorbits and wavelets.
\newblock {\em J.\@ Fourier Anal.\@ Appl.}~\textbf{18} (2012), 1067--1111.

\bibitem{Ullrich3}
Y.~Liang, Y.~ Sawano, T.~Ullrich, D.~Yang and W.~Yuan.
\newblock A new framework for generalized Besov-type and Triebel--Lizorkin-type spaces.
\newblock {\em Diss.\@ Math.}~\textbf{489} (2013).

\bibitem{Moussai}
M.\@ Moussai.
\newblock Characterizations of realized homogeneous Besov and
Triebel--Lizorkin spaces via differences.
\newblock {\em Appl.\@ Math.\@, Ser.\@ B (Engl.\@ Ed.\@)}~\textbf{33} (2018), 188--208.

\bibitem{Peetre}
J.\@ Peetre.
\newblock {\em New Thoughts on Besov spaces}.
\newblock Durham, NC: Duke University, Mathematics Department, 1976.


\bibitem{Rauhut}
H.~Rauhut and T.~Ullrich.
\newblock Generalized coorbit space theory and inhomogeneous function spaces of Besov--Lizorkin--Triebel type.
\newblock {\em J.\@ Funct.\@ Anal.}~\textbf{260} (2011), 3299--3362.

\bibitem{Sawano}
Y.\@ Sawano.
\newblock {\em Theory of Besov spaces}.
\newblock Singapore: Springer, 2018.

\bibitem{Triebelchar}
H.\@ Triebel.
\newblock Characterizations of Besov--Hardy--Sobolev spaces: A unified approach.
\newblock {\em J.\@ Approx.\@ Theory}~\textbf{52} (1988), 162--203.

\bibitem{Triebel1}
H.\@ Triebel.
\newblock {\em Theory of function spaces}.
\newblock Cham: Birkhäuser, 1983.

\bibitem{Triebel2}
H.\@ Triebel.
\newblock {\em Theory of function spaces II}.
\newblock Basel: Birkhäuser, 1992.

\bibitem{Ullrich1}
T.\@ Ullrich.
\newblock Continuous characterizations of Besov--Lizorkin--Triebel spaces and new interpretations as coorbits.
\newblock {\em J.\@ Funct.\@ Spaces Appl.\@}~\textbf{2012} (2012).

\end{biblist}
\end{bibdiv}
\end{document}